\documentclass[times]{amsart}
\usepackage{amsmath}
\usepackage{amsmath}
\usepackage{amssymb}
\usepackage{amscd}
\usepackage{mathrsfs}
\usepackage[all]{xy}
\usepackage[dvipdfmx]{graphicx}
\usepackage{setspace}
\def\mathcal{\mathscr}
\everymath{\displaystyle}
\newtheorem{thm}{Theorem}[section]
\newtheorem{lem}[thm]{Lemma}
\newtheorem{cor}[thm]{Corollary}
\newtheorem{prop}[thm]{Proposition}

\theoremstyle{definition}

\newtheorem{rem}[thm]{Remark}

\newtheorem{defn}[thm]{Definition}

\newtheorem{ex}[thm]{Example}

\newcommand{\mca}[1]{{\mathcal{#1}}}
\newcommand{\fbp}[2]{ \,_{#1}\!\!\times_{#2}}

\def\Z{{\mathbb Z}}
\def\C{{\mathbb C}}
\def\H{{\mathbb H}}
\def\R{{\mathbb R}}

\def\bL{\bar{\mca{L}}}
\def\batvil{\mca{BV}}

\def\dim{\text{\rm dim}\,}

\def\dR{\text{\rm dR}}

\def\E{\mca{E}}
\def\ep{\varepsilon} 
\def\endo{\text{\rm End}}

\def\ger{\mca{G}}

\def\HF{\text{\rm HF}}

\def\Hoch{\text{\rm Hoch}}
\def\Hom{\text{\rm Hom}}

\def\id{\text{\rm id}}
\def\Image{\text{\rm Im}}

\def\vlim{\varinjlim}

\def\Ker{\text{\rm Ker}}

\def\L{\mca{L}}

\def\len{\text{\rm len}}

\def\nm{\text{\rm nm}}

\def\ph{\varphi}
\def\plim{\varprojlim}
\def\pt{\text{\rm pt}}
\def\pr{\text{\rm pr}}

\def\reg{\text{\rm reg}}

\def\shift{\mathfrak{s}}

\def\sm{\text{\rm sm}}

\def\sym{{\mathbb S}}
\def\supp{\text{\rm supp}\,}

\begin{document}
\pagestyle{plain}
\thispagestyle{plain}

\title[A chain level Batalin-Vilkovisky structure in string topology via de Rham chains]
{A chain level Batalin-Vilkovisky structure in string topology via de Rham chains}

\author[Kei Irie]{Kei Irie} 
\address{Research Institute for Mathematical Sciences, Kyoto University,
Kyoto 606-8502, Japan}
\email{iriek@kurims.kyoto-u.ac.jp}

\subjclass[2010]{55P50, 57N75, 18D50} 

\begin{abstract} 
The aim of this paper is to define a chain level refinement of the Batalin-Vilkovisky (BV) algebra structure on the homology of the free loop space of a closed, oriented $C^\infty$-manifold. 
For this purpose, we define a (nonsymmetric) cyclic dg operad which consists of ``de Rham chains'' of free loops with marked points. 
A notion of de Rham chains, which is a certain hybrid of the notions of singular chains and differential forms, is a key ingredient in our construction. 
Combined with a generalization of cyclic Deligne's conjecture, 
this dg operad produces a chain model of the free loop space which admits an action of a chain model of the framed little disks operad, 
recovering the string topology BV algebra structure on the homology level. 
\end{abstract}

\maketitle

\section{Introduction} 

Let us begin with the following facts: 

\begin{enumerate} 
\item[(a):] 
For any differential graded algebra $A$, 
the Hochschild cohomology $H^*(A, A)$ has a Gerstenhaber algebra structure. 
\item[(b):] 
Let $M$ be a closed, oriented $d$-dimensional $C^\infty$-manifold, 
and let $\L M:= C^\infty(S^1, M)$ be the free loop space. 
Then, 
$\H_*(\L M):= H_{*+d}(\L M)$ 
has a Batalin-Vilkovisky (in particular, Gerstenhaber) algebra structure. 
\item[(c):] 
Let $\mca{A}_M$ denote the differential graded algebra of differential forms on $M$. 
There exists a linear map $\H_*(\L M : \R) \to H^*(\mca{A}_M, \mca{A}_M)$ defined by iterated integrals of differential forms, 
which preserves the Gerstenhaber structures.
\end{enumerate} 
The fact (a) is originally due to Gerstenhaber \cite{Gersten_63}. 
The fact (b) is due to Chas-Sullivan \cite{ChSu_99}, which is the very first paper on string topology. 
The fact (c) relates the geometric construction in (b) to the algebraic construction in (a). 
It seems that (c) is also well-known to specialists (see Remark \ref{150205_2}). 

(a)--(c) concern algebraic structures on the homology level, 
and one can expect that these structures lift to chain level structures which have more information. 
A toy model of this idea is the following well-known fact: 
the cup product on (singular) cohomology can be defined at the chain level, and 
(part of) extra information is used to define the classical Massey products; 
see \cite{LodVal_12} Section 9.4.5. 

For (a), what is called Deligne's conjecture claims that a certain chain model of the little disks operad acts on the Hochschild cochain complex. 
Various affirmative solutions to this conjecture and its variations are known; 
see 
\cite{MSS_02} Part I Section 1.19, 
\cite{LodVal_12} Section 13.3.15, 
and the references therein. 
The aim of this paper is to propose a chain level algebraic structure which lifts (b) (the Batalin-Vilkovisky (BV) algebra structure in string topology), 
and compare it with a solution to Deligne's conjecture via a chain map which is a chain level lift of (c). 

Our construction is based on generalizations of Deligne's conjecture in operadic contexts. 
More specifically, we use results in a recent paper \cite{Ward_14}, 
which we describe here very briefly. 
For any (nonsymmetric) dg operad $\mca{O} = (\mca{O}(k))_{k \ge 0}$ and a Maurer-Cartan element $\zeta = (\zeta_k)_{k \ge 2}$ of $\mca{O}$, one can assign a chain complex $(\tilde{\mca{O}}, \partial_\zeta)$ 
(here we use our notation in Section 2.5, which is different from the notation in \cite{Ward_14}). 
Then Theorem A in \cite{Ward_14} asserts that the chain complex $\tilde{\mca{O}}$ admits an action of a certain chain model of the little disks operad. 
Moreover, if $\mca{O}$ has a unital cyclic structure and $\zeta$ is cyclically invariant, 
Theorem B in \cite{Ward_14} asserts that this action extends to an action of  a certain chain model of the framed little disks operad 
(strictly speaking, in Theorem B one has to consider a quasi-isomorphic subcomplex $\tilde{\mca{O}}^\nm \subset \tilde{\mca{O}}$ which consists of ``normalized chains''). 
Actually, in this paper we consider only the simple case that $\zeta_k=0$ for every $k \ne 2$. 
A Maurer-Cartan element satisfying this condition is equivalent to a ``multiplication'' of the operad
(see Definition \ref{150624_4}). 

Let us briefly describe our main result.
For any closed oriented $C^\infty$-manifold $M$, 
we define a nonsymmetric cyclic dg operad $\mca{O}_M$ with a multiplication and a unit. 
Applying Theorem B in \cite{Ward_14}, 
the associated chain complex $\widetilde{\mca{O}_M}^\nm$ admits an action of a chain model of the framed little disks operad. 
In particular, the homology 
$H_*(\widetilde{\mca{O}_M}^\nm) \cong H_*(\widetilde{\mca{O}_M})$ has a BV algebra structure. 
Then we show that there exists an isomorphism $H_*(\widetilde{\mca{O}_M}) \cong \H_*(\mca{L}M:\R)$ of BV algebras. 

The chain complexes $\widetilde{\mca{O}_M}^\nm \subset \widetilde{\mca{O}_M}$ are chain models of the free loop space $\mca{L}M$. 
To show that these chain models have nontrivial chain level information, we establish the following properties: 
\begin{itemize} 
\item There exists a chain map $\widetilde{\mca{O}_M}_* \to C^*(\mca{A}_M, \mca{A}_M)$
which preserves actions of the little disks operad and recovers the map 
$\H_*(\mca{L}M) \to H^*(\mca{A}_M, \mca{A}_M)$ in (c) on the homology level. 
\item 
Algebraic structures on $\widetilde{\mca{O}_M}$ induce 
$A_\infty$ and $L_\infty$ algebra structures on $\H_*(\L M)$. 
In particular, the $A_\infty$ structure on $\H_*(\L M)$ recovers the classical Massey products on $\H_*(M)$. 
\item 
When $M$ has a Riemannian metric, 
$\widetilde{\mca{O}_M}$ and $\widetilde{\mca{O}_M}^\nm$ are equipped with filtrations which come from lengths of loops. 
\end{itemize} 

There may be several different ways to work out chain level structures in string topology, based on choices of chain models of the free loop space.
The singular chain complex has a lot of geometric information, however it has transversality trouble. 
Namely, string topology operations (e.g. the loop product) are defined only for chains transversal to each other. 
The Hochschild cochain complex of differential forms (used e.g. in \cite{Mer_04}) avoids this trouble. 
However it seems that this chain complex is not always a correct chain model of the free loop space (see Remark \ref{160617_1}). 
Also, this chain model loses some geometric information, such as lengths of loops. 
Our chain model is intermediate between these two, 
and a key ingredient in our construction is a notion of ``de Rham chains'', 
which is a certain hybrid of the notions of singular chains and differential forms. 

This paper is organized as follows. 
Section 2 explains preliminary materials on operads, string topology and Deligne's conjecture; 
this section contains no new results. 
Section 3 states our main results in a rigorous form, and discusses previous related work and potential applications to symplectic topology. 
The rest of this paper (Sections 4--8) is devoted to proofs of these results, 
and the plan of the proofs is explained in the last part of Section 3. 

\section{Preliminaries}

The aim of this section is to explain several preliminary materials, fixing various notation and signs. 
We recall basic definitions and facts about operads (Section 2.1), Gerstenhaber structure on Hochschild cohomology (Section 2.2), string topology (Section 2.3) and iterated integrals (Section 2.4).
In Section 2.5, we discuss operadic Deligne's conjecture, mainly following \cite{Ward_14}. 

\subsection{Preliminaries from operads} 

\subsubsection{Operads} 
First we briefly recall the notion of (nonsymmetric) operads. The main aim is to fix conventions, 
and we refer to \cite{MSS_02} Part II Section 1.2 for details. 

Let $\mca{C}$ be any symmetric monoidal category with a multiplication $\times$ and a unit $1_\mca{C}$. 
A \textit{nonsymmetric operad} $\mca{P}$ in $\mca{C}$ consists of the following data:
\begin{itemize}
\item An object $\mca{P}(n)$ for every integer $n \ge 0$. 
\item A morphism $\circ_i: \mca{P}(n) \times \mca{P}(m) \to \mca{P}(n+m-1)$ for every $1 \le i \le n$ and $m \ge 0$. 
These morphisms are called \textit{(partial) compositions}. 
\item A morphism $1_\mca{P}: 1_\mca{C} \to \mca{P}(1)$ called a \textit{unit} of $\mca{P}$. 
\end{itemize} 
We require that compositions satisfy associativity, 
and $1_{\mca{P}}$ is a two-sided unit for compositions. 
When $\mca{P}(n)$ admits a right action of the symmetric group $\sym_n$ ($\sym_0$ is the trivial group) for each $n \ge 0$, 
such that these actions are compatible with compositions, $\mca{P}$ is called an \textit{operad} in $\mca{C}$. 

For any (nonsymmetric) operads $\mca{P}$ and $\mca{Q}$, a morphism of (nonsymmetric) operads $\ph:\mca{P} \to \mca{Q}$
is a sequence of morphisms $(\ph(n): \mca{P}(n) \to \mca{Q}(n))_{n \ge 0}$ which preserves the above structures. 
When $\ph(n)$ are monic for all $n \ge 0$, we say that $\mca{P}$ is a suboperad of $\mca{Q}$. 

\subsubsection{Graded and dg operads} 
Throughout this paper, all vector spaces are defined over $\R$. 
A graded vector space $V$ is a sequence $(V_n)_{n \in \Z}$ of vector spaces. 
A differential graded (or dg) vector space (or chain complex)  is a pair $(V, \partial)$ of a graded vector space $V$ and $\partial: V_* \to V_{*-1}$ satisfying $\partial^2=0$. 
We may consider any graded vector space as a dg vector space with $\partial=0$. 
One can define a symmetric monoidal structure on the category of dg vector spaces as follows: 
\begin{align*}
(V \otimes W)_n&:= \bigoplus_{i+j=n} V_i \otimes W_j, \\
\partial(v \otimes w)&:= \partial v \otimes w + (-1)^{|v|} v \otimes \partial w, \\
V \otimes W \to W \otimes V&; v \otimes w \mapsto (-1)^{|v||w|} w \otimes v.
\end{align*} 
The unit $\bar{\R}$ is defined by $\bar{\R}_*:= \begin{cases} \R \cdot  [1]_0 &(*=0) \\ 0 &(* \ne 0) \end{cases}$
and $\partial=0$. 
In this paper, we mainly work in the categories of graded and dg vector spaces. 
Operads in these categories are called \textit{graded operads} and \textit{dg operads}, respectively. 

For any dg vector spaces $V$ and $W$, $\Hom(V,W)$ has the structure of a dg vector space: 
\[
\Hom(V,W)_n: =\prod_{k \in \Z} \Hom (V_k, W_{k+n}), \qquad
(\partial f)(v):= \partial(f(v)) - (-1)^{|f|} f(\partial v).
\]
For any dg vector space $V$, 
$\endo(V):= (\Hom(V^{\otimes n}, V))_{n \ge 0}$ has the structure of a dg operad defined as follows 
($f \in \Hom(V^{\otimes n}, V)$, $g \in \Hom(V^{\otimes m}, V)$, and $\sigma \in \sym_n$): 
\begin{align*}
(f \circ_i g)(v_1 \otimes \cdots \otimes v_{m+n-1})&:=(-1)^{|g|(|v_1|+\cdots + |v_{i-1}|)} f(v_1 \otimes \cdots \otimes g(v_i \otimes \cdots \otimes v_{i+m-1})\otimes \cdots \otimes v_{m+n-1}), \\
1_{\endo(V)} &:= \id_{V}, \\
(f^\sigma)(v_1 \otimes \cdots \otimes v_n) &:= \biggl( \prod_{\substack{i<j \\ \sigma(i)>\sigma(j)}} (-1)^{|v_i||v_j|}  \biggr) \cdot  f(v_{\sigma^{-1}(1)} \otimes \cdots \otimes v_{\sigma^{-1}(n)}). 
\end{align*} 
This dg operad is called the \textit{endomorphism operad} of $V$. 

For any dg operad $\mca{P}$, a dg $\mca{P}$-algebra is a chain complex $V$ with a
morphism $\mca{P} \to \endo(V)$ of dg operads. 
For each $n \ge 0$ we have a chain map 
\[
\mca{P}(n) \otimes V^{\otimes n} \to V; \quad x \otimes v_1 \otimes \cdots \otimes v_n \mapsto x \cdot (v_1 \otimes \cdots \otimes v_n).
\]
For any dg $\mca{P}$-algebras $V$ and $W$, 
a chain map $\ph:V \to W$ is called a morphism of dg $\mca{P}$-algebras if 
$\ph(x \cdot (v_1 \otimes \cdots \otimes v_n)) = x \cdot (\ph(v_1) \otimes \cdots \otimes \ph(v_n))$ for any 
$x \in \mca{P}(n)$ and $v_1, \ldots, v_n \in V$.  
For any graded operad $\mca{P}$, the notions of graded $\mca{P}$-algebras and their morphisms are defined in a similar way. 

For any dg operad $\mca{P}=(\mca{P}(n))_{n \ge 0}$, 
a \textit{dg ideal} of $\mca{P}$ is a sequence $\mca{Q}=(\mca{Q}(n))_{n \ge 0}$
such that the following conditions hold: 
\begin{itemize} 
\item For every $n \ge 0$, $\mca{Q}(n)$ is a chain subcomplex of $\mca{P}(n)$, which is preserved by the $\sym_n$-action on $\mca{P}(n)$. 
\item For any $x \in \mca{P}(n)$, $y \in \mca{P}(m)$ and $1 \le i \le n$, 
\[
x \in \mca{Q}(n)  \quad \text{or} \quad y \in \mca{Q}(m) \implies x \circ_i y \in \mca{Q}(n+m-1). 
\]
\end{itemize} 
For any dg ideal $\mca{Q} \subset \mca{P}$, the quotient 
$\mca{P}/\mca{Q}:= (\mca{P}(n)/\mca{Q}(n))_{n \ge 0}$ has a natural structure of a dg operad, 
and there exists a natural morphism of dg operads $\mca{P} \to \mca{P}/\mca{Q}$. 
For any graded operad $\mca{P}$, the notions of its graded ideals and associated quotient graded operads are defined in the obvious way 
(see \cite{LodVal_12} Section 5.2.14).

\subsubsection{Gerstenhaber and Batalin-Vilkovisky (BV) operads}

The Gerstenhaber and BV operads are graded operads, which play central roles in this paper. 
We recall definitions of these operads using generators and relations, 
partially following \cite{LodVal_12} Sections 13.3.12 and 13.7.4. 
 
The \textit{Gerstenhaber operad} $\ger$ is generated by
$a \in \ger(2)_0$ and $b \in \ger(2)_1$ with the following relations: 
\begin{enumerate} 
\item[(a):] $a^{(12)} = a, \quad a \circ_ 1 a = a \circ_2 a$. 
\item[(b):] $b^{(12)} = b, \quad b \circ_1 b + (b \circ_1 b)^{(123)} + (b \circ_1 b)^{(321)} = 0$. 
\item[(ab):] $b \circ_1 a = a \circ_2 b  + (a \circ_1 b)^{(23)}$. 
\end{enumerate}
More precisely, we consider the free operad $E$ (see \cite{LodVal_12} Section 5.5)
generated by $a$ and $b$, 
and a graded ideal $R \subset E$ generated by relations (a), (b) and (ab). 
Then, we define $\ger:= E/R$. 
We set $\ger(0):= 0$. 

For any graded $\ger$-algebra $V$, we define operations $\bullet$ and $\{\, , \,\}$ on $V$ as 
\[
v \bullet  w:= a \cdot(v \otimes w), \qquad 
\{v, w\} := (-1)^{|v|} b \cdot (v \otimes w). 
\]
Then, $(V, \bullet)$ is a graded commutative, associative algebra, and 
$(V, \{ \, , \,\})$ is a graded Lie algebra (with grading shifted by $1$). 
The triple $(V, \bullet, \{\, , \,\})$ is called a \textit{Gerstenhaber algebra}. 

The \textit{BV operad} $\batvil$ is generated by 
$a \in \batvil(2)_0$, $b \in \batvil(2)_1$ and $\Delta \in \batvil(1)_1$ 
with the relations (a), (b), (ab) and 
\[
\Delta \circ_1 \Delta = 0, \qquad 
b = \Delta \circ_1 a - a \circ_1 \Delta - a \circ_2 \Delta. 
\]
We set $\batvil(0):=0$. 
Obviously, there exists a natural morphism of graded operads $\ger \to \batvil$. 

For any graded $\batvil$-algebra $V$, 
we define an operation $\Delta$ on $V$ by  
$\Delta(v) := \Delta \cdot v$. 
The triple $(V, \bullet, \Delta)$ is called a \textit{BV algebra}. 
The bracket $\{\, , \,\}$ is recovered by the formula 
\[
\{v,w\} = (-1)^{|v|} \Delta(v \bullet w) - (-1)^{|v|} \Delta v \bullet w - v \bullet \Delta w.
\]

For any integer $r \ge 1$, 
let $f\mca{D}(r)$ be the set of tuples $(D_1, \ldots, D_r, z_1, \ldots, z_r)$ such that 
\begin{itemize}
\item For each $1 \le i \le r$, 
$D_i$ is a closed disk of positive radius contained in $\{ z \in \C \mid |z| \le 1\}$. 
We denote its center by $p_i$, 
and $z_i$ is a point on $\partial D_i$. 
\item $D_1, \ldots, D_r$ are disjoint. 
\end{itemize} 
The set $f\mca{D}(r)$ has a natural topology. 
Let $\mca{D}(r)$ denote the subspace of $f\mca{D}(r)$ which consists of $(D_1, \ldots, D_r, z_1, \ldots, z_r)$ such that 
$z_i - p_i \in \R_{>0}$ for every $1 \le i \le r$. 
We define $f\mca{D}(0)=\mca{D}(0)$ to be the empty set. 
Then $f\mca{D}=(f\mca{D}(r))_{r \ge 0}$ has a natural structure of a topological operad, 
and $\mca{D}=(\mca{D}(r))_{r \ge 0}$ is a suboperad of it. 
The operad $\mca{D}$ (resp. $f\mca{D}$) is called the \textit{little disks} (resp. \textit{framed little disks}) operad. 
There are isomorphisms of graded operads 
$H_*(\mca{D}) \cong \ger$ (\cite{Cohen_76}) and 
$H_*(f\mca{D}) \cong \batvil$ (\cite{Getzler_94}),
which are compatible with the inclusion maps
(recall that we are working with $\R$ coefficients). 

\subsection{Gerstenhaber structure on Hochschild cohomology}

A \textit{differential graded associative algebra} is a dg vector space $A$ 
with a degree $0$ product $A \otimes A \to A$
which is associative and satisfies the Leibniz rule. 
We also assume that it has a unit $1_A \in A_0$. 

\begin{rem} 
We abbreviate the term ``differential graded associative algebra'' as ``dga algebra''. 
Here the letter ``a'' stands for ``associative'', not for algebra (see \cite{LodVal_12} pp. 29).
\end{rem} 

Let $A$ be any dga algebra. 
A dg $A\,$-bimodule is a dg vector space $M$ 
with degree $0$ left and right $A\,$-actions 
$A \otimes M \to M$ and $M \otimes A \to M$, 
which satisfy the Leibniz rule and associativity. 

For every $k \ge 1$ and $i=0, \ldots, k$, we define a chain map 
$\delta_{k,i}: \Hom_*(A^{\otimes k-1}, M) \to \Hom_*(A^{\otimes k}, M)$ by 
\[
\delta_{k,i}(f)(a_1 \otimes \cdots \otimes a_k):= 
\begin{cases}
(-1)^{|a_1||f|} a_1 \cdot f(a_2 \otimes \cdots \otimes a_k) &(i=0), \\
f(a_1 \otimes \cdots \otimes a_i a_{i+1} \otimes \cdots \otimes a_k) &(1 \le i \le k-1), \\
f(a_1 \otimes \cdots \otimes a_{k-1}) \cdot a_k &(i=k). 
\end{cases}
\]
We set $\delta_k: \Hom_*(A^{\otimes k-1}, M) \to \Hom_*(A^{\otimes k}, M)$ by 
$\delta_k(f) := (-1)^{|f|+k-1} \sum_{i=0}^k  (-1)^i \delta_{k,i}(f)$, 
and define the Hochschild cochain complex 
$C^*(A, M):= \biggl(\prod_{k=0}^\infty \Hom_{*+k}(A^{\otimes k}, M), \partial_\Hoch \biggr)$ by 
$\partial_\Hoch(f_k)_{k \ge 0}: = (\partial f_k)_{k \ge 0}  + (\delta_k(f_{k-1}))_{k \ge 1}$.
Notice that $\partial_\Hoch$ \textit{decreases} the degree by $1$. 
The cohomology of this complex is denoted by $H^*(A,M)$, and called the \textit{Hochschild cohomology}.

Notice that $A$ itself has a natural structure of a dg $A\,$-bimodule. 
$C^*(A, A)$ has natural dga and dg Lie algebra structures, 
with operations $\circ$ and $\{\, , \,\}$ defined below. 
Signs in these formulas follow from Theorem \ref{160420_1}. 

The product $\circ$ is defined by 
\begin{align*} 
(f \circ g)_k(a_1 \otimes \cdots \otimes a_k) &:= \sum_{l+m=k} (-1)^\dagger f_l(a_1 \otimes \cdots \otimes a_l) g_m(a_{l+1} \otimes \cdots \otimes a_k), \\
\dagger&:= l |g| + (|g|+m) (|a_1| + \cdots + |a_l|). 
\end{align*} 
The bracket $\{ \, , \,\}$ is defined by 
\[
\{f, g\}: = f*g - (-1)^{ (|f|+1)(|g|+1)} g*f , 
\]
where $*$ is defined by   
\begin{align*} 
&(f*g)_k (a_1 \otimes \cdots \otimes a_k):= \\
&\qquad \sum_{\substack{l+m=k+1 \\ 1 \le i \le l}}
(-1)^\dagger f_l(a_1 \otimes \cdots \otimes a_{i-1} \otimes g_m(a_i \otimes \cdots \otimes a_{i+m-1}) \otimes a_{i+m} \otimes \cdots \otimes a_k),  \\ 
&\dagger := (i-1)(m-1) + (|g|+ m) (|a_1| + \cdots + |a_{i-1}| + l-1). 
\end{align*} 
The operations $\circ$ and $\{ \, , \,\}$ induce a Gerstenhaber structure on $H^*(A, A)$. 
This result is originally due to Gerstenhaber \cite{Gersten_63}. 

\subsection{BV structure in string topology} 

Throughout this paper, we set $S^1:=\R/\Z$. 
Let $M$ be a closed, oriented $C^\infty$-manifold of dimension $d$. 
The free loop space $\L M:= C^\infty(S^1, M)$ is equipped with the $C^\infty$-topology. 
We often abbreviate $\L M$ as $\L$. 
Also, we often use the notation $\H_*(\, \cdot \,):= H_{*+d}(\, \cdot \,)$. 
In \cite{ChSu_99}, Chas-Sullivan introduced the \textit{loop product} on $\H_*(\L)= H_{*+d}(\L)$. 
Let us briefly recall its definition. 

Let us consider the evaluation map  $e: \mca{L} \to M; \gamma \mapsto \gamma(0)$, 
and the fiber product 
\[
\mca{L} \fbp{e}{e} \mca{L}:= \{(\gamma, \gamma') \in \mca{L}^{\times 2} \mid \gamma(0) = \gamma'(0)\}.
\]
Let $U$ be a tubular neighborhood of $\L \fbp{e}{e} \L \subset \L^{\times 2}$, 
and 
$H_*(U, U \setminus \L \fbp{e}{e} \L) \cong H_{*-d}(\L \fbp{e}{e} \L)$ be the Thom isomorphism. 
The Gysin map $H_*(\L^{\times 2}) \to H_{*-d}(\L \fbp{e}{e} \L)$ is defined as the composition of the following maps:
\[
H_*(\L^{\times 2}) \to H_*(\L^{\times 2}, \L^{\times 2} \setminus \L \fbp{e}{e} \L) \cong H_*(U, U \setminus \L \fbp{e}{e} \L) \cong H_{*-d}(\L \fbp{e}{e}  \L). 
\]

Let $c: \L \fbp{e}{e}  \L \to \L$ denote the concatenation map. 
Precisely, it is defined as follows (see the remark on pp. 780 \cite{CJ_02}). 
Let us take an  increasing $C^\infty\,$-function $\nu: [0,1] \to [0,1]$ such that 
$\nu(t)=t$ and $\nu^{(m)}(t)=0\, (\forall m \ge 1)$ ($\nu^{(m)}$ denotes the $m$-th derivative) for any $t \in \{0, 1/2, 1\}$. 
Then, 
$c: \L \fbp{e}{e}  \L \to \L$ is defined by 
\[
c(\gamma_1, \gamma_2)(t):= \begin{cases} 
                                       \gamma_1(2\nu(t)) &(0 \le t \le 1/2), \\
                                       \gamma_2(2\nu(t)-1) &(1/2 \le t \le 1).
                                      \end{cases}
\]
It is easy to see that $H_*(c): H_*(\L \fbp{e}{e}  \L) \to H_*(\L)$ does not depend on the choices of $\nu$. 

The loop product $\bullet: \H_*(\L)^{\otimes 2} \to \H_*(\L)$ is defined as the composition of the following three maps. 
The first map is the cross product and the second map is the Gysin map. 
\[
\xymatrix{
\H_*(\L)^{\otimes 2} \ar[r]^-{\times} & H_{*+2d}(\L^{\times 2})  \ar[r] & \H_*(\L \fbp{e}{e} \L) \ar[r]^-{\H_*(c)}  & \H_*(\L).
}
\]
Let us consider the map $i_M: M \to \mca{L}M$
which takes each $p \in M$ to the constant loop at $p$. 
Let $\cap: \H_*(M)^{\otimes 2} \to \H_*(M)$ denote the intersection product. Then, 
\[
\H (i_M)(x) \bullet \H_*(i_M)(y) = \H_*(i_M)(x \cap y) \qquad(\forall x, y \in \H_*(M)). 
\]

On the other hand, $\L$ admits a natural $S^1$-action 
$r: S^1 \times \L \to \L$, which is defined by 
$r(t, \gamma)(\theta):=\gamma(\theta-t)$. 
We define $\Delta: \H_*(\L) \to \H_{*+1}(\L)$ by 
$\Delta(x) := H_*(r)([S^1] \times x)$, where
$[S^1]  \in H_1(S^1)$ is represented by the singular chain 
$\Delta^1 \to S^1; \, t \mapsto [t]$ (see Section 2.4 for our definition of $\Delta^1$). 

\begin{thm}[\cite{ChSu_99}, \cite{CJ_02}, \cite{Chat_05}]\label{thm:Ch-Su}
For any closed, oriented $C^\infty$-manifold $M$ of dimension $d$, 
the triple 
$(\H_*(\L M), \bullet, \Delta)$ is a BV algebra. 
\end{thm} 

This result is the starting point of string topology. 
The bracket $\{\, , \,\}$ of this BV structure is called the \textit{loop bracket}. 

\subsection{Iterated integrals of differential forms} 
There is a relation between 
the Gerstenhaber structures on loop space homology (Theorem \ref{thm:Ch-Su}) and 
Hochschild cohomology (Section 2.2). 
We explain this relation via iterated integrals of differential forms, 
the theory of which originates in \cite{KTChen_77}. 

To discuss iterated integrals of differential forms, 
it is convenient to work with $C^\infty$-singular chains on $\L M$. 
Let us define the $k$-dimensional simplex $\Delta^k$ by 
\[
\Delta^k:= \begin{cases} \R^0&(k=0), \\ 
\{(t_1, \ldots, t_k) \in \R^k \mid 0 \le t_1 \le \cdots \le t_k \le 1\}. &(k \ge 1). \end{cases}
\]

A map $\sigma: \Delta^k \to \L M $ is said to be of class $C^\infty$, if 
there exists an open neighborhood $U$ of $\Delta^k \subset \R^k$ 
and a map $\bar{\sigma}: U \to \L M$, such that 
$\bar{\sigma}|_{\Delta^k}=\sigma$, and 
$U \times S^1 \to M; (u,\theta) \mapsto \bar{\sigma}(u)(\theta)$ is of class $C^\infty$. 
Let $C^\sm_k(\L M)$ denote the $\R$-vector space generated by all $C^\infty$-maps $\Delta^k \to \L M$. 
It is easy to see that any $C^\infty$-map $\sigma: \Delta^k \to \L M$ is continuous
with respect to the $C^\infty$-topology on $\L M$. 
Therefore, $C^\sm_*(\L M)$ is a subcomplex of the singular chain complex of $\L M$
(see Section 4.7 for our convention for the boundary operator of the singular chain complex). 
In Section 6, we show that this inclusion map is a quasi-isomorphism (Theorem \ref{150219_1}). 
Therefore, $H^\sm_*(\L M):= H_*(C^\sm(\L M)) \cong H_*(\L M)$. 

For any $j \in \Z$, let us define 
\[
\mca{A}^j(M):= \begin{cases} \text{$\R$-vector space of  $C^\infty$ $j$-forms on $M$} &(0 \le j \le d=\dim M), \\ 0 &(\text{otherwise}). \end{cases}
\]
Then, $(\mca{A}^{-*}(M), d, \wedge)$ is a dga algebra, where 
$d$ denotes the exterior derivative, and $\wedge$ denotes the exterior product. 
We denote it by $\mca{A}_M$. 
We define a dg $\mca{A}_M\,$-bimodule structure on 
$\mca{A}_M^\vee[d]_*:=\Hom(\mca{A}^{*+d}(M), \R)$ 
as follows: 
\begin{align*}
(\partial \ph)(\alpha)&:= (-1)^{|\ph|+1} \ph(d\alpha), \\
(\alpha \cdot \ph)(\beta)&:= (-1)^{|\alpha||\ph|} \ph(\alpha \wedge \beta), \qquad
(\ph \cdot \alpha)(\beta):= \ph(\alpha \wedge \beta).
\end{align*} 
The morphism of $\mca{A}_M$-bimodules 
$\mca{A}_M \to \mca{A}_M^{\vee}[d \,]$
defined by 
$\alpha \mapsto (\beta \mapsto \int_M \alpha \wedge \beta)$ is a quasi-isomorphism
(this is an obvious consequence of Poincar\'{e} duality). 

For any $C^\infty$-map 
$\sigma: \Delta^l \to \mca{L}M$ and 
$i=0, \ldots, k$, 
we define $\sigma_{k,i}: \Delta^l \times \Delta^k \to M$ by
\[
\sigma_{k,i}(x, t_1, \ldots, t_k):= \begin{cases} \sigma(x)(0) &(i=0) \\ \sigma(x)(t_i) &(1 \le i \le k), \end{cases}
\]
and define $I_k(\sigma) \in \Hom(\mca{A}_M^{\otimes k}, \mca{A}_M^{\vee}[d])$ by 
\[
I_k(\sigma)(\eta_1 \otimes \cdots \otimes \eta_k)(\eta_0) :=  (-1)^{l(d+1) + (k+l)(k+l-1)/2} 
\int_{\Delta^l \times \Delta^k} \sigma_{k,1}^*\eta_1 \wedge \cdots \wedge \sigma_{k,k}^* \eta_k \wedge \sigma_{k,0}^*\eta_0. 
\]
It is easy to see that
\begin{equation}\label{141222_4}
I: C^\sm_{*+d}(\L M) \to C^*(\mca{A}_M, \mca{A}_M^{\vee}[d]); \quad \sigma \mapsto (I_k(\sigma))_{k \ge 0} 
\end{equation} 
is a chain map (signs are checked in Section 8.3). 
Taking homology, we obtain a map 
\begin{equation}\label{141218_01}
\H_*(\L M) \to H^*(\mca{A}_M, \mca{A}_M^\vee[d]) \cong H^*(\mca{A}_M, \mca{A}_M).
\end{equation} 
This map preserves the Gerstenhaber structures on 
$\H_*(\L M)$ and $H^*(\mca{A}_M, \mca{A}_M)$. 

\begin{rem}\label{150205_2}
The fact that (\ref{141218_01}) preserves the Gerstenhaber structures seems to be known; 
see \cite{Mer_04} for the product, and \cite{Fuk_06} Section 7 for the bracket. 
We can recover this fact as a consequence of Theorem \ref{160408_1}, 
see Remark \ref{160421_1} (ii). 
\end{rem} 

\begin{rem}\label{160617_1} 
It is a fundamental result in the theory of iterated integrals that 
the map (\ref{141218_01}) is an isomorphism if $M$ is simply-connected (\cite{KTChen_77}). 
It seems that (\ref{141218_01}) fails to be an isomorphism for arbitrary $M$, 
although the author is not aware of any specific example. 
\end{rem} 

\subsection{Operadic Deligne's conjecture} 

The original Deligne's conjecture concerns Hochschild cochains of associative algebras. 
Here we review generalizations in operadic contexts, which seem to originate in \cite{Kaufmann_07}. 
After some algebraic preliminaries, 
we recall these results in both cyclic and noncyclic versions, mainly following a recent paper \cite{Ward_14}. 
Finally we discuss sign conventions in \cite{Ward_14} and compare them to sign conventions
in the present paper. 

\subsubsection{Algebraic preliminaries} 

A double complex $C$ consists of a sequence $(C(k))_{k \ge 0}$ of chain complexes and 
anti-chain maps $\delta_k: C(k-1)_* \to C(k)_*$ (i.e. $\delta_k$ anti-commutes with differentials) 
for every $k \ge 1$, such that 
$\delta_{k+1} \circ \delta_k = 0$  for every $k \ge 1$. 
We denote $\delta_k$ by $\delta^C_k$ if necessary. 
For any double complexes $C$ and $D$, a morphism $\ph: C \to D$ is a sequence $\ph=(\ph(k))_{k \ge 0}$ such that, 
$\ph(k): C(k)_* \to D(k)_*$ is a chain map and 
$\delta^D_k \circ \ph(k-1) = \ph(k) \circ \delta^C_k$ for every $k \ge 1$. 

For any double complex $C$, 
we define the \textit{total complex} $(\tilde{C}, \tilde{\partial})$ by 
\[
\tilde{C}_*:= \prod_{k=0}^\infty C(k)_{*+k}, \qquad 
(\tilde{\partial} x)_k := \begin{cases} \partial x_0 &(k=0) \\ \partial x_k + \delta_k (x_{k-1}) &(k \ge 1) \end{cases} 
\]
where $(\tilde{\partial}x)_k$ denotes the factor of $\tilde{\partial} x$ in $C(k)$. 
A morphism $\ph: C \to D$ of double complexes induces 
a chain map $\tilde{\ph}: \tilde{C} \to \tilde{D}; \,(x_k)_{k \ge 0} \mapsto (\ph(k)(x_k))_{k \ge 0}$. 

Recall that a \textit{cosimplicial chain complex} $C$ consists of a family of chain complexes $(C(k))_{k \ge 0}$
with a family of chain maps 
\begin{align*} 
\delta_{k,i} &: C(k-1)_* \to C(k)_*  \quad  (0 \le i \le k) \\ 
\sigma_{k,i} &: C(k+1)_* \to C(k)_* \quad (0 \le i \le k)
\end{align*} 
satisfying the following relations: 
\begin{align*} 
\delta_{k+1, j} \circ \delta_{k, i} &= \delta_{k+1, i} \circ \delta_{k, j-1} \quad (i < j) \\
\sigma_{k-1, j} \circ \sigma_{k, i}  &= \sigma_{k-1, i} \circ \sigma_{k, j+1} \quad (i \le j) \\
\sigma_{k, j} \circ \delta_{k+1, i} &=
\begin{cases} 
\delta_{k, i} \circ \sigma_{k-1, j-1} &(i < j) \\ 
\id  &(i=j, j+1) \\
\delta_{k, i-1} \circ \sigma_{k-1, j} &(i > j+1). 
\end{cases}
\end{align*} 
For later use let us also recall the notion of \textit{cocyclic chain complexes}; 
a cocyclic chain complex is a cosimplicial chain complex $C$ with a family of chain maps 
\[
\tau_k: C(k)_* \to C(k)_* \qquad( k \ge 0)
\]
satisfying the following relations: 
\begin{align*} 
\tau_k^{k+1} &= \id \\ 
\tau_k \circ \delta_{k,i} &= \begin{cases} \delta_{k, k} &(i=0) \\ \delta_{k, i-1} \circ \tau_{k-1} &(1 \le i \le k) \end{cases} \\
\tau_k \circ \sigma_{k,i} &= \begin{cases} \sigma_{k,k} \circ \tau_{k+1}^2 &(i=0) \\ \sigma_{k,i-1} \circ \tau_{k+1} &(1 \le i \le k). \end{cases}
\end{align*} 

For any cosimplicial chain complex $C$ and each $k \ge 1$, 
let us define an anti-chain map $\delta_k: C(k-1)_* \to C(k)_*$ by 
$\delta_k(x):= (-1)^{|x| + k-1} \sum_{i=0}^k (-1)^i \delta_{k, i}(x)$. 
Then $C$ is a double complex. 
For every $k \ge 1$, we call $x \in C(k)$ \textit{normalized} if 
$\sigma_{k-1, i}(x)= 0$ for every $0 \le i \le k-1$. 
The set of normalized elements in $C(k)$ is denoted by $C^\nm(k)$. 
Namely, for every $k \ge 1$ we set 
\[
C^\nm(k) := \{ x \in C(k) \mid \sigma_{k-1, i}(x) = 0 \, (0 \le \forall i \le k-1) \}. 
\]
We also set $C^\nm(0):= C(0)$. 
It is easy to see that $\delta_k (C^\nm(k-1)) \subset C^\nm(k)$ for every $k \ge 1$, 
thus 
$\tilde{C}^\nm_*:= \prod_{k=0}^\infty C^\nm(k)_{*+k}$ is a subcomplex of $\tilde{C}_*$. 
Lemma \ref{160503_1} below seems to be well-known 
(see Proposition 1.5 in \cite{Loday_98} or Theorem 8.3.8 in \cite{Weibel_94}), 
nevertheless we sketch a proof of it. 

\begin{lem}\label{160503_1} 
The inclusion map $\tilde{C}^\nm_* \to \tilde{C}_*$ is a quasi-isomorphism.
\end{lem} 
\begin{proof} 
For any integers $m,k \ge 0$, we define 
\[
F^m C(k) := \{ x \in C(k) \mid \sigma_{k-1, i} (x)=0 \quad(0 \le \forall i \le \min \{m-1, k-1\}) \}. 
\]
For each $k \ge 0$, there holds 
\[
C(k) = F^0 C(k) \supset F^1C(k) \supset \cdots \supset F^k C(k) = F^{k+1}C(k)= \cdots = C^\nm(k). 
\]
It is easy to check that for every $k \ge 1$ and $m \ge 0$, there holds 
$\delta_k(F^m C(k-1)) \subset F^mC(k)$. 
For every $m \ge 0$, let us define $F^m \tilde{C}_*:= \prod_{k=0}^\infty F^m C(k)_{*+k}$. 
Then we obtain a decreasing sequence of chain complexes 
\[
\tilde{C} = F^0 \tilde{C} \supset F^1 \tilde{C} \supset F^2 \tilde{C} \supset \cdots 
\]
such that $\bigcap_{m \ge 0} F^m \tilde{C}:= \tilde{C}^\nm$. 

For each $m \ge 0$, 
let us define a map $K$ on $F^m \tilde{C}/ F^{m+1} \tilde{C} = \prod_{k=m+1}^\infty F^mC(k)/ F^{m+1}C(k)$ by 
\[ 
(Kx)_k:= (-1)^{|x|} \sigma_{k,m} (x_{k+1})  \qquad ( \forall k \ge m+1). 
\] 
Then a direct computation shows that 
$K\tilde{\partial} + \tilde{\partial} K = (-1)^m \id$. 
Therefore $F^m \tilde{C}/ F^{m+1}\tilde{C}$ is acyclic for every $m \ge 0$, 
and thus, $\tilde{C}/F^m\tilde{C}$ is acyclic for every $m \ge 0$. 
Since $\tilde{C}/F^{m+1} \tilde{C} \to \tilde{C}/F^m\tilde{C}$ is surjective for every $m$,
Theorem 3.5.8 in \cite{Weibel_94} 
shows that $H_*(\tilde{C}/\tilde{C}^\nm)=0$. 
\end{proof}

\subsubsection{Operadic  Deligne's conjecture: noncyclic version} 

Let us recall the notion of operads with multiplications, 
which seems to originate in \cite{GerstVoronov_95}. 

\begin{defn}\label{150624_4} 
Let $\mca{O} = (\mca{O}(k))_{k \ge 0}$ be a nonsymmetric dg operad. 
$\mu \in \mca{O}(2)_0$ is called a \textit{multiplication} of $\mca{O}$, if 
$\partial \mu = 0$ and $\mu \circ_1 \mu = \mu \circ_2 \mu$. 
$\ep \in \mca{O}(0)_0$ is called a \textit{unit} of $\mu$, if 
$\partial \ep = 0$ and $\mu \circ_1 \ep = \mu \circ_2 \ep = 1_\mca{O}$. 
\end{defn} 

Let $(\mca{O}, \mu, \ep)$ be a nonsymmetric dg operad with a multiplication $\mu$ and a unit $\ep$. 
Then $\mca{O}$ has a structure of a cosimplicial chain complex with operations 
\[
\delta_{k,i}: \mca{O}(k-1)_* \to \mca{O}(k)_* \,(0 \le i \le k), \qquad 
\sigma_{k,i}: \mca{O}(k+1)_* \to \mca{O}(k)_* \, (0 \le i \le k)
\]
defined by 
\[
\delta_{k,i} (x):= \begin{cases} \mu \circ_2 x &(i=0) \\  x \circ_i \mu &(1 \le i \le k-1) \\ \mu \circ_1 x &(i=k), \end{cases}  \qquad
\sigma_{k,i} (x):= x \circ_{i+1} \ep. 
\]
Then we obtain the total chain complex $(\tilde{\mca{O}}, \tilde{\partial})$ which is defined by 
\[ 
\tilde{\mca{O}}_*:= \prod_{k=0}^\infty \mca{O}(k)_{*+k},  \qquad 
(\tilde{\partial} x )_k:= \begin{cases} \partial x_0 &(k=0) \\ \partial x_k + (-1)^{|x|} \sum_{i=0}^k (-1)^i \delta_{k,i}(x_{k-1}) &(k \ge 1) \end{cases}
\] 
and a subcomplex $\tilde{\mca{O}}^\nm \subset \tilde{\mca{O}}$ 
such that the inclusion map is a quasi-isomorphism. 
Notice that the chain complex $\tilde{\mca{O}}$ can be defined even when the unit $\ep$ does not exist. 

\begin{ex}
For any dga algebra $A$, 
the endomorphism operad $\endo(A):= (\Hom(A^{\otimes k}, A))_{k \ge 0}$ 
has a multiplication $\mu \in \Hom_0(A^{\otimes 2}, A)$ and a unit $\ep \in \Hom_0(\bar{\R}, A)$, defined by 
$\mu(a_1 \otimes a_2):= a_1 a_2$ and 
$\ep(1):= 1_A$. 
In particular, $\endo(A)$ has the structure of a cosimplicial chain complex. 
The total complex $\widetilde{\endo(A)}$ is isomorphic to the Hochschild cochain complex $C^*(A,A)$. 
\end{ex} 

Now let us state a version of Deligne's conjecture for dg operads with multiplications. 
A dg operad $\mca{P}$ is called a \textit{chain model} of the little disks operad $\mca{D}$, 
if there exists a zig-zag of quasi-isomorphisms (of dg operads) 
connecting $\mca{P}$ and $C_*(\mca{D})$ (the dg operad consisting of singular chains of $\mca{D}$). 

\begin{thm}\label{160420_1} 
Let $\mca{O} = (\mca{O}(k))_{k \ge 0}$ be a nonsymmetric dg operad with a multiplication $\mu \in \mca{O}(2)_0$. 
\begin{enumerate} 
\item[(i):] 
The chain complex $\tilde{\mca{O}}$ has a dga algebra structure with a product $\bullet$ defined by 
\[
(x \bullet y)_k := \sum_{l+m= k} (-1)^{l |y|} (\mu \circ_1 x_l) \circ_{l+1} y_m. 
\]
\item[(ii):]
$\tilde{\mca{O}}$ has a dg pre-Lie algebra structure (with grading shifted by $1$) with a pre-Lie product $*$ 
and a Lie bracket $\{\, ,\,\}$ which are defined by 
\begin{align*} 
(x * y)_k &:= \sum_{\substack{l+m= k+1 \\ 1 \le i \le l}}  (-1)^{(i-1)(m-1) + (l-1)(|y|+m)} x_l \circ_i y_m, \\
\{x, y \}&:= x * y - (-1)^{(|x|-1)(|y|-1)} y*x. 
\end{align*} 
\item[(iii):] 
The operations $\bullet$ and $\{\, , \,\}$ define a Gerstenhaber algebra structure on $H_*(\tilde{\mca{O}})$. 
\item[(iv):] 
There exists a chain model $\mca{P}$ of the little disks operad $\mca{D}$ 
with an isomorphism $H_*(\mca{P}) \cong \ger$, 
such that $\tilde{\mca{O}}$ has a dg $\mca{P}$-algebra structure
which lifts the Gerstenhaber algebra structure on $H_*(\tilde{\mca{O}})$. 
\item[(v):] 
If $\mca{O}$ has a unit $\ep \in \mca{O}(0)$ of the product $\mu$, 
the action of $\mca{P}$ on $\tilde{\mca{O}}$ restricts to $\tilde{\mca{O}}^\nm$. 
\item[(vi):] 
Let $\mca{O}^1$, $\mca{O}^2$ be nonsymmetric dg operads with multiplications, 
and $\mca{O}^1 \to \mca{O}^2$ be a morphism of dg operads preserving multiplications. 
Then the induced chain map $\tilde{\mca{O}_1} \to \tilde{\mca{O}_2}$ is a map of dg $\mca{P}$-algebras. 
\end{enumerate} 
\end{thm} 
\begin{proof} 
These statements follow from Theorem A in \cite{Ward_14}, 
where the Maurer-Cartan element $\zeta = (\zeta_k)_{k \ge 2}$
is given by $\zeta_2= -\mu$ and $\zeta_k = 0 \,(k \ne 2)$.

More specifically, 
statements (i)--(iii) follow from Lemma 2.32 in \cite{Ward_14}. 
The signs in the formulas in (i), (ii) will be discussed in Section 2.5.4. 
Also notice that in (i) the product $\bullet$ is strictly associative
(thus defines a dga algebra structure on $\tilde{\mca{O}}$), 
since we assume $\zeta_k=0$ for every $k \ne 2$. 

The statement (iv) follows from Theorem 2.33 in \cite{Ward_14}, 
where the dg operad $\mca{P}$ in our statement is the 
``minimal operad'' defined in \cite{Ward_14} Section 2.4.  
Although \cite{Ward_14} assumes that $\mca{O}(0)=0$ (see the first part of \cite{Ward_14} Section 2), 
the action of the minimal operad (see \cite{Ward_14} Section 2.5) 
does not involve this assumption.  
The statements (v) and (vi) are straightforward from the definition of this operad action. 
\end{proof} 

\subsubsection{Operadic Deligne's conjecture: cyclic version} 

Let us recall the notion of cyclic dg operads. 

\begin{defn}\label{150723_1} 
Let $\mca{O}  = (\mca{O}(k))_{k \ge 0}$ be a nonsymmetric dg operad. 
A \textit{cyclic structure} on $\mca{O}$ is a sequence $(\tau_k)_{k \ge 0}$ with the following properties. 
\begin{itemize} 
\item For any $k \ge 0$, $\tau_k$ is a chain map on $\mca{O}(k)_*$ of degree $0$, satisfying $\tau_k^{k+1} = \id_{\mca{O}(k)}$. 
\item $1_\mca{O} \in \mca{O}(1)_0$ is cyclically invariant, i.e. $\tau_1(1_{\mca{O}}) = 1_{\mca{O}}$. 
\item For any $1 \le i \le k$, $l \ge 0$, $x \in \mca{O}(k)_*$ and $y \in \mca{O}(l)_*$, there holds
\[
\tau_{k+l-1}(x \circ_i y) = \begin{cases} \tau_k x \circ_{i-1} y &(i \ge 2) \\ (-1)^{|x||y|} \tau_l y \circ_l \tau_k x &(i=1, l \ge 1) \\ \tau_k^2x \circ_k y &(i=1, l=0). \end{cases} 
\]
A pair $(\mca{O}, (\tau_k)_{k \ge 0})$ is called a \textit{nonsymmetric cyclic dg operad}. 
\end{itemize} 
\end{defn} 

It is easy to check that if a nonsymmetric dg operad $\mca{O}$ has a cyclic structure, 
then it has a structure of a cocylic chain complex. 
Now let us state a version of Deligne's conjecture for \textit{cyclic} dg operads with multiplications. 

\begin{thm}\label{160420_3}
Let $\mca{O} = (\mca{O}(k))_{k \ge 0}$ be a nonsymmetric dg operad 
with a cyclic structure $(\tau_k)_{k \ge 0}$, 
a multiplication $\mu$ satisfying $\tau_2(\mu)=\mu$, 
and a unit $\ep$. 
\begin{enumerate} 
\item[(i):] 
$\tilde{\mca{O}}^\nm$ admits an anti-chain map $\Delta$ of degree $1$, defined by 
\[ 
(\Delta x)_k =  \sum_{i=1}^{k+1} (-1)^{|x|+k(i-1)+1}  (\tau^i_{k+1} x_{k+1}) \circ_{k+2-i} \ep. 
\]
\item[(ii):]
The operations 
$\Delta$ and $\bullet$ 
define a BV algebra structure on $H_*(\tilde{\mca{O}}^\nm) \cong H_*(\tilde{\mca{O}})$. 
\item[(iii):] 
There exists a chain model $f\mca{P}$ of the framed little disks operad 
and an isomorphism $H_*(f\mca{P}) \cong \batvil$, such that the following statements hold: 
\begin{itemize}
\item There exists an inclusion of dg operads $\mca{P} \to f\mca{P}$ such that the following diagram commutes: 
\[ 
\xymatrix{ 
H_*(\mca{P})  \ar[r] \ar[d]_-{\cong} & H_*(f\mca{P})  \ar[d]^-{\cong} \\ 
\ger  \ar[r] & \batvil. 
}
\] 
\item For any $\mca{O}$ which satisfies the assumption in this theorem, 
$\widetilde{\mca{O}}^\nm$ has a dg $f\mca{P}$-algebra structure which lifts the BV algebra structure 
on $H_*(\tilde{\mca{O}}^\nm) \cong H_*(\tilde{\mca{O}})$. 
\end{itemize} 
\end{enumerate} 
\end{thm} 
\begin{proof} 
The result follows from Theorem B (or Theorem 4.6) in \cite{Ward_14}. 
The operad $f\mca{P}$ in our statement is the operad $\mca{TS}_\infty$, which is introduced in \cite{Ward_12} Section 4. 
The inclusion of operads $\mca{P} \to f\mca{P}$ follows from Lemma 5.9 in \cite{Ward_12}. 
The sign for the operator $\Delta$ in the statement (i) will be discussed in Section 2.5.4.
\end{proof}

\subsubsection{Sign conventions} 

Here we briefly review sign conventions in \cite{Ward_14}, and explain where the signs in Theorems \ref{160420_1} and \ref{160420_3} come from. 
The author appreciates Benjamin Ward for patiently replying to several questions on sign conventions in \cite{Ward_14}. 
Nevertheless, only the author is responsible for the accuracy and correctness of the following explanations. 

First notice that we use homological grading convention (boundary operators \textit{decrease} grading by $1$), 
whereas \cite{Ward_14} uses cohomological grading convention (boundary operators \textit{increase} grading by $1$). 
In the following we review sign conventions in \cite{Ward_14} with all gradings reversed. 

For any graded vector space $V$, we write $\Sigma V$(resp. $\Sigma^{-1}V$) 
for the graded vector space shifted \textit{down} (resp. \textit{up}) 
$1$ degree from that of $V$. 
Namely, $(\Sigma V)_* = V_{*+1}$ and $(\Sigma^{-1}V)_* = V_{*-1}$. 
In particular, 
\[ 
(\Sigma \bar{\R})_* = \begin{cases} \R  [1]_{-1} &(*=-1) \\  0 &(* \ne -1)\end{cases} \qquad
(\Sigma^{-1} \bar{\R})_* = \begin{cases} \R [1]_1 &(*=1) \\ 0 &(* \ne 1). \end{cases} 
\] 
Notice that there are natural isomorphisms 
$\Sigma V \cong V \otimes \Sigma \bar{\R}$ and 
$\Sigma^{-1} V \cong V \otimes \Sigma^{-1} \bar{\R}$. 

It is obvious that 
\[ 
\endo (\Sigma^{-1}  \bar{\R}) (k)_* = \begin{cases} \R \cdot \iota_k  &(*=1-k) \\ 0 &(* \ne 1-k), \end{cases} \qquad 
\endo (\Sigma \bar{\R}) (k)_* = \begin{cases} \R \bar{\iota}_k &(*=k-1) \\ 0 &(* \ne k-1), \end{cases} 
\] 
where $\iota_k$ and $\bar{\iota}_k$ are defined by 
$\iota_k ([1]_1^{\otimes k}) = [1]_1$ and 
$\bar{\iota}_k ([1]_{-1}^{\otimes k}) = [1]_{-1}$. 
It is easy to check that 
\[
\iota_k \circ_i \iota_l = (-1)^{(i-1)(l-1)} \iota_{k+l-1} \qquad (1\le i \le k, \, l \ge 0). 
\] 
For any dg operad $\mca{O}$, we set 
$\shift \mca{O} := \mca{O} \otimes \endo(\Sigma^{-1} \bar{\R})$ and 
$\shift^{-1} \mca{O} := \mca{O} \otimes \endo(\Sigma \bar{\R})$.  
For every $k \ge 0$, there exists an isomorphism 
\begin{equation}\label{160503_2} 
\mca{O}(k)_* \cong \shift \mca{O}(k)_{*+1-k} ; \qquad x \mapsto \hat{x}:= x \otimes \iota_k.
\end{equation} 
For $x \in \mca{O}(k)$, $y \in \mca{O}(l)$ and $1 \le i \le k$, we obtain (see Remark 2.4 in \cite{Ward_14}): 
\begin{align*} 
\hat{x} \circ_i \hat{y} &= (x \otimes \iota_k) \circ_i (y \otimes \iota_l) = (-1)^{(k-1)|y|} (x \circ_i y) \otimes (\iota_k \circ_i \iota_l) \\
&=(-1)^{(i-1)(l-1) + (k-1)|y|} \cdot \widehat{x \circ_i y}. 
\end{align*} 

\textbf{pre-Lie product} 

As explained in \cite{Ward_14} Sections 2.1.1 and 2.1.3, 
the chain complex $\Sigma \prod_{k=0}^\infty \shift \mca{O}(k)$ 
has a structure of a dg pre-Lie algebra with grading shifted by $1$ (or odd dg pre-Lie algebra). 
On the other hand, 
$\tilde{\mca{O}}_* = \prod_{k=0}^\infty \mca{O}(k)_{*+k}$
is naturally identified with 
$\Sigma \prod_{k=0}^\infty \shift \mca{O}(k)$ 
by applying the isomorphism (\ref{160503_2}) to all $k \ge 0$. 
With this identification, 
$\tilde{\mca{O}}$ has an odd dg pre-Lie algebra structure, 
where the boundary operator is defined by $(\partial x)_k:= \partial x_k \, (\forall k \ge 0)$, 
and the pre-Lie product $*$
(notice that we use different notation from \cite{Ward_14}, where the pre-Lie product is denoted by $\circ$)
 is defined by: 
\begin{equation}\label{160503_3} 
(x * y)_k = \sum_{\substack{l+m= k+1 \\ 1 \le i \le l}}  (-1)^{(i-1)(m-1) + (l-1)(|y|+m)} x_l \circ_i y_m.
\end{equation} 
This formula coincides with the formula (2.1) in \cite{Ward_14}. 
Notice that the degree of $y_m \in \mca{O}(m)$ is $|y|+m$, since $\tilde{\mca{O}}_* = \prod_{m=0}^\infty \mca{O}(m)_{*+m}$. 
Moreover, the bracket $\{ \, , \,\}$ is defined by 
\[ 
\{x, y\} := x  * y - (-1)^{(|x|-1)(|y|-1)} y*x, 
\] 
see \cite{Ward_14} Section 2.1.1. 
These computations explain the signs  in Theorem \ref{160420_1} (ii). 

\textbf{Boundary operator $\tilde{\partial}$} 

Suppose that $\mca{O}$ has a multiplication $\mu \in \mca{O}(2)_0$. 
Then $\zeta = (\zeta_k)_{k \ge 2} \in \tilde{\mca{O}}_{-2}$ defined by $\zeta_k = \begin{cases} -\mu &(k=2) \\ 0 &(k \ne 2) \end{cases}$ 
is a Maurer-Cartan element of $\mca{O}$, 
i.e. it satisfies the equation $\partial \zeta + \zeta*\zeta=0$
(see Section 2.2 in \cite{Ward_14}). 
Then, $\Sigma \prod_{k=0}^\infty \shift \mca{O}(k)$ admits a boundary operator 
$\partial_\zeta$ which is defined by the following formula (see Lemma 2.20 in \cite{Ward_14}): 
\[
\partial_\zeta x := \partial x + \{ \zeta, x\} = \partial x + \zeta * x  - (-1)^{|x|+1} x * \zeta.
\]
A short computation using (\ref{160503_3}) implies 
\[ 
(\partial_\zeta x)_k = \partial x_k  + (-1)^{|x|} \biggl( \mu \circ_2 x_{k-1}  + \sum_{i=1}^{k-1}  (-1)^i x_{k-1} \circ_i \mu + (-1)^k  \mu \circ_1 x_{k-1} \biggr). 
\]
Hence the boundary operator $\partial_\zeta$ on $\Sigma \prod_{k=0}^\infty \shift \mca{O}(k)$ corresponds to the
boundary operator $\tilde{\partial}$ on $\tilde{\mca{O}}_* = \prod_{k=0}^\infty \mca{O}(k)_{*+k}$. 

\textbf{Product $\bullet$} 

\cite{Ward_14} defines a product $\cup$ on $\Sigma \prod_{k=0}^\infty \shift \mca{O}(k)$ by 
$x \cup y:= B_2(\zeta; x, y)$ (see Definition 2.31 in \cite{Ward_14}), 
where $B_2 \in \shift^{-1} B^0(3)$ is defined by 
$B_2:= B^0_2 \otimes \bar{\iota}_3$ (see Definition 2.13 in \cite{Ward_14}), 
and $B^0_2$ is a ``brace operation'' on $\prod_{k=0}^\infty \shift \mca{O}(k)$ (see \cite{Ward_14} Section 2.1.2). 
More explicitly 
\[
B^0_2(a; b,c) = \sum_{1 \le i<j \le k} (a \circ_i b) \circ_{j+l-1} c \qquad (a \in \shift \mca{O}(k), \, b \in \shift \mca{O}(l), \, c \in \shift \mca{O}(m))
\]
where $\circ$ denotes the composition in $\shift \mca{O}$. 
Notice that the $\cup$ product is strictly associative in our case, since $\zeta_k= 0$ for $k \ne 2$. 

Applying the isomorphism $\Sigma V \cong V \otimes \Sigma \bar{\R}$ for $V = \prod_{k=0}^\infty \shift \mca{O}(k)$, 
let us denote $x= \Sigma^{-1} x \otimes [1]_{-1}$, $y = \Sigma^{-1} y \otimes [1]_{-1}$, and $\zeta = \Sigma^{-1} \zeta \otimes [1]_{-1}$. 
Then we obtain 
\begin{align*} 
B_2(\zeta; x, y) &= (B^0_2 \otimes \bar{\iota}_3) ( \Sigma^{-1} \zeta \otimes [1]_{-1} ; \Sigma^{-1} x \otimes [1]_{-1} , \Sigma^{-1} y \otimes [1]_{-1} ) \\
&= (-1)^{|x|+1} B^0_2(\Sigma^{-1}\zeta ; \Sigma^{-1} x, \Sigma^{-1} y) \otimes \bar{\iota}_3 ([1]_{-1}^{\otimes 3}) \\
&= (-1)^{|x|+1} B^0_2(\Sigma^{-1}\zeta ; \Sigma^{-1} x, \Sigma^{-1} y) \otimes [1]_{-1}. 
\end{align*} 
The second equality follows from the Koszul sign rule with 
$|\Sigma^{-1} \zeta|=-1$, 
$|\Sigma^{-1} x|=|x|+1$, 
$|\Sigma^{-1} y|=|y|+1$ and
$|\bar{\iota}_3|=2$. 
Recalling that $\zeta_k = \begin{cases} -\mu &(k=2) \\ 0 &(k \ne 2) \end{cases}$
and using (\ref{160503_3}), we obtain 
\begin{align*} 
(x \cup  y)_k & = (-1)^{|x|+1} \sum_{l+m=k} (-1)^{0 \cdot (l-1) + 1 \cdot (|x|+l) + l \cdot (m-1) + l \cdot (|y|+m)}  (-\mu \circ_1 x_l) \circ_{l+1}  y_m \\
&= (-1)^{l |y|} \sum_{l+m=k} (\mu \circ_1 x_l) \circ_{l+1} y_m. 
\end{align*} 
Therefore, under the isomorphism $\Sigma \prod_{k=0}^\infty \shift \mca{O}(k) \cong \tilde{\mca{O}}$, 
the product $\cup$ corresponds to the product $\bullet$ 
in Theorem \ref{160420_1} (i). 

\textbf{Operator $\Delta$} 

When $\mca{O}$ has a cyclic structure and 
a unit $\ep \in \mca{O}(0)_0$, 
one can define an operator $\Delta$ of degree $1$ 
on $\Sigma \prod_{k=0}^\infty \shift \mca{O}(k)$ by the following formula 
(see the last paragraph in the proof of Lemma 4.5 in \cite{Ward_14}): 
\begin{equation}\label{160508_1} 
(\Delta x)_k  := (Ns_0) (x_{k+1}) = 
\biggl( \sum_{i=0}^k  t_k^i \biggr) (t_k s_1 t_{k+1}^{-1}) (x_{k+1}).
\end{equation} 
The operator $s_0$ is defined in the formula (3.1), 
and the operator $N$ is defined right after Corollary 3.3 in \cite{Ward_14}. 
Chain maps $t_k: \shift \mca{O}(k)_* \to \shift \mca{O}(k)_*$ and 
$t_{k+1}: \shift \mca{O}(k+1)_* \to \shift \mca{O}(k+1)_*$ are defined by 
\[ 
t_k (a \otimes \iota_k):= (-1)^k \tau_k a \otimes \iota_k, \qquad
t_{k+1} (a \otimes \iota_{k+1}) = (-1)^{k+1} \tau_{k+1} a \otimes \iota_{k+1}. 
\]
On the other hand, $s_1: \shift \mca{O}(k+1)_* \to \shift \mca{O}(k)_{*+1}$ is an anti-chain map which is defined by 
\[ 
s_1 (a \otimes \iota_{k+1}):= (-1)^{|a|+k} (a \otimes \iota_{k+1}) \circ_1 ( -\ep \otimes \iota_0) 
= (-1)^{|a|+k+1} (a \circ_1 \ep) \otimes \iota_k.
\] 
Notice that a unit for $\zeta$ is $-\ep$, since we set $\zeta_2:=-\mu$. 

Now the operator $\Delta$ in (\ref{160508_1}) 
(which is defined on $\Sigma \prod_{k=0}^\infty \shift \mca{O}(k)$)
corresponds to the following operator $\Delta$ defined on $\tilde{\mca{O}}$: 
\[
(\Delta x)_k = \sum_{i=0}^k (-1)^{|x|+ki+1} \tau_k^{i+1} ((\tau_{k+1}^{-1} x_{k+1}) \circ_1 \ep )= \sum_{i=1}^{k+1} (-1)^{|x|+k(i-1)+1} (\tau_{k+1}^i x_{k+1}) \circ_{k+2-i} \ep. 
\]
The sign $|x|+ki+1$ is a sum of a 
sign $k+1$ from $t_{k+1}^{-1}$, 
a sign $|x_{k+1}|+k+1 \equiv |x|\,(\mathrm{mod}\,2)$, and 
a sign $k(i+1)$ from $t_k^{i+1}$. 
This is exactly the same as the formula in Theorem \ref{160420_3} (i). 

\section{Results and discussions} 

In Section 3.1, we state our main result Theorem \ref{160408_1} and 
a supplementary result Proposition \ref{160420_2} about the length filtration. 
Section 3.2 discusses previous related work, and 
Section 3.3 discusses potential applications to symplectic topology. 
The rest of this paper (Sections 4--8) is devoted to proofs of the results stated in Section 3.1.
The plan of the proofs is explained in Section 3.4. 

\subsection{Summary of results} 

First let us recall some notation; 
for any $C^\infty$-manifold $M$ of dimension $d$, 
we set 
$\L M:= C^\infty(S^1, M)$ 
and 
$\H_*(\L M):= H_{*+d}(\L M: \R)$. 
$\mca{A}_M$ denotes the dga algebra of differential forms on $M$. 
The main result in this paper is a construction of a nonsymmetric cyclic dg operad $\mca{O}_M$ with a multiplication and a unit. 
By Theorem \ref{160420_3}, the associated chain complex $\widetilde{\mca{O}_M}^\nm$ has a dg $f\mca{P}$-algebra structure, 
which turns out to be a chain level refinement of the BV algebra structure on $\H_*(\mca{L}M)$. 
Let us spell out formal statements in Theorem \ref{160408_1} below. 

\begin{thm}\label{160408_1} 
For any closed, oriented $C^\infty$-manifold $M$, 
there exist the following data: 
\begin{enumerate} 
\item[(i):] 
A nonsymmetric cyclic dg operad $\mca{O}_M$ with a multiplication $\mu \in \mca{O}_M(2)_0$ and a unit $\ep \in \mca{O}_M(0)_0$. 
By Theorem \ref{160420_3}, 
$\widetilde{\mca{O}_M}^\nm$ has a dg $f\mca{P}$-algebra structure. 
In particular, 
$H_*(\widetilde{\mca{O}_M}^\nm) \cong H_*(\widetilde{\mca{O}_M})$ has a BV algebra structure. 
\item[(ii):] 
An isomorphism $\Phi: H_*(\widetilde{\mca{O}_M}) \cong \H_*(\mca{L}M)$ of BV algebras, 
where we consider the string topology BV algebra structure on $\H_*(\mca{L} M)$. 
\item[(iii):] 
A morphism $\mca{O}_M \to \endo(\mca{A}_M)$ of nonsymmetric dg operads preserving multiplications, such that 
the induced map on homology $H_*(\widetilde{\mca{O}_M} ) \to H^*(\mca{A}_M, \mca{A}_M)$ 
coincides with the map (\ref{141218_01}): $\H_*(\mca{L}M) \to H^*(\mca{A}_M, \mca{A}_M)$ 
via the isomorphism $H_*(\widetilde{\mca{O}_M}) \cong \H_*(\mca{L} M)$ in (ii). 
\item[(iv):] 
An injective  chain map 
$\iota_M: (\mca{A}_M)_* \to \widetilde{\mca{O}_M}_*$ such that 
\begin{equation}\label{150801_2} 
\iota_M(x) \bullet \iota_M(y) = \iota_M (x \wedge y), \qquad
\{\iota_M(x), \iota_M(y)\} = 0 \qquad(\forall x,y \in \mca{A}_M) 
\end{equation} 
where the operations $\bullet$ and $\{ \, , \,\}$ are defined in Theorem \ref{160420_1}. 
Moreover, the following diagram commutes
(recall that the map $i_M: M \to \mca{L}M$ 
takes each $p \in M$ to the constant loop at $p$): 
\begin{equation}\label{150801_3} 
\xymatrix{
\H_*(M)  \ar[r]^-{\H_*(i_M)}\ar[d]_{\cong}&  \H_*(\mca{L}M) \ar[d]^{\cong} \\
H^{-*}_\dR(M) \ar[r]_{H_*(\iota_M)} & H_*(\widetilde{\mca{O}_M}). 
}
\end{equation} 
\end{enumerate} 
\end{thm} 

\begin{rem}\label{160421_1} 
Several remarks on Theorem \ref{160408_1} are in order. 
\begin{itemize} 
\item[(i):] 
Theorem \ref{160408_1} (ii) implies that 
the dg $f\mca{P}$-algebra structure on $\widetilde{\mca{O}_M}^\nm$ is a chain level refinement of the string topology BV algebra structure on $\H_*(\mca{L}M)$. 
\item[(ii):] 
The morphism of operads $\mca{O}_M \to \endo(\mca{A}_M)$ induces a chain map 
$\widetilde{\mca{O}_M}_* \to C^*(\mca{A}_M, \mca{A}_M)$. 
By Theorem \ref{160420_1} (vi), 
this chain map is a morphism of dg $\mca{P}$-algebras. 
Since the induced map on homology coincides with the map 
(\ref{141218_01}): $\H_*(\mca{L} M) \to H^*(\mca{A}_M, \mca{A}_M)$, 
we recover the fact that the map (\ref{141218_01}) preserves the Gerstenhaber structures
(see Remark \ref{150205_2}). 
\item[(iii):] 
By Theorem \ref{160420_1} the chain complex $\widetilde{\mca{O}_M}$ has a dga algebra structure with a product $\bullet$, 
and the induced product on $H_*(\widetilde{\mca{O}_M})$ corresponds to the loop product on $\H_*(\mca{L}M)$. 
The homotopy transfer theorem (see \cite{LodVal_12} Section 10.3) shows that 
$\H_*(\mca{L}M)$ has an $A_\infty$ algebra structure $(\mu_k)_{k \ge 1}$ such that 
$\mu_1=0$, $\mu_2=\bullet$ and $(\H_*(\mca{L}M), (\mu_k)_{k \ge 1})$ is homotopy equivalent to the dga algebra $(\widetilde{\mca{O}_M}, \bullet)$. 
Moreover, Theorem \ref{160408_1}  (iv) implies that one may choose $(\mu_k)_{k \ge 1}$ so that 
$\mu_k(\H_*(M)^{\otimes k}) \subset \H_*(M)$ for every $k \ge 1$, 
and $(\H_*(M), (\mu_k)_{k \ge 1})$ is homotopy equivalent to the dga algebra $(\mca{A}_M, d, \wedge)$. 
In particular, $(\mu_k)_{k \ge 1}$ recovers the classical Massey products on $\H_*(M)$. 
\item[(iv):] 
By the same arguments, one can define an $L_\infty$ algebra structure $(l_k)_{k \ge 1}$ on $\H_*(\mca{L}M)$ such that 
$l_1=0$, $l_2 = \{ \, , \,\}$ and $(\H_*(\mca{L} M), (l_k)_{k \ge 1})$ is homotopy equivalent to the dg Lie algebra $(\widetilde{\mca{O}_M}, \{ \, , \,\})$. 
Moreover one may choose $(l_k)_{k \ge 1}$ so that $l_k = 0$ on $\H_*(M)^{\otimes k} \subset \H_*(\mca{L}M)^{\otimes k}$. 
\end{itemize} 
\end{rem} 

Remark \ref{160421_1} (i) says that $\widetilde{\mca{O}_M}^\nm$ is a chain model of $\mca{L}M$ which is suitable to define string topology operations on it. 
Compared to the Hochschild cochain complex $C^*(\mca{A}_M, \mca{A}_M)$, which is often used as a chain model of $\mca{L}M$, our chain model $\widetilde{\mca{O}_M}^\nm$ has the following features:
\begin{itemize}
\item The isomorphism $H_*(\widetilde{\mca{O}_M}^\nm) \cong \H_*(\mca{L}M)$ holds for an arbitrary closed oriented $C^\infty$-manifold $M$, 
whereas in most of the literature 
the map (\ref{141218_01}): $\H_*(\mca{L}M) \to H^*(\mca{A}_M, \mca{A}_M)$ is proved to be an isomorphism 
under the assumption that $M$ is simply-connected (see Remark \ref{160617_1}). 
\item It seems difficult to define an action of a chain model of $f\mca{D}$ on $C^*(\mca{A}_M, \mca{A}_M)$, since $\mca{A}_M$ is infinite-dimensional and the map 
$\mca{A}_M \to \mca{A}^\vee_M[d]$ (see Section 2.4) is not an isomorphism, 
whereas in most of the literature (e.g. \cite{Kaufmann_08}) 
this condition is required to prove a cyclic version of Deligne's conjecture. 
\end{itemize} 

Another  remarkable feature of our chain model $\widetilde{\mca{O}_M}^\nm$ is that, 
when $M$ has a Riemannian metric it is equipped with a \textit{length filtration} which is compatible with string topology operations. 
On the homology level, relations between string topology operations and the length filtration are studied in \cite{GoHi_09}. 
The length filtration also plays a role in comparison of 
the loop space homology and the Floer homology of cotangent bundles, see \cite{CiLa_09} Section 7. 

To state this property let us introduce some notation. 
For any $\gamma \in \L M$, let $\len(\gamma):= \int_{S^1} |\dot{\gamma}|$. 
For any $a \in (0, \infty]$, we define 
$\mca{L}^a M:= \{\gamma \in \mca{L}M \mid \len(\gamma) <a\}$. 
In particular, $\mca{L}^\infty M = \mca{L}M$. 
For an arbitrary chain complex $C$, 
a filtration (indexed by $(0, \infty]$) on $C$ is a family 
$(F^aC)_{a \in (0,\infty]}$ of subcomplexes of $C$ 
such that $a \le b \implies F^a C \subset F^bC$ and $F^\infty C= C$. 
For any $x \in C$, we set $|x|:= \inf\{a \mid x \in F^aC \}$. 

\begin{prop}\label{160420_2} 
Let $M$ be a closed, oriented $C^\infty$-manifold with a Riemannian metric. 
Then, for every $k \ge 0$, the chain complex $\mca{O}_M(k)$ is equipped with a filtration 
$(F^a \mca{O}_M(k))_{a \in (0,\infty]}$ 
such that the following properties hold. 
\begin{enumerate} 
\item[(i):] 
There holds 
\begin{align*} 
|x \circ_i y| &\le |x| + |y| \qquad(\forall x \in \mca{O}_M(k), \, \forall y \in \mca{O}_M(l), \, 1 \le \forall  i \le k), \\ 
|\tau_k x| &= |x| \qquad\qquad (\forall x \in \mca{O}_M(k)), \\ 
|\mu|&= |\ep| = 0. 
\end{align*} 
Hence $(F^a \mca{O}_M(k))_{k \ge 0}$ is a cocylic chain complex for every $a \in (0, \infty]$. 
In particular, $\widetilde{\mca{O}_M}$ and $\widetilde{\mca{O}_M}^\nm$ have filtrations defined by 
\[ 
F^a \widetilde{\mca{O}_M}_* = \prod_{k=0}^\infty F^a \mca{O}_M(k)_{*+k}, \qquad 
F^a \widetilde{\mca{O}_M}^\nm = F^a \widetilde{\mca{O}_M} \cap \widetilde{\mca{O}_M}^\nm, 
\] 
and the inclusion $F^a \widetilde{\mca{O}_M}^\nm \to F^a \widetilde{\mca{O}_M}$ is a quasi-isomorphism. 
\item[(ii):]
The filtration on $\widetilde{\mca{O}_M}^\nm$ is compatible with the $f\mca{P}$-algebra structure. 
Namely, 
\[
|x \cdot (y_1 \otimes \cdots \otimes y_r)| \le |y_1| + \cdots + |y_r|
\]
for any $x \in f\mca{P}(r)$ and $y_1,\ldots, y_r \in \widetilde{\mca{O}_M}^\nm$. 
\item[(iii):] 
There exists an isomorphism 
$\H_*(\L^a M) \cong H_*(F^a \widetilde{\mca{O}_M})$ for every $a \in (0, \infty]$, 
and this family of isomorphisms is compatible with inclusions. 
Namely, the following diagram commutes for every $0 < a \le b \le \infty$ 
where the vertical maps are induced by inclusions: 
\[ 
\xymatrix{\H_* (\L^a M) \ar[r]^-{\cong} \ar[d] & H_*(F^a \widetilde{\mca{O}_M})  \ar[d]  \\ \H_*(\L^b M)  \ar[r]^-{\cong}& H_*(F^b \widetilde{\mca{O}_M}). }
\] 
When $a=\infty$, the isomorphism $\H_*(\L M) \cong H_*(\widetilde{\mca{O}_M})$ coincides with the isomorphism $\Phi$ in Theorem \ref{160408_1} (ii). 
\end{enumerate} 
\end{prop} 
\begin{proof} 
For the definition of the filtration $F^a \mca{O}_M$ and verification of  (i), see Remark \ref{160626_1}. 
To verify (ii), it is sufficient to check that the generators of the operad $\mca{TS}_\infty$ (see \cite{Ward_12} Section 4.1) 
preserve the length filtration, and this is straightforward from (i). 
The isomorphism $\H_*(\L^a M) \cong H_*(F^a \widetilde{\mca{O}_M})$ in (iii) will be defined in Section 8.2. 
\end{proof}


\subsection{Previous work} 

Rich algebraic structures in chain level string topology were outlined in \cite{Sul_07} by D. Sullivan, 
and there have been several papers working out details. 

X. Chen \cite{XChen_11} introduced a chain model of the free loop space using Whitney differential forms, and 
defined several string topology operations (the loop product, loop bracket and rotation) on that chain model, recovering the BV algebra structure at the homology level
(\cite{XChen_11} also studied the $S^1$-equivariant case). 
It is not clear whether these operations extend to actions of a dg operad on this chain model. 

On the other hand, a recent paper \cite{DCPR_15} by G. Drummond-Cole, K. Poirier and N. Rounds
proposed a more geometric approach using short geodesic segments and diffuse intersection classes. 
\cite{DCPR_15} defines operations on the singular chain complex of the free loop space, recovering the homology level structure defined by Cohen-Godin \cite{CohGod_04}.
In particular, \cite{DCPR_15} covers operations with multiple outputs and those corresponding to surfaces of higher genus. 
However, the operations in \cite{DCPR_15} are associative only up to homotopy, and it seems that the resulting algebraic structure is yet to be fully worked out. 

\subsection{Potential applications to symplectic topology} 

Let us discuss some potential applications of results in this paper to symplectic topology. 
For any $C^\infty$-manifold $M$, the cotangent bundle $T^*M$ has a natural symplectic structure. 
When $M$ is closed, oriented and spin, the Floer homology $\HF_*(T^*M)$ has a BV algebra structure, and there exists an isomorphism
of BV algebras $\HF_*(T^*M) \cong \H_*(\mca{L}M)$ (see \cite{Abouzaid_13} and the references therein). 
There should be chain level refinements of this correspondence, and we expect that our chain level structures in string topology fit into this picture. 
More specifically, we expect that one can define an $A_\infty$ (resp. $L_\infty$) structure on $\HF_*(T^*M)$ via counting solutions of appropriate Floer equations 
which is homotopy equivalent to the $A_\infty$ (resp. $L_\infty$) refinement of the loop product (resp. bracket) on $\H_*(\mca{L}M)$ defined in 
Remark \ref{160421_1} (iii) and (iv). 

On the other hand, Fukaya \cite{Fuk_06} used a chain level loop bracket for compactifications of the moduli space of pseudo-holomorphic disks with Lagrangian boundary conditions, 
and obtained restrictions on topological types of Lagrangian submanifolds. 
We expect that our definition of a chain level loop bracket could be used to work out details of this approach. 

Finally, there is a very interesting program by Cieliebak-Latschev \cite{CiLa_09}, which compares the 
symplectic field theory of sphere cotangent bundles and the string topology of $S^1$-equivariant chains on free loops modulo constant loops. 
We hope that the construction in this paper will be a first step towards working out details of the string topology side of this program.

\subsection{Plan of proofs}

The rest of this paper is devoted to proofs of results presented in Section 3.1. 
The main step is to define the nonsymmetric cyclic dg operad $\mca{O}_M$ which appears in Theorem \ref{160408_1}. 
Roughly speaking, for every integer $k \ge 0$, the chain complex $\mca{O}_M(k)$ consists of ``chains'' of (Moore) loops on $M$ with $k+1$-marked points, 
and operad compositions  are defined by taking fiber products at marked points. 
The idea of using loops with marked points is partially inspired by the theory of iterated integrals. 

A main difficulty is that we need to define fiber products at the chain level, and the usual singular chains are not appropriate for this purpose. 
To avoid this trouble, we introduce a notion of \textit{de Rham chains}, 
which is a certain hybrid of the notions of singular chains and differentiable forms. 
We also introduce a notion of \textit{differentiable spaces}, on which de Rham chains are defined. 
A $C^\infty$-manifold and its free loop space have natural structures of differentiable spaces. 
In Section 4, we introduce these notions and define a de Rham chain complex (the chain complex which consists of de Rham chains) for any differentiable space. 

In Section 5, we show that the homology of the de Rham chain complex of a $C^\infty$-manifold is naturally isomorphic to the usual singular homology. 
In Section 6, we prove the same result for the free loop space of a closed $C^\infty$-manifold. 

In Section 7, we define differentiable spaces which consist of Moore loops on a $C^\infty$-manifold $M$ with marked points, 
and show that the collection of de Rham chain complexes of these spaces has a natural structure of a nonsymmetric cyclic dg operad with a multiplication and a unit. 
This operad is the operad $\mca{O}_M$. 
Finally in Section 8, we prove the results presented in Section 3.1.
Most of these results follow naturally from the definition of $\mca{O}_M$, 
nevertheless we need some complicated arguments to give complete proofs.

\section{Differentiable spaces and de Rham chains} 

In this section, we introduce the notions of \textit{differentiable spaces} and \textit{de Rham chain complexes}, which are basic for the arguments in this paper. 
The notion of differentiable spaces is introduced in Section 4.2, and de Rham chain complexes for these spaces are defined in Section 4.3. 
The rest of this section (4.4--4.8) is devoted to establishing several basic results about de Rham chain complexes of differentiable spaces. 

\subsection{Integration along fibers} 
 
Let $P$ be a $d$-dimensional $C^\infty$-manifold. 
Throughout this paper, all manifolds are without boundary, unless otherwise specified. 
Let $k \ge 0$ be an integer, and $E \to P$ be an $\R^k$-bundle. 
We define $\det E:= \wedge^k E$. 
When $k=0$, we define $\det E$ to be the trivial $\R$-bundle on $P$. 
Any exact sequence 
$0 \to E_0 \to E_1 \to E_2 \to 0$ 
of real vector bundles induces an isomorphism 
$\det E_1 \cong \det E_2 \otimes \det E_0$. 
An orientation of $E$ is a section of the double cover $(\det E)/\R_{>0} \to P$. 
An orientation of $TP \to P$ is called an orientation of $P$. 

Recall that we denote $\mca{A}^j(P):= C^\infty(\wedge^j T^*P)$ when $j = 0, \ldots, d$. 
Let $\mca{A}^j_c(P)$ denote the subspace of $\mca{A}^j(P)$ 
which consists of compactly supported $j$-forms. 
When $j \notin \{0, \ldots, d\}$, we set $\mca{A}^j_c(P) = \mca{A}^j(P)=0$. 
When $P$ is oriented, we define 
\[
\int_P: \mca{A}^*_c(P) \to \R; \quad \omega \mapsto \begin{cases} \int_P \omega &(*=d) \\ 0 &(* \ne d). \end{cases}
\]

Let $P_1$, $P_0$ be oriented $C^\infty$-manifolds, and 
$\pi: P_1 \to P_0$ be a submersion (i.e. $\pi$ is of class $C^\infty$ and $d\pi_p: T_pP_1 \to T_{\pi(p)}P_0$ is surjective for any $p \in P_1$). 
Let $d:= \dim P_1 - \dim P_0$. 
The \textit{integration along fibers} is a chain map $\pi_!: \mca{A}^*_c(P_1) \to \mca{A}^{*-d}_c(P_0)$, 
which is defined in the following way. 

First we consider the case $P_1=\R^{n+d}$, $P_0=\R^n$ and $\pi(x_1,\ldots, x_n, y_1,\ldots, y_d)= (x_1,\ldots,x_n)$. 
We assume that $dx_1 \wedge \cdots \wedge dx_n \wedge dy_1 \wedge \cdots \wedge dy_d \in \mca{A}^{n+d}(\R^{n+d})$ and 
$dx_1 \wedge \cdots \wedge dx_n \in \mca{A}^n(\R^n)$ are positive with respect to the orientations on $\R^{n+d}$ and $\R^n$. 

For $\omega(x,y):= u(x,y) dx_{i_1} \cdots dx_{i_k} dy_{j_1} \cdots dy_{j_l} \in \mca{A}^*_c(\R^{n+d})$ where 
$1 \le i_1 < \cdots < i_k \le n$ and $1 \le j_1 < \cdots < j_l \le d$, we define $\pi_! \omega \in \mca{A}^{*-d}_c(\R^n)$ by 
\[
\pi_! \omega(x):=  \begin{cases}
                          0 &(l < d) \\
                        \biggl( \int_{\R^d} u(x,y) \, dy_1 \cdots dy_d \biggr) dx_{i_1} \ldots dx_{i_k} &(l=d).
                         \end{cases}
\]
In the general case, $\pi_!$ is defined by taking local charts and partitions of unity on $P_1$. 
Below is a list of some basic properties of the integration along fibers. 

\begin{itemize}
\item If $P_0$ is a positively oriented point, then $\pi_! \omega = \int_{P_1} \omega$. 
\item $\pi_!$ is a chain map, i.e. there holds $d(\pi_! \omega) = \pi_! (d\omega)$. 
\item For any $\eta \in \mca{A}^*(P_0)$, $\pi_!(\pi^* \eta \wedge \omega) = \eta \wedge \pi_! \omega$. 
\item For any submersion $\pi': P_2 \to P_1$ and $\omega \in \mca{A}^*_c(P_2)$, there holds 
$(\pi \circ \pi')_! \omega = \pi_! (\pi'_! \omega)$. 
\end{itemize} 

\subsection{Differentiable spaces} 

For any integers $n \ge m\ge 0$, 
let $\mca{U}_{n,m}$ denote the set of oriented $m$-dimensional submanifolds in $\R^n$. 
We set $\mca{U}:= \bigsqcup_{n \ge m \ge 0} \mca{U}_{n,m}$. 

Let $X$ be a set. 
A \textit{differentiable structure} on $X$ is a family of maps called \textit{plots}, which satisfies the following conditions:  
\begin{itemize}
\item  Every plot is a map from $U \in \mca{U}$ to $X$. 
\item  If $\ph: U \to X$ is a plot, $U' \in \mca{U}$ and 
 $\theta: U' \to U$ is a submersion, then 
$\ph \circ \theta: U' \to X$ is a plot. 
\end{itemize} 
A \textit{differentiable space} is a pair of a set and a differentiable structure on it. 
For any differentiable space $X$, let 
$\mca{P}(X):= \{(U,\ph) \mid U \in \mca{U}, \text{$\ph:U \to X$ is a plot} \}$. 
A map $f: X \to Y$ between differentiable spaces $X$ and $Y$ is called \textit{smooth}, if there holds
\[
(U,\ph) \in \mca{P}(X) \implies (U, f \circ \ph) \in \mca{P}(Y).
\]

\begin{rem} 
The term ``plot'' is originally used in 
the theory of Chen's differentiable spaces (\cite{KTChen_86}) and 
the theory of diffeological spaces (\cite{Sou_80}, \cite{IgZe_13}). 
Our notion of differentiable spaces is weaker than these notions. 
In particular, in axioms of both of these spaces, it is required that 
all constant maps are plots, while we do not require this condition 
(see Example \ref{150623_1} (i)-(b) below). 
\end{rem} 

\begin{ex}\label{150623_1} 
Let us explain some examples of differentiable structures. 
\begin{enumerate} 
\item[(i):]  
Let $M$ be a $C^\infty$-manifold. 
One can consider the following differentiable structures on $M$: 
\begin{itemize}
\item[(a):] 
$\ph: U \to M$ is a plot if $\ph$ is of class $C^\infty$. 
We denote the resulting differentiable space by $M$. 
\item[(b):] 
$\ph: U \to M$ is a plot if $\ph$ is a submersion (we always assume that any submersion is of class $C^\infty$). 
We denote the resulting differentiable space by $M_\reg$. 
\end{itemize}
The identity map $\id_M: M_\reg \to M$ is smooth, but $\id_M: M \to M_\reg$ is not. 
\item[(ii):] 
$\mca{L}M:= C^\infty(S^1, M)$ has the following differentiable structure: 
a map $\ph: U \to \mca{L} M$ is a plot if 
$U \times S^1 \to M; \, (u,\theta) \mapsto \ph(u)(\theta)$ is of class $C^\infty$. 
\item[(iii):]
Let $X$ be a differentiable space, $Y$ a subset of $X$, 
and $i: Y \to X$ be the inclusion map. 
One can define the following differentiable structure on $Y$: a map 
$\ph: U \to Y$ is a plot if $i \circ \ph: U \to X$ is a plot of $X$. 
\item[(iv):]
Let $(X_s)_{s \in S}$ be a family of differentiable spaces parametrized by the nonempty set $S$. 
The product $X:= \prod_{s \in S} X_s$ has the following differentiable structure: 
$\ph:U \to X$ is a plot if $\pi_s \circ \ph$ is a plot of $X_s$ for every $s \in S$ ($\pi_s$ denotes the projection to $X_s$).
\end{enumerate}
\end{ex}

Smooth maps $f,g: X \to Y$ are called \textit{smoothly homotopic}, 
and denoted as $f \sim g$, 
if there exists a smooth map $h: X  \times \R  \to Y$ such that 
\[
h(x,s) = \begin{cases} 
            f(x) &(s<0), \\
            g(x) &(s>1).
           \end{cases}
\]
$h$ is called a \textit{smooth homotopy} between $f$ and $g$. 
Differentiable spaces $X$ and $Y$ are called smoothly homotopic 
if there exist smooth maps $f: X \to Y$ and $g: Y \to X$ such that 
$g \circ f \sim \id_X$ and $f \circ g \sim \id_Y$. 

\begin{rem} 
The differentiable structure on $\R$ is defined as in Example \ref{150623_1} (i)-(a), i.e. 
$(U,\ph) \in \mca{P}(\R) \iff \ph \in C^\infty(U, \R)$. 
The differentiable structure on $X \times \R$ is defined as in Example \ref{150623_1} (iv). 
\end{rem} 

\begin{rem} 
It seems difficult to prove the transitivity of $\sim$, 
since our definition of differentiable structures requires only very weak assumptions. 
The author expects that the transitivity does not hold in general. 
\end{rem} 

Unless otherwise specified, any $C^\infty$-manifold $M$ will be equipped with the differentiable structure in Example \ref{150623_1} (i)-(a). 
When $C^\infty$-manifolds $M$, $N$ are equipped with these differentiable structures, 
a map $f: M \to N$ is smooth if and only if $f$ is of class $C^\infty$.

\subsection{de Rham chain complex} 

Let $X$ be a differentiable space. 
For any $k \in \Z$, we set 
\[
\bar{C}^\dR_k(X):= \bigoplus_{(U,\ph) \in \mca{P}(X)} \mca{A}^{\dim U-k}_c(U). 
\]
Notice that $\bar{C}^\dR_k(X)=0$ for any $k<0$. 

For any $(U,\ph) \in \mca{P}(X)$ and $\omega \in \mca{A}^{\dim U-k}_c(U)$, 
let $(U,\ph,\omega)$ denote the image of $\omega$ by the natural injection 
$\mca{A}^{\dim U-k}_c(U) \to \bar{C}^\dR_k(X)$. 
Let $Z_k(X)$ denote the subspace of $\bar{C}^\dR_k(X)$, which is generated by 
\begin{align*} 
&\{(U,\ph, \pi_!\omega) - (V, \ph \circ \pi, \omega) \mid (U,\ph) \in \mca{P}(X), \quad V \in \mca{U}, \quad  \omega \in \mca{A}^{\dim V-k}_c(V), \\
& \qquad \text{$\pi:V \to U$ is a submersion}\}.
\end{align*} 
We define $C^\dR_k(X):= \bar{C}^\dR_k(X)/Z_k(X)$. 
For every $k \in \Z$, 
\[
\partial: C^\dR_k(X) \to C^\dR_{k-1}(X); \qquad [(U,\ph,\omega)] \mapsto [(U, \ph, d\omega)] 
\]
is well-defined, since $d(\pi_! \omega) = \pi_! (d\omega)$. 
Moreover, $\partial^2=0$ since $d^2=0$. 
We call $(C^\dR_*(X), \partial)$ the \textit{de Rham chain complex} of $X$, 
and denote its homology by $H^\dR_*(X)$. 
Elements of $C^\dR_*(X)$ are called \textit{de Rham chains} of $X$. 

\begin{rem}
Our notion of de Rham chains is inspired by the notion of \textit{approximate de Rham chains} by K. Fukaya (\cite{Fuk_06} Definition 6.4). 
However, an explicit definition of a chain complex is not given in \cite{Fuk_06}. 
\end{rem}

The \textit{augmentation map} 
$\ep: C^\dR_0(X) \to \R$ is defined by 
\[
\ep( [(U,\ph,\omega)]):= \int_U  \omega. 
\]
$\ep$ vanishes on $\partial  C^\dR_1(X)$ by Stokes' theorem. 

Next we define the \textit{fiber product} on de Rham chain complexes. 
Let $M$ be an oriented $C^\infty$-manifold of dimension $d$. 
Let us consider the differentiable space $M_\reg$ in Example \ref{150623_1} (i)-(b). 
Let $X$, $Y$ be differentiable spaces, and $e_X: X \to M_\reg$, $e_Y: Y \to M_\reg$ be smooth maps. 
We define a differentiable structure on 
\[
X \times_M Y:= \{(x,y) \in X \times Y \mid e_X(x)= e_Y(y)\}
\]
as a subset of $X \times Y$ (see Example \ref{150623_1} (iii) and (iv)). 

We are going to define a chain map  
\begin{equation}\label{150127_1} 
C^\dR_{k+d}(X) \otimes C^\dR_{l+d}(Y) \to C^\dR_{k+l+d}(X \times_M Y); \quad a \otimes b \mapsto a \times_M b, 
\end{equation} 
which we call the fiber product on de Rham chain complexes. 

Let $(U,\ph) \in \mca{P}(X)$, $(V,\psi) \in \mca{P}(Y)$. 
Then, 
$e_U:= e_X \circ \ph: U \to M$ and $e_V:= e_Y \circ \psi: V \to M$ are submersions. 
Thus, 
$U \times_M V:= \{(u,v) \in U \times V \mid e_U(u) = e_V(v) \}$ 
is a $C^\infty$-manifold, moreover it is a submanifold of a Euclidean space (since $U$ and $V$ are in $\mca{U}$). 
The map $e_{UV}: U \times_M V \to M; (u,v) \mapsto e_U(u)=e_V(v)$ is also a submersion.

Let us define an orientation on $U \times_M V$. 
Let $F_U:= \Ker \, de_U \subset TU$, 
$F_V:=\Ker \, de_V \subset TV$. 
There exist exact sequences 
\begin{align*} 
0& \to F_U \to TU \to e_U^*TM \to 0, \qquad 0 \to F_V \to TV \to e_V^*TM \to 0, \\
0& \to F_U \oplus F_V \to T(U\times_M V) \to e_{UV}^*TM \to 0.
\end{align*} 
Then we obtain isomorphisms 
\begin{align*}
\det(TU) &\cong (e_U)^* \det(TM) \otimes \det(F_U), \\
\det(TV) &\cong (e_V)^* \det(TM) \otimes \det(F_V),  \\
\det(T(U \times_M V)) &\cong (e_{UV})^* \det(TM) \otimes \det(F_U) \otimes \det(F_V).
\end{align*}
Since $M$, $U$ and $V$ are oriented, 
one can define an orientation on $F_U$ (resp. $F_V$) so that the first (resp. second) isomorphism is orientation-preserving. 
Then, 
one can define an orientation of $U \times_M V$ so that the third isomorphism is orientation-preserving. 
Equipped with this choice of orientation, $U \times_M V$ is an oriented submanifold of a Euclidean space and is therefore an element of $\mca{U}$. 

Let $\pi_X: X \times_M Y \to X$, $\pi_Y: X \times_M Y \to Y$, $\pi_U: U \times_M V \to U$ and $\pi_V: U \times_M V \to V$ be the projection maps. Then 
\begin{itemize} 
\item Since $\pi_U$ is a submersion, $(U \times_M V, \pi_X \circ (\ph\times \psi)) = (U \times_M V, \ph \circ \pi_U) \in  \mca{P}(X)$. 
\item Since $\pi_V$ is a submersion, $(U \times_M V, \pi_Y \circ (\ph \times \psi)) = (U \times_M V, \psi \circ \pi_V) \in \mca{P}(Y)$. 
\end{itemize}
Therefore $(U \times_M V, \ph \times \psi) \in \mca{P}(X \times_M Y)$. 
Now, let us define (\ref{150127_1}) by 
\[
[(U, \ph, \omega)] \times_M [(V, \psi, \eta)]:= (-1)^{l(\dim U-d)} [(U \times_M V, \ph \times \psi, \omega \times \eta)]. 
\]
It is easy to check that this is a well-defined chain map, and the product is associative. 

The smooth map $i: X \times_M Y \to Y \times_M X; \, (x,y) \mapsto (y,x)$ 
induces a chain map $i_*: C^\dR_*(X \times_M Y) \to C^\dR_*(Y \times_M X)$ (see Section 4.5). 
We need the lemma below for later use. 

\begin{lem} 
$i_*(a \times_M b)= (-1)^{kl} b \times_M a$
for any $a \in C^\dR_{k+d}(X)$ and $b \in C^\dR_{l+d}(Y)$. 
\end{lem}
\begin{proof} 
Setting $a:= [(U, \ph, \omega)]$ and $b:=[(V, \psi, \eta)]$, we obtain 
\begin{align*} 
i_*(a \times_M b) &= (-1)^{l (\dim U-d)} [(U \times_M V, i \circ (\ph \times \psi), \omega \times \eta)] \\
&=(-1)^{l (\dim U- d) + (\dim U-d) (\dim V-d) + |\omega||\eta|} [(V \times_M U, \psi \times \ph, \eta \times \omega)]
= (-1)^{kl} b \times_M a.
\end{align*} 
Notice that $(\dim U- d) (\dim V-d)$ appears in the exponent on the second line, 
since the map $U \times_M V \to V \times_M U; (u,v) \mapsto (v,u)$ changes the orientations by $(-1)^{(\dim U-d)(\dim V-d)}$.
\end{proof} 

When $M$ is a positively oriented point (thus $d=0$), the chain map (\ref{150127_1}) is 
\[
C^\dR_k(X) \otimes C^\dR_l(Y) \to C^\dR_{k+l}(X \times Y); \qquad a \otimes b \mapsto a \times b,
\]
which we  call the \textit{cross product} on de Rham chain complexes. 

\subsection{de Rham chain complex of $\pt$}

Let $\pt$ be a set which has a unique element. 
For any $U \in \mca{U}$, there exists a unique map $\ph_U: U \to \pt$. 
We define a differentiable structure on $\pt$ by 
$\mca{P}(\pt):= \{ (U,\ph_U) \mid U \in \mca{U}\}$. 

Let us consider $\{0\} \in \mca{U}_{1,0} \subset \mca{U}$ with the positive orientation. 
For any $U \in \mca{U}$, let us denote the unique map $U \to \{0\}$ by $\pi_U$. 
For any $\omega \in \mca{A}^*_c(U)$, 
we obtain
\[
[(U, \ph_U, \omega)] = [(\{0\}, \ph_{\{0\}}, (\pi_U)_! \omega)] = [(\{0\}, \ph_{\{0\}}, \int_U \omega)].
\]
Therefore, 
$C^\dR_k(\pt)=0$ if $k \ne 0$, 
and the augmentation map $\ep: C^\dR_0(\pt) \to \R$ is an isomorphism. 
In particular, 
\[
H^\dR_k(\pt) \cong  \begin{cases} 
                      \R &(k =0),  \\
                        0 &(k \ne 0).
                     \end{cases}
\] 

\subsection{Functoriality}

For any differentiable spaces $X$, $Y$ and a smooth map $f:X \to Y$, 
\[
f_*: C^\dR_*(X) \to C^\dR_*(Y); \qquad
[(U,\ph,\omega)] \mapsto [(U, f \circ \ph, \omega)]
\]
is a well-defined chain map.

\begin{prop}\label{prop:homotopy}
Let $X$, $Y$ be differentiable spaces, and $f,g: X \to Y$ be smooth maps. 
If $f$, $g$ are smoothly homotopic, then 
$f_*, g_*: C^\dR_*(X) \to C^\dR_*(Y)$ are chain homotopic. 
\end{prop}
\begin{proof} 
Let $h: X \times \R \to Y$ be a smooth homotopy between $f$ and $g$, 
i.e., 
$h$ is a smooth map such that 
$h(x,s)=f(x)$ if $s<0$ and $h(x,s)=g(x)$ if $s>1$. 

Take $a \in C_c^\infty(\R)$ so that $a \equiv 1$ on $[-\ep, 1+\ep]$ for some $\ep>0$. 
Let $u:= [(\R, \id_\R, a)] \in C^\dR_1(\R)$, and 
define a linear map 
$K: C^\dR_*(X) \to C^\dR_{*+1}(Y)$ by 
$K(x):= (-1)^{|x|} h_*(x \times u)$. Since $h_*$ and the cross product are chain maps, for any $x \in C^\dR_*(X)$, 
we have 
\[
\partial K(x) + K(\partial x) = (-1)^{|x|} h_*(\partial (x \times u)) + (-1)^{|x|-1} h_*(\partial x \times u) =  h_*( x \times \partial u). 
\]
Thus, 
it is enough to show that $h_*( x \times \partial u) = f_*(x) - g_*(x)$. 
Since both sides are linear on $x$, we may assume that 
$x=[(U,\ph,\omega)]$ for some $(U,\ph) \in \mca{P}(X)$ and $\omega \in \mca{A}^*_c(U)$. 

Let $i_0: \R_{<0} \to \R$, $i_1:\R_{>1} \to \R$ be the inclusion maps, 
and 
$\alpha_0:= da|_{\R_{<0}}$, $\alpha_1:= da|_{\R_{>1}}$. 
Define $v_0, v_1 \in C_0^\dR(\R)$ by 
$v_0:=[(\R_{<0} , i_0, \alpha_0)]$, $v_1:=[(\R_{>1}, i_1, \alpha_1)]$. 
Then, since $da$ is supported on $\R_{<0}  \cup \R_{>1}$, we have 
$\partial u = v_0 + v_1$. 

Let $x= [(U, \ph, \omega)] \in C^\dR_*(X)$. 
Then, 
\[
h_*(x \times v_0) = \big[  (U \times \R_{<0},  h \circ (\ph \times i_0), \omega \times \alpha_0) \big].
\]
There holds $h \circ (\ph \times i_0)= f \circ \ph \circ \pr_U$, 
where $\pr_U: U \times \R_{<0} \to U$ is the projection map.
Moreover, $(\pr_U)_!(\omega \times \alpha_0) = \omega$, since $\int_{\R_{<0}} \alpha_0 = a(0) =1$.
Thus, 
$h_*(x \times v_0)= f_*(x)$. 
A similar argument shows $h_*(x \times v_1)= - g_*(x)$. 
Hence, we get 
\[
h_*(x \times \partial u) =  h_*(x \times (v_0+v_1)) = f_*(x) - g_*(x). 
\]
\end{proof} 

\subsection{Truncated spaces} 

Let $X$ be a differentiable space. 
For any function $f:X \to \R$ and $a \in \R \cup \{\infty\}$, let us define 
\[
X_{f,a}:= \{ x \in X \mid f(x) < a\}. 
\]
This is a subspace of $X$ truncated by the inequality $f(x)<a$. 
Obviously $X_{f, \infty} = X$. 
We prove some technical results about de Rham chain complexes of these truncated spaces. 

A function $f: X \to \R$ is called smooth, if $f \circ \ph \in C^\infty(U)$ for any $(U,\ph) \in \mca{P}(X)$. 
$f$ is called \textit{approximately smooth}, if there exists 
a decreasing sequence $(f_j)_{j \ge 1}$ of smooth functions on $X$ such that 
$f(x) = \lim_{j \to \infty} f_j(x)$ for every $x \in X$. 

\begin{rem} 
If $f: X \to \R$ is approximately smooth, $f \circ \ph: U \to \R$ is upper semi-continuous for any $(U,\ph) \in \mca{P}(X)$.
This is because $(f \circ \ph)^{-1}(\R_{<a}) = \bigcup_{j=1}^\infty (f_j \circ \ph)^{-1}(\R_{<a})$ for any $a \in \R$. 
\end{rem}

An important example of an approximately smooth (but not smooth) function 
is the length functional on the free loop space. 

\begin{lem}
Let $M$ be a Riemannian manifold, and consider the differentiable structure on $\mca{L}M$ 
as in Example \ref{150623_1} (ii). 
Then, $\len: \mca{L} M \to \R; \, \gamma \mapsto \int_{S^1} |\dot{\gamma}|$ 
is approximately smooth. 
\end{lem}
\begin{proof} 
It is easy to check that, for any $\rho \in C^\infty(\R_{\ge 0})$, 
the functional 
\[
\E_\rho: \L M \to \R; \qquad \gamma \mapsto \int_{S^1} \rho(|\dot{\gamma}|^2) 
\]
is smooth. 
Let us take a decreasing sequence $(\rho_j)_{j \ge 1}$ on $C^\infty(\R_{ \ge 0})$, such that  
$\lim_{j \to \infty} \rho_j(t) = \sqrt{t}$ for any $t \ge 0$. 
Then, $(\E_{\rho_j})_{j \ge 1}$ is a decreasing sequence of smooth functions on $\L M$, such that
$\lim_{j \to \infty} \E_{\rho_j} = \len$. 
\end{proof} 

\begin{lem}\label{150210_1}
Let $X$ be a differentiable space. 
\begin{enumerate} 
\item[(i):] 
Let $(f_j)_{j \ge 1}$ be a decreasing sequence of approximately smooth functions on $X$, such that 
$\lim_{j \to \infty} f_j(x) < 0$ for every $x \in X$. 
Then $\vlim_{j \to \infty} C^\dR_*(X_{f_j,0}) \to C^\dR_*(X)$, which is induced by the inclusion maps, is surjective. 
\item[(ii):]
Let $f$ be an approximately smooth function on $X$, 
$(c_j)_{j \ge 1}$ be an increasing sequence of real numbers, and $c:= \sup_j c_j$. 
Then $\vlim_{j \to \infty} C^\dR_*(X_{f,c_j}) \to C^\dR_*(X_{f,c})$ is surjective. 
\end{enumerate} 
\end{lem}
\begin{proof}
(ii) follows from (i) by setting $f_j: = (f-c_j)|_{X_{f,c}}$. 
Thus it is enough to prove (i). 
Let $(U,\ph) \in \mca{P}(X)$ and $\omega \in \mca{A}^*_c(U)$. 
Setting $\ph_j:= f_j \circ \ph$ for each $j \ge 1$, 
$(\ph_j)_{j \ge 1}$ is a decreasing sequence of upper semi-continuous functions on $U$. 
Since $\supp \, \omega$ is compact, there exists $j$ such that 
$\ph_j(u)<0$ for any $u \in \supp \  \omega$. 
Setting $U_j:=\{ u \in U \mid \ph_j(u) <0\}$, 
the chain map $C^\dR_*(X_{f_j,0}) \to C^\dR_*(X)$ maps
$[(U_j, \ph|_{U_j}, \omega|_{U_j})]$ to 
$[(U,\ph,\omega)]$. 
\end{proof}

\begin{lem}\label{150210_2}
Let $f: X \to \R$ be an approximately smooth function. 
Then, the chain map 
$I_0:C^\dR_*(X_{f,0}) \to C^\dR_*(X)$, which is induced by the inclusion map, is injective. 
\end{lem}
\begin{proof}
First let us consider the case when $f$ is a smooth function on $X$. 
For any $c < 0$, 
\[
I_{c,0}: C^\dR_*(X_{f,c}) \to C^\dR_*(X_{f,0}), \qquad
I_c: C^\dR_*(X_{f,c}) \to C^\dR_*(X)
\]
denote the chain maps induced by the inclusion maps. 
Suppose that 
$u \in C^\dR_*(X_{f,0})$ satisfies $I_0(u)=0$. 
By Lemma \ref{150210_1} (ii), 
there exists $c<0$ such that $u \in \Image I_{c,0}$. 
Take $v \in C^\dR_*(X_{f,c})$ so that $u= I_{c,0}(v)$. 

If there exists a linear map $J: C^\dR_*(X) \to C^\dR_*(X_{f,0})$ such that 
$J \circ I_c = I_{c,0}$, 
we can prove $u=0$ by 
\[
u= I_{c,0}(v) = J \circ I_c (v) = J \circ I_0 \circ I_{c,0}(v) = J \circ I_0(u)=0.
\]
To define such $J$, 
we fix $\kappa \in C^\infty(\R, [0,1])$ so that $\kappa \equiv 1$ on $\R_{\le c}$ and $\supp \kappa \subset \R_{<0}$. 
For any $(U,\ph) \in \mca{P}(X)$, 
we set $U_{f,0}:= \{ u \in U \mid f(\ph(u)) < 0\}$. 
Then, the linear map 
\[
J: C^\dR_*(X) \to C^\dR_*(X_{f,0}); \qquad 
[(U,\ph,\omega)] \mapsto [ (U_{f,0}, \ph|_{U_{f,0}},  ((\kappa \circ f \circ \ph) \cdot \omega)|_{U_{f,0}})]
\]
is well-defined (it is \textit{not} a chain map). 
$J \circ I_c = I_{c,0}$ is obvious since $\kappa \equiv 1$ on $\R_{\le c}$. 
This completes the proof when $f$ is a smooth function on $X$. 

Finally, we consider the case when $f$ is any approximately smooth function on $X$. 
By definition, there exists a decreasing sequence of smooth functions $(f_j)_{j \ge 1}$ such that 
$\lim_{j \to \infty} f_j = f$. 
Then Lemma \ref{150210_1} (i) applied to $(f_j|_{X_{f,0}})_{j \ge 1}$ shows that 
the chain map $\vlim_{j \to \infty} C^\dR_*(X_{f_j,0}) \to C^\dR_*(X_{f,0})$ is surjective. 
On the other hand, since each $f_j$ is smooth, 
$C^\dR_*(X_{f_j,0}) \to C^\dR_*(X)$ is injective. 
Thus $C^\dR_*(X_{f,0}) \to C^\dR_*(X)$ is injective. 
\end{proof} 

\begin{cor}\label{150213_1} 
Let $X$ be a differentiable space. 
\begin{enumerate} 
\item[(i):] 
Let $(f_j)_{j \ge 1}$ be a decreasing sequence of approximately smooth functions on $X$, such that 
$\lim_{j \to \infty} f_j(x)<0$ for every $x \in X$. 
Then $\vlim_{j \to \infty} C^\dR_*(X_{f_j,0}) \to C^\dR_*(X)$, which is induced by the inclusion maps, is an isomorphism. 
\item[(ii):] 
Let $f$ be an approximately smooth function on $X$,
$(c_j)_{j \ge 1}$ be an increasing sequence of real numbers, 
and $c:= \sup_j c_j$. 
Then $\vlim_{j \to \infty} C^\dR_*(X_{f,c_j}) \to C^\dR_*(X_{f,c})$ is an isomorphism. 
\end{enumerate} 
\end{cor}
\begin{proof}
(ii) follows from (i) by setting $f_j: = (f-c_j)|_{X_{f,c}}$. 
Thus it is enough to prove (i). 
We already proved the surjectivity in Lemma \ref{150210_1} (i). 
By Lemma \ref{150210_2}, 
$C^\dR_*(X_{f_j,0}) \to C^\dR_*(X)$ is injective for every $j$, and thus 
$\vlim_{j \to \infty} C^\dR_*(X_{f_j,0}) \to C^\dR_*(X)$ is also injective. 
\end{proof}

\subsection{Smooth singular chains} 

We define \textit{smooth singular chains} on differentiable spaces, 
and compare them with de Rham chains. 

Let $X$ be a differentiable space, and $k \ge 0$ be an integer. 
A map $\sigma: \Delta^k \to X$ is called \textit{strongly smooth}, if 
there exists an open neighborhood $U$ of $\Delta^k \subset \R^k$, 
and a smooth map $\bar{\sigma}: U \to X$ such that $\bar{\sigma}|_{\Delta^k}= \sigma$. 
$\Delta^k$ and $U$ are equipped with the differentiable structures as subsets of $\R^k$. 

For any $k \ge 0$, 
let $C^\sm_k(X)$ denote the $\R$-vector space generated by all strongly smooth maps $\Delta^k \to X$. 
For any $k<0$, we set $C^\sm_k(X):=0$. 
A differential on $C^\sm_*(X)$ is defined in the same way as in the singular chain complex.
To fix notation, let us spell out the details. 
For any $k \ge 1$ and $0 \le j \le k$, 
we define a map $d_{k,j}: \Delta^{k-1} \to \Delta^k$ by 
\begin{equation}\label{eq:dkj}
d_{k,j}(t_1,\ldots,t_{k-1}):= \begin{cases}
                                 (0,t_1,\ldots, t_{k-1}) &( j=0),  \\
                                 (t_1, \ldots, t_j, t_j, \ldots, t_{k-1}) &(1 \le j \le k-1), \\
                                 (t_1,\ldots, t_{k-1}, 1)&(j=k).
                                 \end{cases}
\end{equation}
In particular, $d_{1,j}: \Delta^0 \to \Delta^1$ is defined as 
$d_{1,j}(0)=j$ for $j=0, 1$. 
For any $k \ge 1$, 
a differential $\partial: C^\sm_k(X) \to C^\sm_{k-1}(X)$ is defined as 
\[
\partial \sigma := \sum_{j=0}^k  (-1)^j  \sigma \circ d_{k,j}. 
\]
We call the chain complex $(C^\sm_*(X), \partial)$ the \textit{smooth singular chain complex}, and 
its homology $H^\sm_*(X)$ the \textit{smooth singular homology}. 
For any smooth map $f: X \to Y$ between differentiable spaces, one can define the chain map $f_*: C ^\sm_*(X) \to C^\sm_*(Y)$ in the obvious way. 
If smooth maps $f,g: X \to Y$ are smoothly homotopic, $f_*, g_*: C^\sm_*(X) \to C^\sm_*(Y)$ are chain homotopic. 

\begin{lem}\label{150624_10} 
\begin{enumerate}
\item[(i):]
There exists a sequence $(u_k)_{k \ge 0}$ such that 
$u_k \in C^\dR_k(\Delta^k)$ for any $k \ge 0$, 
$u_0 \in C^\dR_0(\Delta^0)$ is characterized by $\ep(u_0)=1$, 
and 
\[
\partial u_k = \sum_{j=0}^k (-1)^j (d_{k,j})_*(u_{k-1})  \qquad(\forall k \ge 1).
\]
\item[(ii):]
Suppose that $(u_k)_{k \ge 0}$ and $(u'_k)_{k \ge 0}$ satisfy the conditions in (i). 
Then there exists a sequence $(v_k)_{k \ge 1}$ such that 
$v_k \in C^\dR_{k+1}(\Delta^k)$ for any $k \ge 1$, 
$\partial v_1 = u_1- u'_1$, and 
\[
\partial v_k  = u_k - u'_k - \sum_{j=0}^k (-1)^j (d_{k,j})_*(v_{k-1}) \qquad(\forall k \ge 2). 
\]
\end{enumerate}
\end{lem}
\begin{proof}
For any $k \ge 0$, $\Delta^k$ is smoothly homotopic to $\pt$, and thus $H^\dR_*(\Delta^k) \cong H^\dR_*(\pt)$. 
Using this fact, the assertions are easy to prove by induction on $k$. 
\end{proof} 

For any $u=(u_k)_{k \ge 0}$ which satisfies Lemma \ref{150624_10} (i), one can define 
a natural transformation $\iota^u: C^\sm_* \to C^\dR_*$ as follows (we set $\iota^u_k=0$ when $k<0$): 
\[
\iota^u(X)_k: C^\sm_k(X) \to C^\dR_k(X); \quad  \sigma \mapsto \sigma_*(u_k).
\]
Then $\iota^u(X)_*$ is a chain map. 
Lemma \ref{150624_10} (ii) shows that the homotopy equivalence class 
of $\iota^u(X)_*$ does not depend on the choices of $u$. 
In particular, the linear map 
$H_*(\iota^u(X)): H^\sm_*(X) \to H^\dR_*(X)$
does not depend on $u$. 

Finally, we define the cross product on $C^\sm_*$. 
Let us take $\tau_{k,l} \in C^\sm_{k+l}(\Delta^k \times \Delta^l)$ for all $k,l \ge 0$, such that 
$\tau_{0,0}$ is characterized by $\ep(\tau_{0,0})=1$, 
and the following equation holds for any $k,l \ge 0$:
\[
\partial \tau_{k,l} = \sum_{0 \le i \le k} (-1)^i (d_{k,i} \times \id_{\Delta^l})_*( \tau_{k-1,l}) + (-1)^k \sum_{0 \le j \le l} (-1)^j (\id_{\Delta^k} \times d_{l,j})_*(\tau_{k, l-1}).
\]
We define the cross product 
$C^\sm_k(X) \otimes C^\sm_l(Y) \to C^\sm_{k+l}(X \times Y)$ by 
\[
\sigma_X \otimes \sigma_Y \mapsto (\sigma_X \times \sigma_Y)_*(\tau_{k,l}) \qquad
(\sigma_X: \Delta^k \to X, \, \sigma_Y: \Delta^l \to Y, \, k,l \ge 0). 
\]
The homotopy equivalence class of this map does not depend on the choices of $(\tau_{k,l})_{k,l \ge 0}$. 

On the other hand, we defined the cross product for de Rham chains
(see Section 4.3). 
The next lemma is proved by the standard method of acyclic models. 

\begin{lem}\label{150627_1} 
For any differentiable spaces $X$, $Y$ and $u=(u_k)_{k \ge 0}$ satisfying Lemma \ref{150624_10} (i), 
the following diagram commutes up to homotopy, where horizontal maps are cross products: 
\[
\xymatrix{
C^\sm_*(X) \otimes C^\sm_*(Y) \ar[r] \ar[d]_{\iota^u(X)  \otimes \iota^u(Y) } & C^\sm_*(X \times Y) \ar[d]^{\iota^u(X \times Y) } \\
C^\dR_*(X) \otimes C^\dR_*(Y) \ar[r] & C^\dR_*(X \times Y).
}
\]
\end{lem} 

\subsection{Integration over de Rham chains} 

Let $M$ be a $C^\infty$-manifold, and $n \ge 0$. 
We define $\langle \, , \,\rangle: \mca{A}^n(M) \otimes C^\dR_n(M) \to \R$ by 
\[
\langle \alpha, [(U, \ph, \omega)] \rangle := \int_U \ph^* \alpha \wedge \omega.
\]
It induces a linear map $H^n_\dR(M) \otimes H^\dR_n(M) \to \R$, which we also denote by $\langle \, , \, \rangle$. 

For any subset $S \subset M$, let $\mca{A}^*(S):= \vlim_ U \mca{A}^*(U)$ where $U$ runs over all open neighborhoods of $S$, 
and let $H^*_\dR(S):=H^*(\mca{A}^*(S), d)$. 
Then, one can define $H^*_\dR(S) \otimes H^\dR_*(S) \to \R$. 
In the next lemma, we consider the case $S= \Delta^{k_1} \times \cdots \times \Delta^{k_m} \subset \R^{k_1+\cdots+k_m}$. 

\begin{lem}\label{150629_2} 
Suppose that $(u_k)_{k \ge 0}$ satisfies Lemma \ref{150624_10} (i). 
For any nonnegative integers $k_1, \ldots, k_m$ and 
$\alpha \in \mca{A}^K(\Delta^{k_1} \times \cdots \times \Delta^{k_m})$, where $K:= k_1 + \cdots + k_m$, there holds 
\[
\langle \alpha, u_{k_1} \times \cdots \times u_{k_m} \rangle = (-1)^{K(K-1)/2} \int_{\Delta^{k_1}\times \cdots \times \Delta^{k_m}} \alpha.
\]
\end{lem}
\begin{proof}
We fix $m$ and prove the lemma by induction on $K$.  
When $K=0$, i.e. $k_1=\cdots=k_m=0$, the lemma can be directly checked. 
If the lemma is established for $K \le N-1$, the case $K=N$ is proved as follows.
Let us take $\beta \in \mca{A}^{K-1}(\Delta^{k_1} \times \cdots \times \Delta^{k_m})$ such that $d\beta=\alpha$ (this is always possible since $H^K_\dR(\Delta^{k_1}\times \cdots \times \Delta^{k_m})=0$). 
Then, 
\begin{align*}
\langle d\beta , u_{k_1} \times \cdots \times u_{k_m} \rangle &= (-1)^K \sum_{j=1}^m (-1)^{k_1+\cdots+ k_{j-1}} \langle \beta,  u_{k_1} \times \cdots \times \partial u_{k_j} \times \cdots \times u_{k_m} \rangle \\
&=(-1)^{K+ (K-1)(K-2)/2} \int_{\partial(\Delta^{k_1}\times \cdots \times \Delta^{k_m})} \beta \\
&=(-1)^{K(K-1)/2} \int_{\Delta^{k_1}\times \cdots \times \Delta^{k_m}} d \beta.
\end{align*}
The second equality follows from the induction hypothesis, 
and the last equality follows from Stokes' theorem. 
\end{proof} 

\section{de Rham chains on $C^\infty$-manifolds} 

Let $M$ be an oriented $C^\infty$-manifold, 
and consider the differentiable structure as in Example \ref{150623_1} (i)-(a), i.e. 
$\mca{P}(M):= \{(U,\ph) \mid U \in \mca{U}, \, \ph \in C^\infty(U, M)\}$. 
With this differentiable structure, a map $\Delta^k \to M$ is strongly smooth if 
it extends to a $C^\infty$-map $U \to M$ for some open neighborhood $U$ of $\Delta^k \subset \R^k$. 
In particular, $C^\sm_*(M)$ is a subcomplex of $C_*(M)$ (the usual singular chain complex of $M$). 
It is known that $C^\sm_*(M) \to C_*(M)$ is a quasi-isomorphism (see \cite{Lee_13} Theorem 18.7). 
Therefore, we obtain a natural isomorphism $H^\sm_*(M) \cong H_*(M)$. 

In Section 4.7, 
we defined the map $H^\sm_*(X) \to H^\dR_*(X)$ for any differentiable space $X$. 
The goal of this section is to prove the following theorem. 
As an immediate consequence, we obtain a natural isomorphism $H_*(M) \cong H^\dR_*(M)$. 

\begin{thm}\label{150623_2}
For any oriented $C^\infty$-manifold $M$, 
the map $H^\sm_*(M) \to H^\dR_*(M)$ is an isomorphism. 
\end{thm} 

Let us denote the map $H^\sm_*(M) \to H^\dR_*(M)$ by $I_0$. 
The proof of Theorem \ref{150623_2} is separated into  two steps. 
Let $d:= \dim M$. 
\begin{itemize}
\item In Section 5.1, we define an isomorphism $I_1: H^{d-*}_{c,\dR}(M) \cong H^\dR_*(M)$. 
\item In Section 5.2, we define an isomorphism $I_2: H^\sm_*(M) \cong H^{d-*}_{c,\dR}(M)$ via Poincar\'{e} duality, 
and show that $I_0 = \pm I_1 \circ I_2$. 
\end{itemize} 

\subsection{Comparison with compactly supported de Rham cohomology}

Let us consider the differentiable space $M_\reg$ (see Example \ref{150623_1} (i)-(b)). 
It is easy to check that 
\begin{equation}\label{150210_3} 
C^\dR_*(M_\reg) \to \mca{A}^{d-*}_c(M); \quad [(U,\ph,\omega)] \mapsto \ph_!\omega
\end{equation} 
is a well-defined chain map. 

On the other hand, for any $\omega \in \mca{A}^{d-*}_c(M)$, 
let us take $U \in \mca{U}$ and an orientation-preserving open embedding $\ph:U \to M$ such that 
$\supp\, \omega \subset \ph(U)$. 
Then, $[(U,\ph, \ph^*\omega)] \in C^\dR_*(M_\reg)$ does not depend on the choices of $U$ and $\ph$. 
Thus, one can define a chain map 
\[
\mca{A}^{d-*}_c(M) \to C^\dR_*(M_\reg); \quad \omega \mapsto [(U,\ph,\ph^*\omega)], 
\]
and this is the inverse of (\ref{150210_3}). 
Therefore, (\ref{150210_3}) is an isomorphism of chain complexes. 
In particular, $H^\dR_*(M_\reg) \cong H^{d-*}_{c,\dR}(M)$. 

$\id_M: M_\reg \to M$ is a map of differentiable spaces, as noted in Example \ref{150623_1} (i). 
The goal of this subsection is to prove the next proposition. 

\begin{prop}\label{150210_4} 
$\id_M: M_\reg \to M$ induces an isomorphism 
$H^\dR_*(M_\reg) \cong H^\dR_*(M)$. 
\end{prop}
As an immediate consequence, we obtain an isomorphism $I_1: H^{d-*}_{c,\dR}(M) \cong H^\dR_*(M)$. 

To prove Proposition \ref{150210_4}, 
let us take a proper $C^\infty$-function $f: M \to \R_{\ge 0}$, 
and set $M^{(k)}:= f^{-1}(\R_{<k})$ for every $k \in \Z_{>0}$. 
Then Corollary \ref{150213_1} (ii) implies isomorphisms 
\[
H^\dR_*(M) \cong \vlim_k H^\dR_*(M^{(k)}), \qquad
H^\dR_*(M_\reg) \cong \vlim_k H^\dR_*(M^{(k)}_\reg).
\]
Here we need the following lemma. 

\begin{lem}\label{150210_5} 
Let $N$ be an oriented $C^\infty$-manifold, and let $W$ be an open set in $N$ with compact closure. 
Then, there exists a chain map 
$J: C^\dR_*(W) \to C^\dR_*(N_\reg)$ such that the following diagram of chain maps commutes up to homotopy. 
\[
\xymatrix{
C^\dR_*(W_\reg) \ar[r]^{I_\reg} \ar[d]_-{(\id_W)_*} & C^\dR_*(N_\reg)\ar[d]^-{(\id_N)_*} \\
C^\dR_*(W) \ar[r]_{I} \ar[ru]_{J}& C^\dR_*(N)}
\]
$I_\reg$ and $I$ are induced by the inclusion map $W \to N$. 
\end{lem} 

Let us apply Lemma \ref{150210_5} for $N=M^{(k+1)}$, $W=M^{(k)}$, 
and take a chain map  $J^{(k)}: C^\dR_*(M^{(k)}) \to C^\dR_*(M^{(k+1)}_\reg)$ as in Lemma \ref{150210_5}. 
Then, 
\[
\vlim_k  H_*(J^{(k)}): \vlim_k H^\dR_*(M^{(k)}) \to \vlim_k H^\dR_*(M^{(k)}_\reg)
\]
is the inverse of $H^\dR_*(M_\reg) \to H^\dR_*(M)$, 
thus Proposition \ref{150210_4} is proved. 

To prove Lemma \ref{150210_5}, we need the following lemma. 

\begin{lem}\label{150210_6}
Let $N$ be a $C^\infty$-manifold, and $K$ be a compact set in $N$.
There exists an integer $D>0$ and a $C^\infty$-map $F: N \times \R^D \to N$ such that the following conditions hold. 
\begin{itemize} 
\item For any $z \in \R^D$, $F_z: N \to N; x \mapsto F(x,z)$ is a diffeomorphism.
\item $F_{(0,\ldots,0)} = \id_N$. 
\item For any $x \in K$, $\R^D \to N; z \mapsto F(x,z)$ is a submersion. 
\end{itemize} 
\end{lem} 
\begin{proof}
Let $\mca{X}_c(N)$ denote the space of compactly supported $C^\infty\,$-vector fields on $N$. 
For any $\xi \in \mca{X}_c(N)$, let $(\Phi^t_\xi)_{t \in \R}$ denote the flow generated by $\xi$. 

Let us take a sequence $\xi=(\xi_j)_{1 \le j \le D}$ on $\mca{X}_c(N)$, 
such that $(\xi_j(x))_j$ spans $T_xN$ for any $x \in K$. 
For $z=(z_1, \ldots, z_D) \in \R^D$, we set $z \cdot \xi:= \sum_j z_j \xi_j$. 
Let us define a $C^\infty$-map 
 $f: N \times \R^D \to N$ by $f(x,z):= \Phi^1_{z \cdot \xi}(x)$. 
Then, $\partial_z f(x,z): \R^D \to T_{f(x,z)}N$ is onto for any $x \in K$ and $|z|<\ep$ for some $\ep>0$. 
Finally, take any diffeomorphism $g: \R^D \to \{z \in \R^D \mid |z|<\ep\}$ preserving the origin. 
Then, $F(x,z):= f(x,g(z))$ satisfies the requirements of the lemma. 
\end{proof}

Remark \ref{150626_4} below will be used in the proof of Lemma \ref{150626_3}. 

\begin{rem}\label{150626_4} 
When $N$ is a Riemannian manifold, for any $\delta>0$ we may further require the following condition: 
for any $v \in TN$ and $z \in \R^D$, 
$|dF_z(v)| \le (1+\delta) |v|$. 
Indeed, this condition is satisfied if we take $\ep$ in the above proof sufficiently small for given $\delta$, 
since as $w \in \R^D$ converges to $(0,\ldots, 0)$ the map 
$N \to N; \,x \mapsto f(x,w)$ converges to $\id_N$ in the $C^1$ (moreover in the $C^\infty$) topology. 
\end{rem} 

\begin{proof}[\textbf{Proof of Lemma \ref{150210_5}}]
Let us apply Lemma \ref{150210_6} for $K=\bar{W}$, 
and take an integer $D>0$ and $F: N \times \R^D \to N$. 
For any $\ph \in C^\infty(U, W)$, 
$F \circ (\ph \times \id_{\R^D}): U \times \R^D \to N$ is a submersion. 
We take $\nu_D \in \mca{A}^D_c(\R^D)$ so that $\int_{\R^D} \nu_D = 1$, 
and define
\[
J: C^\dR_*(W) \to C^\dR_*(N_\reg); \quad [(U,\ph,\omega)] \mapsto [(U \times \R^D, F \circ (\ph \times \id_{\R^D}), \omega \times \nu_D)]. 
\]
It is easy to see that $J$ is a well-defined chain map. 

To show  that $J \circ (\id_W)_*$ and $I_\reg$ are chain homotopic, 
let us take $a, b \in C^\infty(\R, [0,1])$ so that 
\begin{itemize} 
\item $a(s)=0$ for any $s \le 0$, and $a(s)=1$ for any $s \ge 1$. 
\item $\supp b$ is compact, and there exists $\ep>0$ such that $b(s)=1$ for any $s \in [-\ep, 1+\ep]$. 
\end{itemize} 
For any $(U,\ph) \in\mca{P}(W_\reg)$, 
we define $\Phi: U \times \R^D \times \R \to N$ by 
\[
\Phi(u,z,s):= F(\ph(u), a(s)z) = F_{a(s)z}(\ph(u)).
\]
Since $\ph: U \to W$ is a submersion, and $F_{a(s)z}$ is a diffeomorphism of $N$ for any $(z,s) \in \R^D \times \R$, 
$\Phi$ is also a submersion. 
Therefore, $(U \times \R^D \times \R, \Phi) \in \mca{P}(N_\reg)$. 

Now, it is easy to see that 
\[
K: C^\dR_*(W_\reg) \to C^\dR_{*+1} (N_\reg); \quad  [(U,\ph,\omega)] \mapsto (-1)^{|\omega|+D}  [(U \times \R^D \times \R, \Phi, \omega \times \nu_D \times b(s))]
\]
is a well-defined linear map. 
We can also prove 
$\partial K + K \partial = I_\reg - J \circ (\id_W)_*$
as follows: 
\begin{align*} 
&(\partial K + K\partial)([(U, \ph, \omega)]) = [ ( U \times \R^D \times \R , \Phi, \omega \times \nu_D \times db)] \\
&\qquad= [(U \times \R^D \times \R_{<0}, \Phi, \omega \times \nu_D \times db)] 
+ [(U \times \R^D \times \R_{>1}, \Phi, \omega \times \nu_D \times db)] 
\\
&\qquad= [(U \times \R^D , \ph \circ \pr_U , \omega \times \nu_D)] - [(U \times \R^D, F \circ (\ph \times \id_{\R^D}), \omega \times \nu_D)] \\
&\qquad= [(U, \ph, \omega)] -  [(U \times \R^D, F \circ (\ph \times \id_{\R^D}), \omega \times \nu_D)]. 
\end{align*} 
The first equality follows from $d(\omega \times \nu_D \times b) = d\omega \times \nu_D \times b + (-1)^{|\omega|+D} \omega \times \nu_D \times db$. 
The second equality follows since $db$ is supported on $\R_{<0} \cup \R_{>1}$.
The third equality follows from $\Phi(u,z,s) = \begin{cases} \ph(u) &(s<0) \\ F(\ph(u), z) &(s>1) \end{cases}$ and applying integration along fibers for projection maps to $U \times \R^D$. 
The last equality folllows from applying integration along fibers for $\pr_U: U \times \R^D \to U$. 

Similar arguments show that $(\id_N)_* \circ J$ and $I$ are chain homotopic. 
The homotopy operator is given by exactly the same formula as $K$. 
This case is easier, since we do not have to care about the submersion condition. 
\end{proof} 

\subsection{Proof of Theorem \ref{150623_2}}

Let us define an isomorphism 
$I_2: H^\sm_*(M) \cong H^{d-*}_{c,\dR}(M)$. 
When $H^*_\dR(M)$ is finite-dimensional, it is defined by 
\[
I_2: H^\sm_*(M) \cong (H^*_\dR(M))^* \cong H^{d-*}_{c,\dR}(M). 
\]
The first isomorphism follows from an isomorphism 
$H^*_\dR(M) \to H^*_\sm(M)$ which is defined by integration of differential forms on smooth chains 
(for the proof that this is an isomorphism, see \cite{Lee_13} Theorem 18.14), 
and the second isomorphism follows from Poincar\'{e} duality 
$H^*_\dR(M) \cong (H^{d-*}_{c,\dR}(M))^*$. 

To define $I_2$ in the general case, let us define a set $\mca{U}_M$ by 
\[
\mca{U}_M:= \{ \text{a relatively compact open set $U \subset M$ such that $\dim H^*_\dR(U)<\infty$} \}. 
\]
Then, we define $I_2$ by 
\[
I_2: 
H^\sm_*(M) \cong \vlim_{U \in \mca{U}_M} H^\sm_*(U) \cong \vlim_{U \in \mca{U}_M} H^{d-*}_{c,\dR}(U) \cong H^{d-*}_{c,\dR}(M). 
\]

To prove that $I_0$ is an isomorphism, it is enough to check that $I_0 = \pm I_1 \circ I_2$. 
We may assume that $H^*_\dR(M)$ is finite-dimensional, since the general case follows from this case by taking limits. 
Let us consider the following diagram: 
\[
\xymatrix{
H^\sm_*(M)  \ar[r]^-{\cong}\ar[rd]_{I_0}& (H^*_\dR(M))^*  & \ar[l]_-{\cong}\ar[ld]^{I_1} H^{d-*}_{c,\dR}(M) \\
& H^\dR_*(M). \ar[u] &
}
\]
The vertical map $H^\dR_*(M) \to (H^*_\dR(M))^*$ is defined by the pairing 
$\langle \, , \, \rangle: H^*_\dR(M) \otimes H^\dR_*(M) \to \R$ (see Section 4.8). 

To show that $I_0=\pm I_1 \circ I_2$, it is enough to check that the above diagram commutes up to sign. 
The commutativity of the left triangle follows from Lemma \ref{150629_2} (the case $m=1$). 
The commutativity of the right triangle can be checked as follows. 
Let $\omega \in \mca{A}^{d-*}_c(M)$ be a closed form, then $I_1$ maps $[\omega]$ to 
$[(U, \ph, \ph^*\omega)]$, where $U \in \mca{U}$ and $\ph: U \to M$ is any orientation-preserving embedding such that $\supp \omega \subset \ph(U)$. 
Then, the commutativity of the right triangle follows from 
\[ 
\int_M \eta \wedge \omega = \langle \eta , [(U, \ph, \ph^*\omega)]  \rangle \qquad( \forall \eta \in \mca{A}^*(M))
\]
which is obvious from the definition of $\langle \, , \, \rangle$. 

\section{de Rham chains on $C^\infty$ free loop spaces} 

Let $M$ be a closed, oriented Riemannian manifold. 
We abbreviate $\mca{L} M:= C^\infty(S^1, M)$ as $\mca{L}$. 
We consider the differentiable structure on $\mca{L}$ as in Example \ref{150623_1} (ii). 
For any $a \in (0, \infty]$, we set 
$\mca{L}^a:= \{ \gamma \in \mca{L} \mid \len(\gamma) < a\}$, and consider the differentiable structure as a subset of $\mca{L}$
(see Example \ref{150623_1} (iii)). 

Any strongly smooth map $\sigma: \Delta^k \to \L^a$ is continuous with respect to the $C^\infty\,$-topology on $\L^a$. 
Therefore, we obtain a map $H^\sm_*(\L^a) \to H_*(\L^a)$, where the right hand side denotes the singular homology. 
On the other hand, for any differentiable space $X$, we defined the map 
$H^\sm_*(X) \to H^\dR_*(X)$. 
The aim of this section is to prove the following result. 

\begin{thm}\label{150219_1} 
For any closed, oriented Riemannian manifold $M$ and $a \in (0,\infty]$, the maps 
$H^\sm_*(\L^a) \to H^\dR_*(\L^a)$ and $H^\sm_*(\L^a) \to H_*(\L^a)$ are isomorphisms. 
\end{thm} 

As an immediate consequence, we obtain an isomorphism $H^\dR_*(\L^a) \cong H_*(\L^a)$. 
The proof of Theorem \ref{150219_1} uses finite-dimensional approximations of the free loop space $\mca{L}^a$, 
which we explain in Section 6.1.

Recall that the rotation operator $\Delta: H_*(\mca{L}^a) \to H_{*+1}(\mca{L}^a)$ is defined as 
$\Delta(x):=H_*(r)([S^1] \times x)$, where $r: S^1 \times \mca{L}^a \to \mca{L}^a$ denotes the rotation. 
Via isomorphisms $H_*(S^1) \cong H^\sm_*(S^1) \cong H^\dR_*(S^1)$, one can define the rotation operators on 
$H^\sm_*(\mca{L}^a)$ and $H^\dR_*(\mca{L}^a)$ in the same way. 
It is easy to see that the isomorphism $H_*(\mca{L}^a) \cong H^\sm_*(\mca{L}^a)$ preserves the rotation operators, 
since $C^\sm_*(X)$ is a subcomplex of $C_*(X)$ for $X=S^1, \mca{L}^a$. 
The isomorphism $H^\sm_*(\mca{L}^a) \cong H^\dR_*(\mca{L}^a)$ also preserves the rotation operators, since 
$H^\sm_* \to H^\dR_*$ is functorial and commutes with the cross product (Lemma \ref{150627_1}). 
Thus we have proved the following corollary.

\begin{cor}\label{150627_2} 
The isomorphism $H_*(\mca{L}^a) \cong H^\dR_*(\mca{L}^a)$ preserves the rotation operators. 
\end{cor} 

\subsection{Finite-dimensional approximations of $\mca{L}$}

Let us define $\mca{E}: \L \to \R$ by 
$\E(\gamma):= \int_{S^1} |\dot{\gamma}|^2$. 
Then $\E$ is smooth as a function on the differentiable space $\L$. 
For any $E \in (0,\infty]$, we define $\L^{a,E} \subset \L^a$ by 
$\L^{a,E}:= \{ \gamma \in \L^a \mid \mca{E}(\gamma)< E\}$. 

For any positive integer $N$, let us define 
\[
\mca{F}_N:= \{(x^j)_{0 \le j \le N} \in M^{N+1} \mid x^0=x^N\}, \quad 
f_N: \mca{L} \to \mca{F}_N; \, \gamma \mapsto (\gamma(j/N))_{0 \le j \le N}.
\]
$f_N$ is smooth as a map between differentiable spaces. 

For any $x=(x^j)_{0 \le j \le N} \in \mca{F}_N$, let us define ($d$ denotes the distance on $M$): 
\[
\len(x):= \sum_{1 \le j \le N} d(x^j, x^{j-1}), \qquad
\mca{E}(x):= N \sum_{1 \le j \le N} d(x^j, x^{j-1})^2. 
\]
For any $a, E \in (0,\infty]$, we define 
\[
\mca{F}^{a,E}_N:= \{ x \in \mca{F}_N \mid \len(x)<a, \, \mca{E}(x)<E\}.
\]
It is easy to see that $f_N(\mca{L}^{a,E}) \subset \mca{F}^{a,E}_N$ for any $a,E \in (0,\infty]$. 

Let $r_M$ be the injectivity radius of $M$ (since $M$ is closed, $r_M>0$). 
For any $p,q \in M$ such that $d(p,q)< r_M$, there exists a unique 
shortest geodesic $\gamma:[0,1] \to M$ such that 
$\gamma(0)=p$, $\gamma(1)=q$. 
We denote it by $\gamma_{pq}$. 

Suppose that $N$ is sufficiently large, such that $\sqrt{E/N} <r_M$. 
For any $x=(x^j)_{0 \le j \le N} \in \mca{F}^{a,E}_N$ and $0 \le j \le N-1$, 
there holds  $d(x^j, x^{j+1}) < \sqrt{E/N} <r_M$. 
For any integer $m \ge 1$, we define $y=(y^k)_{0 \le k \le mN} \in \mca{F}^{a,E}_{mN}$ by 
\[
y^{jm+l}:= \gamma_{x^j x^{j+1}} (l/m) \qquad ( 0 \le j \le N-1, \, 0 \le l \le m).
\]
For any $a' \ge a$ and $E' \ge E$, 
we define $i_m: \mca{F}^{a,E}_N \to \mca{F}^{a',E'}_{mN}$ by 
$i_m(x):=y$. 
This is a $C^\infty$-map between $C^\infty$-manifolds. 

\begin{lem}\label{150625_1} 
For any positive real numbers 
$a<a'$ and $E<E'$, there exists 
$N(a, E, a', E')$ such that the following holds: 
for any integer $N \ge N(a, E, a', E')$ and any integer $m \ge 1$, 
there exists a map 
$g: \mca{F}^{a,E}_N \to \mca{L}^{a',E'}$ such that the following diagram commutes up to homotopy
($i$ denotes the inclusion map): 
\[
\xymatrix{
\L^{a,E} \ar[r]^-{i}\ar[d]_-{f_N} &\L^{a',E'} \ar[d]^-{f_{mN}} \\
\mca{F}^{a,E}_N \ar[r]_-{i_m}\ar[ru]^-{g}&\mca{F}^{a',E'}_{mN}.
}
\]
The map $g$ and the homotopies are 
both continuous ($\L^{a,E}$, $\L^{a',E'}$ are equipped with the $C^\infty$-topology) and 
smooth (as maps between differentiable spaces). 
\end{lem} 

To prove Lemma \ref{150625_1}, 
we need the following preliminary Lemma \ref{150625_2}. 

We define $F: [0,1] \times \{(p,q) \in M^{\times 2} \mid d(p,q)<r_M \} \to M$ by 
$F(s, p,q):= \gamma_{pq}(s)$. 
For any $s \in [0,1]$, we define a map $F_s$ by $F_s(p,q)=F(s,p,q)$. 
For any $\gamma_0, \gamma_1 \in \L$ such that 
$\max_{t \in S^1} d(\gamma_0(t), \gamma_1(t))<r_M$, we define 
$\gamma_s \in \L$ by 
$\gamma_s(t):= F_s(\gamma_0(t), \gamma_1(t))$. 

\begin{lem}\label{150625_2}
For any $\delta>0$, 
there exists $r(\delta) \in (0, r_M)$ 
such that the following holds: 
if $\gamma_0, \gamma_1 \in \mca{L}$ satisfy $\max_{t \in S^1} d(\gamma_0(t), \gamma_1(t))< r(\delta)$, 
for any $0 \le s \le 1$ there holds
\[
\len(\gamma_s) \le (1+\delta)((1-s)\len(\gamma_0) + s \len(\gamma_1)), \quad
\E(\gamma_s) \le (1+\delta)^2((1-s) \E(\gamma_0) + s \E(\gamma_1)).
\]
\end{lem}
\begin{proof}
The following assertion is easy to prove by contradiction: 
there exists $r(\delta) \in (0, r_M)$ 
such that, if $p,q \in M$ satisfy $d(p,q)<r(\delta)$, then 
\[
|dF_s(v,w)| \le (1+\delta)((1-s)|v|+s|w|) \quad (\forall v \in T_pM, \, \forall w \in T_qM, \, \forall s \in [0,1]).
\]

Take $r(\delta)>0$ as above. Then, if $\gamma_0, \gamma_1 \in \L$ satisfy 
$\max_{t \in S^1} d(\gamma_0(t), \gamma_1(t))<r(\delta)$, there holds 
$|\dot{\gamma_s}(t)| \le (1+\delta)((1-s)|\dot{\gamma_0}(t)| + s| \dot{\gamma_1}(t)|)$ for any $s \in [0,1]$ and $t \in S^1$. 
The lemma follows from this estimate. 
\end{proof}

\begin{proof}[\textbf{Proof of Lemma \ref{150625_1}}]
Let us take $\delta>0$ so that 
$1+\delta  < a'/a$ and 
$(1+\delta)^4 < E'/E$. 
Let us also take a $C^\infty$-function $\mu:[0,1] \to [0,1]$ with the following properties: 
\begin{itemize}
\item $0 \le \mu'(t) \le 1+ \delta$ for any $t \in [0,1]$. 
\item $\mu(j/m)=j/m$ for any integer $0 \le j \le m$. 
\item $\mu$ is constant on some neighborhoods of $0$ and $1$. 
\end{itemize}

Let us take an integer $N$ so that $\sqrt{E/N} < r_M$. 
For any $x=(x^j)_{0 \le j \le N} \in \mca{F}^{a,E}_N$, we define $\gamma \in \L$ by 
\[
\gamma((j+t)/N):= \gamma_{x^j, x^{j+1}} (\mu(t)) \qquad (0 \le j \le N-1, \, 0 \le t \le 1).
\]
Then $\gamma$ satisfies 
$\len(\gamma) = \len (x) < a'$ and 
$\E(\gamma) \le (1+\delta)^2 \E(x) < E'$, and thus one can define 
$g: \mca{F}^{a,E}_N \to \L^{a',E'}$ by $g(x):= \gamma$. 
It is clear that $f_{mN} \circ g = i_m$. 
The map $g$ is both smooth (as a map between differentiable spaces) and 
continuous ($\L^{a', E'}$ is equipped with the $C^\infty$-topology). 
To see that $g$ is smooth, 
let us take $(U,\ph) \in \mca{P}(\mca{F}^{a,E}_N)$ and denote
$\ph(u) = (x^j(u))_{0 \le j \le N}$. 
Then 
\[
U \times S^1  \to M; \, (u, \theta) \mapsto \gamma_{x^j(u), x^{j+1}(u)} ( \mu(N\theta-j)) \quad( j/N \le \theta  \le (j+1)/N)
\]
is a $C^\infty$-map, which means that 
$(U, g \circ \ph) \in \mca{P}(\mca{L}^{a', E'})$. 
Hence $g$ is a smooth map of differentiable spaces. 

Let us define a homotopy between $i$ and $g \circ f_N$. 
For any $\gamma \in \L^{a,E}$, we set 
$\gamma_0:= \gamma$, and $\gamma_1:= g \circ f_N(\gamma)$. 
Then $\gamma_1$ satisfies 
$\len(\gamma_1) < a$ and 
$\E(\gamma_1) < (1+\delta)^2 E$, since 
\begin{align*}
& \len (\gamma_1) = \len (f_N(\gamma)) \le \len (\gamma) < a, \\
& \mca{E}(\gamma_1) \le (1+\delta)^2 \mca{E}(f_N(\gamma)) \le (1+\delta)^2 \mca{E}(\gamma) < (1+\delta)^2 E.
\end{align*} 

Let us suppose that $N$ is sufficiently large so that 
$\max_{t \in S^1} d(\gamma_0(t), \gamma_1(t)) \le r(\delta)$. 
Then, for any $0 \le s \le 1$, $\gamma_s$ satisfies 
$\len(\gamma_s) < (1+\delta) a < a'$ and 
$\E(\gamma_s) < (1+\delta)^4 E < E'$. 

Finally, take $\nu \in C^\infty(\R, [0,1])$ so that $\nu(s)=0$ for any $s \le 0$, 
and $\nu(s)=1$ for any $s \ge 1$. 
Then, 
$h: \L^{a, E} \times \R \to \L^{a', E'};  (\gamma,s) \mapsto \gamma_{\nu(s)}$ is a homotopy between $i$ and $g \circ f_N$. 
$h$ is both smooth and continuous.
To see that $h$ is smooth, 
let us take $(U,\ph) \in \mca{P}(\L^{a,E} \times \R)$, 
and set $\ph(u) = (\gamma(u), s(u))$. 
Then 
$U \times \R \to M; \,(u,t) \mapsto \gamma_0(u)(t)$ is of class $C^\infty$ since $i$ is smooth, and 
$U \times \R \to M; \,(u,t) \mapsto \gamma_1(u)(t)$ is of class $C^\infty$ since $g \circ f_N$ is smooth. 
Then 
\[ 
U \times \R \to M; \, (u,t) \mapsto \gamma_{\nu(s(u))}(t) = F_{\nu(s(u))} (\gamma_0(u)(t) , \, \gamma_1(u)(t))
\]
is of class $C^\infty$, which means that 
$(U, h \circ \ph) \in \mca{P}(\L^{a', E'})$. 
Hence $h$ is a smooth map of differentiable spaces. 
\end{proof}

\begin{rem}\label{rem:a=infty}
As is clear from the above proof, Lemma \ref{150625_1} holds even when 
$a=a'=\infty$. 
In this case, we denote $N(a,E,a',E')$ by $N(E,E')$. 
\end{rem} 

\subsection{Proof of Theorem \ref{150219_1}}

Let us take strictly increasing sequences of positive real numbers $(a_j)_{j \ge 1}$ and $(E_j)_{j \ge 1}$, 
such that 
$\lim_{j \to \infty} a_j =a$, $\lim_{j \to \infty}  E_j=\infty$. 
Then, $(\L^{a_j,E_j})_{j \ge 1}$ is an increasing sequence of open sets (with respect to the $C^\infty$-topology) of $\L^a$, 
and $\bigcup_{j \ge 1} \L^{a_j,E_j} = \L^a$. 
Thus, 
we have isomorphisms 
\[
\vlim_j H_*(\L^{a_j, E_j}) \cong H_*(\L^a), \qquad 
\vlim_j H^\sm_*(\L^{a_j,E_j}) \cong H^\sm_*(\L^a).
\]
Since the functionals $\len$ and $\mca{E}$ are approximately smooth, 
by Corollary \ref{150213_1}(ii) we obtain 
\[ 
H^\dR_*(\L^a) \cong \lim_{j \to \infty} H^\dR_*(\L^{a_j}) 
\cong \lim_{j \to \infty} (\lim_{j' \to \infty} H^\dR_*(\L^{a_j, E_{j'}})) \cong \lim_{j \to \infty} H^\dR_*(\L^{a_j, E_j}). 
\]
Therefore,  it is enough to prove that the following maps are isomorphisms:
\[
\vlim_j H^\sm_*(\L^{a_j,E_j}) \to \vlim_j H_*(\L^{a_j,E_j}), \qquad
\vlim_j H^\sm_*(\L^{a_j,E_j}) \to \vlim_j H^\dR_*(\L^{a_j,E_j}).
\]

Now we apply Lemma \ref{150625_1} for each $j$. 
Let us take a sequence $(N_j)_{j \ge 1}$ of positive integers so that 
$N_j \ge N(a_j, E_j, a_{j+1}, E_{j+1})$ and $N_j | N_{j+1}$ for every $j$. 
Then, there exists a map  
$g_j: \mca{F}^{a_j,E_j}_{N_j} \to \L^{a_{j+1}, E_{j+1}}$ such that the following diagram commutes up to homotopy: 
\[
\xymatrix{
\L^{a_j, E_j} \ar[r]\ar[d]_-{f_{N_j}} &\L^{a_{j+1}, E_{j+1}} \ar[d]^-{f_{N_{j+1}}}\\
\mca{F}^{a_j, E_j}_{N_j} \ar[r] \ar[ru]^-{g_j}&\mca{F}^{a_{j+1}, E_{j+1}}_{N_{j+1}}. 
}
\]
Then $\vlim_j H_*(f_{N_j}): \vlim_j H_*(\L^{a_j,E_j}) \to \vlim_j H_*(\mca{F}^{a_j,E_j}_{N_j})$ is an isomorphism, since 
$\vlim_j H_*(g_j)$ is its inverse. 
The same argument also works for $H^\sm_*$ and $H^\dR_*$, 
and we obtain isomorphisms 
\[
\vlim_j H^\sm_*(\L^{a_j, E_j}) \cong \vlim_j H^\sm_*(\mca{F}^{a_j,E_j}_{N_j}), \qquad
\vlim_j H^\dR_*(\L^{a_j, E_j}) \cong \vlim_j H^\dR_*(\mca{F}^{a_j,E_j}_{N_j}). 
\]
These isomorphisms fit into the following commutative diagram:
\[
\xymatrix{
\vlim_j H_*(\L^{a_j,E_j}) \ar[d]^-{\cong} & \vlim_j H^\sm_*(\L^{a_j, E_j}) \ar[d]^-{\cong} \ar[r] \ar[l] & \vlim_j H^\dR_*(\L^{a_j,E_j}) \ar[d]^-{\cong} \\
\vlim_j H_*(\mca{F}^{a_j,E_j}_{N_j}) & \vlim_j H^\sm_*(\mca{F}^{a_j,E_j}_{N_j}) \ar[r] \ar[l] & \vlim_j H^\dR_*(\mca{F}^{a_j,E_j}_{N_j}).
}
\]
Since $\mca{F}^{a_j,E_j}_{N_j}$ is an oriented \textit{finite-dimensional} $C^\infty$-manifold for every $j$, 
the maps 
$H^\sm_*(\mca{F}^{a_j,E_j}_{N_j}) \to H_*(\mca{F}^{a_j,E_j}_{N_j})$ and 
$H^\sm_*(\mca{F}^{a_j,E_j}_{N_j}) \to H^\dR_*(\mca{F}^{a_j,E_j}_{N_j})$ are isomorphisms. 
Thus, 
the horizontal maps on the second row are isomorphisms. 
Therefore, 
the horizontal maps on the first row are also isomorphisms. 
This completes the proof of Theorem \ref{150219_1}. 

\section{Moore loops with marked points} 

Constructions of string topology operations (e.g. the loop product) require (at least) two steps: 
\begin{itemize}
\item Fiber products of (de Rham) chains of the loop space with respect to evaluation maps. 
\item Concatenations of loops. 
\end{itemize} 
The differentiable space $\mca{L}=\mca{L}M$, which we studied in the previous section, is not adequate for either step. 
To avoid this trouble, in this section we introduce \textit{Moore loops with marked points}. 

We explain the plan of this section. Let $M$ denote a closed, oriented Riemannian manifold of dimension $d$. 
In Section 7.1, we introduce the space $\Pi$, which consists of Moore paths on $M$. 
In Section 7.2, we introduce the space $\bL_k$, which consists of Moore loops with $k+1$ marked points. 
We define two differentiable structures on $\bL_k$, and denote the resulting differentiable spaces by $\bL_k$ and $\bL_{k, \reg}$. 
The latter space $\bL_{k,\reg}$ is adequate to define fiber products on de Rham chain complexes, and we show that
the sequence of de Rham chain complexes $(C^\dR_{*+d}(\bL_{k, \reg}))_{k \ge 0}$ has a natural cyclic dg operad structure with a multiplication and a unit (see Definitions \ref{150624_4} and \ref{150723_1}). 
This is the operad $\mca{O}_M$ which appeared in Theorem \ref{160408_1}, 
namely $\mca{O}_M(k):= C^\dR_{*+d}(\bL_{k,\reg})$ for every $k \ge 0$. 
Sections 7.3--7.5 are devoted to proofs of Lemmas \ref{150624_2}, \ref{150624_3} and \ref{150626_5}, which we state at the end of Section 7.2. 
These technical lemmas play important roles in Section 8, which is devoted to proofs of results presented in Section 3.1.  

\subsection{Moore paths} 

Let $M$ be a closed, oriented $C^\infty$-manifold. 
We define the set of Moore paths on $M$ as follows: 
\[
\Pi:= \{(\gamma, T) \mid T \in \R_{\ge 0}, \, \gamma \in C^\infty([0,T], M), \, \gamma^{(m)}(0)=\gamma^{(m)}(T)=0 \,(\forall m \ge 1) \}.
\]
$\gamma^{(m)}$ denotes the $m$-th derivative of $\gamma$. 
The last condition is required to take concatenations of $C^\infty$-paths. 
We define $e_0, e_1: \Pi \to M$ by 
$e_0(\gamma, T):= \gamma(0)$, 
$e_1(\gamma, T):= \gamma(T)$. 
For any $p \in M$, let us define a map $\gamma_p$ and a Moore path $c_p \in \Pi$ by  
\[
\gamma_p: \{0\} \to M; \quad  0 \mapsto p, \qquad c_p:= (\gamma_p, 0) \in \Pi.
\]
The concatenation map 
$\Pi \fbp{e_1}{e_0} \Pi \to \Pi; \, (\gamma_0, T_0, \gamma_1, T_1) \mapsto (\gamma_0*\gamma_1, T_0+T_1)$
is defined by 
\[
\gamma_0*\gamma_1(t):= \begin{cases} \gamma_0(t) &(0 \le t \le T_0) \\ \gamma_1(t-T_0) &(T_0 \le t \le T_0+T_1).\end{cases}
\]
Functionals $\len$ and $\mca{E}$ on $\Pi$ are defined by 
\[
\len(\gamma, T):= \int_0^T |\dot{\gamma}|, \qquad \mca{E}(\gamma, T):= \int_0^T  | \dot{\gamma}|^2. 
\]
To define a differentiable structure on $\Pi$, we need the following definition.

\begin{defn}\label{150625_3} 
Let $X$ and $Y$ be $C^\infty$-manifolds, and $S$ be any subset of $X$. 
Then, a map $f: S \to Y$ is of class $C^\infty$, if there exists an open neighborhood $U$ of $S \subset X$ and a $C^\infty$-map $\bar{f}: U \to Y$ 
such that $\bar{f}|_S = f$.
\end{defn} 

We define two differentiable structures on $\Pi$, and denote the resulting differentiable spaces as $\Pi$ and $\Pi_\reg$. 
The set of plots $\mca{P}(\Pi)$ and $\mca{P}(\Pi_\reg)$ are defined as follows: 
\begin{itemize}
\item 
Let $U \in \mca{U}$ and $\ph: U \to \Pi$. 
Denote $\ph(u)= (\gamma(u), T(u))$ for any $u \in U$. 
Then, $(U,\ph) \in \mca{P}(\Pi)$ if
$T \in C^\infty(U)$ and 
\[
\tilde{U}:=\{(u, t) \mid u \in U, \, 0 \le t \le T(u)\} \to M; \quad (u,t) \mapsto \gamma(u)(t)
\]
is of class $C^\infty$ in the sense of Definition \ref{150625_3}, 
where $X=U \times \R$, $S = \tilde{U}$ and $Y= M$. 
\item 
$\mca{P}(\Pi_\reg)$ consists of $(U,\ph) \in \mca{P}(\Pi)$ such that 
$e_j \circ \ph: U \to M$ is a submersion for $j=0, 1$.
\end{itemize} 
The identity map $\id_\Pi: \Pi_\reg \to \Pi$ is smooth as a map of differentiable spaces. 
The functional $\mca{E}$ is smooth, and $\len$ is approximately smooth with both differentiable structures ($\Pi$ and $\Pi_\reg$). 
The goal of this subsection is to prove the next lemma. 

\begin{lem}\label{150626_1} 
The concatenation maps 
\[
\Pi \fbp{e_1}{e_0} \Pi \to \Pi, \qquad \Pi_\reg \fbp{e_1}{e_0} \Pi_\reg \to \Pi_\reg
\]
are smooth as maps of differentiable spaces. 
\end{lem}
\begin{rem}
Set theoretically, the two maps in Lemma \ref{150626_1} are same.
 However, the differentiable structures on the domain and the target of the maps are different. 
\end{rem} 

First we prove the following lemma, which is a subtle point of the proof. 

\begin{lem}\label{150626_2} 
Let $U$ be a $C^\infty$-manifold, and $T \in C^\infty(U, \R_{\ge 0})$. 
Let $\tilde{U}:= \{(u,t) \mid u \in U, 0 \le t \le T(u)\} \subset U \times \R$, 
and suppose that $f: \tilde{U} \to \R$ satisfies the following conditions: 
\begin{itemize}
\item $f$ is of class $C^\infty$ in the sense of Definition \ref{150625_3}. 
\item For any $u \in U$, $f(u, 0) = f(u, T(u)) = 0$. 
\item For any $u \in U$ such that $T(u)>0$ and an integer $m \ge 1$, $\partial_t^m f(u,0) = \partial_t^m f(u, T(u)) = 0$. 
\end{itemize} 
Then, $\tilde{f}: U \times \R \to \R$ defined by 
$\tilde{f}(u,t) = \begin{cases} f(u,t) &((u,t) \in \tilde{U}) \\ 0 &((u,t) \notin \tilde{U}) \end{cases}$ is of class $C^\infty$. 
\end{lem}
\begin{proof}
We may assume that $U$ is an open set in $\R^n$ with coordinates $x_1, \ldots, x_n$.
For any sequence of nonnegative integers $\alpha = (\alpha_0, \ldots, \alpha_n)$, we set 
$\partial^\alpha:= \partial_t^{\alpha_0} \partial_{x_1}^{\alpha_1} \cdots \partial_{x_n}^{\alpha_n}$. 

Since $f$ is of class $C^\infty$, there exists a $C^\infty$-function $F$, which is defined on an open neighborhood of $\tilde{U}$ and
$F|_{\tilde{U}} = f$. 
Let $U_0 := \{ u \in U \mid \text {$T \equiv 0$ on a neighborhood of $u$} \}$. 
We are going to prove that 
$\partial^\alpha F(u,0) = \partial^\alpha F(u, T(u))=0$ for any $\alpha$ and $u \notin U_0$. 
When $T(u)>0$ this is easy to check, since for any $v$ near $u$ one has 
\[ 
\partial_t^{\alpha_0} F(v,0) = \partial_t^{\alpha_0} f(v,0) = 0, \qquad
\partial_t^{\alpha_0} F(v, T(v)) = \partial_t^{\alpha_0} f(v, T(v)) = 0. 
\] 
Even when $T(u)=0$, 
there exists a sequence $(u_j)_{j \ge 1}$ such that 
$\lim_{j \to \infty} u_j=u$ and $T(u_j)>0$, 
since $u \notin U_0$. 
Thus, 
$\partial^\alpha F(u,0) = \lim_{j \to \infty} \partial^\alpha F(u_j,0)=0$. 

To prove the lemma, it is enough to prove the claim $\text{Cl}(m)$ for every integer $m \ge 0$: 
\begin{quote}
$\text{Cl}(m)$: 
For any $\alpha$ such that $\alpha_0 + \cdots + \alpha_n = m$, 
$\partial^\alpha \tilde{f}$ is totally differentiable, and 
$D(\partial^\alpha \tilde{f})(u,t)=0$ unless $0<t< T(u)$. 
\end{quote}

Let us prove $\text{Cl}(0)$. 
It is enough to check $D\tilde{f}(u,0) = D\tilde{f}(u, T(u))=0$ for any $u \in U$. 
We only prove $D\tilde{f}(u,0)=0$, since the proof for $D\tilde{f}(u, T(u))=0$ is parallel. 
The case $u \in U_0$ is easy since $\tilde{f} \equiv 0$ near $(u,0)$. 
Let us consider the case $u \notin U_0$. 
If $D\tilde{f}(u,0)=0$ does not hold, 
there exists a sequence $(u_j,t_j)_{j \ge 1}$ such that 
\[
\lim_{j \to \infty} (u_j, t_j) = (u,0), \qquad 
\liminf_j |\tilde{f}(u_j,t_j) - \tilde{f}(u,0)|/|(u_j,t_j)-(u,0)| >0.
\]
Since $\tilde{f}(u,0)=0$, one has $\tilde{f}(u_j,t_j) \ne 0$ for sufficiently large $j$. 
Then $0 < t_j < T(u_j)$, 
and thus, 
$F(u_j,t_j) = \tilde{f}(u_j, t_j)$. 
On the other hand $F(u,0)=0$,
and thus, 
$\liminf_j |F(u_j,t_j) - F(u,0)|/|(u_j,t_j)-(u,0)| >0$.
This contradicts $DF(u,0)=0$, and $\text{Cl}(0)$ is proved. 

Let us prove $\text{Cl}(m-1) \implies \text{Cl}(m)$. 
It is enough to check
$D(\partial^\alpha \tilde{f})(u,0) = D(\partial^\alpha \tilde{f})(u, T(u)) = 0$ for any $u \in U$. 
We only prove $D(\partial^\alpha \tilde{f})(u,0)=0$, since the proof for $D(\partial^\alpha \tilde{f})(u, T(u))=0$ is parallel. 
The case $u \in U_0$ is easy since $\tilde{f} \equiv 0$ near $(u,0)$. 
Let us consider the case $u \notin U_0$. 
If $D(\partial^\alpha \tilde{f})(u,0)=0$ does not hold, 
there exists a sequence $(u_j,t_j)_{j \ge 1}$ such that 
$\lim_{j \to \infty} (u_j, t_j) = (u,0)$ and 
\[
\liminf_j |\partial^\alpha \tilde{f}(u_j,t_j) - \partial^\alpha \tilde{f}(u,0)|/|(u_j,t_j) - (u,0)| >0.
\]
By $\text{Cl}(m-1)$, $\partial^\alpha \tilde{f}(u,0)=0$. 
Thus $\partial^\alpha \tilde{f}(u_j,t_j) \ne 0$ for sufficiently large $j$. 
By $\text{Cl}(m-1)$, this implies $0 < t_j < T(u_j)$, and thus, 
$\partial^\alpha F(u_j, t_j) = \partial^\alpha \tilde{f}(u_j, t_j)$. 
On the other hand $\partial^\alpha F(u,0)=0$, 
and thus, 
$\liminf_j |\partial^\alpha F(u_j,t_j) - \partial^\alpha F(u,0)|/|(u_j,t_j) - (u,0)| >0$. 
This contradicts $D(\partial^\alpha F)(u,0) =0$, and $\text{Cl}(m)$ is proved. 
\end{proof} 

\begin{proof}[\textbf{Proof of Lemma \ref{150626_1}}]
We first prove that $\Pi \fbp{e_1}{e_0} \Pi \to \Pi$ is smooth. 
Let $(U, \ph) \in \mca{P}(\Pi \fbp{e_1}{e_0} \Pi)$, 
and let $\ph(u)$ denote 
$(\gamma_0(u), T_0(u), \gamma_1(u), T_1(u))$. 
We need to show that 
\[
U \to \Pi; u \mapsto (\gamma_0(u)*\gamma_1(u), T_0(u)+T_1(u))
\]
is a plot on $\Pi$. 
It is enough to show that 
\[
\Gamma_{01}: \{(u,t) \mid u \in U, \, 0 \le t \le T_0(u)+T_1(u)\} \to M; \quad (u,t) \mapsto \gamma_0(u)*\gamma_1(u)(t)
\]
is of class $C^\infty$ in the sense of Definition \ref{150625_3}. 
We may assume that $M$ is embedded in $\R^N$ for some integer $N$. 
For $j=0,1$, we define $\delta_j: U \times \R \to \R^N$ by 
\[
\delta_j(u,t):= \begin{cases} \partial_t \gamma_j(u,t) &(0 \le t \le T_j(u)) \\  0 &(\text{otherwise}). \end{cases}
\]
Since $\partial^m_t \gamma_j(u,0) = \partial^m_t \gamma_j(u,T_j(u))=0$ for any $u \in U$ and $m \ge 1$, 
Lemma \ref{150626_2} shows $\delta_j \in C^\infty(U \times \R, \R^N)$. 
Let us define $\delta_{01}, \widetilde{\Gamma_{01}}\in C^\infty(U \times \R, \R^N)$ by
\[
\delta_{01}(u,t):= \delta_0(u,t) + \delta_1(u, t- T_0(u)), \qquad 
\widetilde{\Gamma_{01}}(u,t):= \gamma_0(u)(0) + t \int_0^1 \delta_{01}(u, ts) \, ds.
\]
Then, it is easy to see that $\widetilde{\Gamma_{01}}(u,t)=\Gamma_{01}(u,t)$ for any $u \in U$ and $0 \le t \le T_0(u)+T_1(u)$. 
This shows that $\Gamma_{01}$ is of class $C^\infty$, hence $\Pi \fbp{e_1}{e_0} \Pi \to \Pi$ is smooth. 

Now, it is easy to see that $\Pi_\reg \fbp{e_1}{e_0} \Pi_\reg \to \Pi_\reg$ is smooth, since any $(\Gamma_0, \Gamma_1) \in \Pi \fbp{e_1}{e_0} \Pi$ satisfies
$e_0(\Gamma_0*\Gamma_1) = e_0(\Gamma_0)$ and $e_1(\Gamma_0*\Gamma_1) = e_1(\Gamma_1)$. 
\end{proof} 

\subsection{Moore loops with marked points} 

For any integer $k \ge 0$, we define
\[
\bL_k:= \{ (\Gamma_0, \ldots, \Gamma_k) \in \Pi^{k+1} \mid e_1(\Gamma_j)=e_0(\Gamma_{j+1})\,(0 \le j \le k-1), \, e_1(\Gamma_k)=e_0(\Gamma_0) \}. 
\]
In particular, $\bL_0 = \{ \Gamma \in \Pi \mid e_1(\Gamma)=e_0(\Gamma)\}$. 
For any $0 \le i \le k$, we define
\[
e_i: \bL_k \to M; \quad  (\Gamma_0, \ldots, \Gamma_k) \mapsto e_0(\Gamma_i).
\]
We also define $i_k: M \to \bL_k$ by 
$i_k(p):= (c_p, \ldots, c_p)$ for any $p \in M$. 
Recall that $c_p$ denotes the constant Moore path at $p$. 

\begin{rem}\label{150701_1}
We can identify $\bL_k$ with 
\begin{align*} 
&\{(\gamma, t_1, \ldots, t_k, T) \mid T \in \R_{\ge 0}, \, \gamma \in C^\infty([0,T], M), \, 0 \le t_1 \le \cdots \le t_k \le T, \\ 
&\qquad \gamma(0)=\gamma(T), \quad \gamma^{(m)}(0)=\gamma^{(m)}(T)= \gamma^{(m)}(t_j)=0 \, (\forall m \ge 1, \, 1 \le \forall j \le k) \}. 
\end{align*} 
Indeed, for any $(\gamma_j, T_j)_{0 \le j \le k} \in \bL_k$, one can assign 
$(\gamma_0* \cdots * \gamma_k, T_1, T_1+T_2, \ldots, T_1+\cdots+ T_k)$. 
\end{rem} 

Recall that we defined two differentiable structures on $\Pi$, and denote the resulting differentiable spaces $\Pi$ and $\Pi_\reg$. 
Since $\bL_k$ is a subset of $\Pi^{k+1}$, we can define two differentiable structures on $\bL_k$ (see Example \ref{150623_1} (iii) and (iv)). 
We denote the resulting differentiable spaces by $\bL_k$ and $\bL_{k, \reg}$. 

Let us define $\len: \bL_k \to \R$ by 
$\len(\Gamma_0, \ldots, \Gamma_k):= \len(\Gamma_0) + \cdots + \len(\Gamma_k)$. 
The function $\len$ is approximately smooth with both differentiable structures ($\bL_k$ and $\bL_{k, \reg}$). 

For any $a \in (0,\infty]$, we define 
$\bL^a_k:= \{(\Gamma_0, \ldots, \Gamma_k) \in \bL_k \mid \len(\Gamma_0, \ldots, \Gamma_k) < a\}$. 
We define differentiable structures on $\bL^a_k$ as a subspace of $\bL_k$ and $\bL_{k, \reg}$, and denote the resulting differentiable spaces as 
$\bL^a_k$ and $\bL^a_{k, \reg}$. 
The chain maps $C^\dR_*(\bL^a_{k,\reg}) \to C^\dR_*(\bL_{k, \reg})$ and $C^\dR_*(\bL^a_k) \to C^\dR_*(\bL_k)$ are injective. 

For any $1 \le i \le k$ and $l \ge 0$, we define a map 
\[
c_{k,i,l}: \bL_{k, \reg} \fbp{e_i}{e_0} \bL_{l, \reg} \to \bL_{k+l-1, \reg} 
\]
by 
\begin{align*}
&c_{k,i,l}(\Gamma_0, \ldots, \Gamma_k, \Gamma'_0, \ldots, \Gamma'_l):=  \\
&\begin{cases} (\Gamma_0, \ldots, \Gamma_{i-2}, \Gamma_{i-1}*\Gamma'_0, \Gamma'_1, \ldots, \Gamma'_{l-1}, \Gamma'_l *\Gamma_i, \Gamma_{i+1}, \ldots, \Gamma_k) &(l \ge 1), \\
(\Gamma_0, \ldots, \Gamma_{i-2}, \Gamma_{i-1}*\Gamma'_0*\Gamma_i, \Gamma_{i+1}, \ldots, \Gamma_k) &(l=0). \end{cases} 
\end{align*} 
Then, $c_{k,i,l}$ is a smooth map by Lemma \ref{150626_1}. 
Also, 
\[
R_k : \bL_{k,\reg} \to \bL_{k, \reg}; \qquad (\Gamma_0, \ldots, \Gamma_k) \mapsto (\Gamma_1, \ldots, \Gamma_k, \Gamma_0)
\]
is a smooth map. 

For any integer $k \ge 0$, let 
$\mca{O}_M(k):= C^\dR_{*+d}(\bL_{k, \reg})$. 
As described below, $\mca{O}_M:= (\mca{O}_M(k))_{k \ge 0}$ has a structure of a nonsymmetric cyclic dg operad with a multiplication and a unit. 

\begin{itemize}
\item 
For any $1 \le i \le k$ and $l \ge 0$, we define 
$\circ_i: \mca{O}_M(k)_* \otimes \mca{O}_M(l)_* \to \mca{O}_M(k+l-1)_*$ by 
\[
u \circ_i v:= (c_{k,i,l})_*(u \fbp{e_i}{e_0} v).
\]
$u \fbp{e_i}{e_0} v$ denotes the fiber product on de Rham chain complexes, which was defined in Section 4.3. 
\item
For any $k \ge 0$, 
we define $\tau_k: \mca{O}_M(k)_* \to \mca{O}_M(k)_*$ by 
$\tau_k:= (R_k)_*$. 
We define $\tau_0$ to be the identity map on $\mca{O}_M(0)_*$. 
\item
Let us take $M' \in \mca{U}$ and an orientation-preserving diffeomorphism $\ph: M' \to M$. 
Then, we define 
\begin{align*} 
\ep&:= [(M', i_0 \circ \ph, 1)] \in  \mca{O}_M(0)_0,  \quad 
1_{\mca{O}_M}:= [(M', i_1 \circ \ph, 1)] \in \mca{O}_M(1)_0, \\ 
\mu&:= [(M', i_2 \circ \ph, 1)] \in \mca{O}_M(2)_0. 
\end{align*} 
It is easy to check that these elements are well-defined, i.e. do not depend on the choices of $M'$ and $\ph$. 
Also, $1_{\mca{O}_M}$ and $\mu$ are cyclically invariant. 
\end{itemize} 

\begin{rem}\label{160626_1} 
For each integer $k \ge 0$, 
let us define a filtration  $(F^a \mca{O}_M(k))_{a \in (0,\infty]}$ on $\mca{O}_M(k)$ by 
$F^a \mca{O}_M(k)_* := C^\dR_{*+d}(\bL^a_{k, \reg})$. 
It is easy to verify 
\begin{align*} 
|x \circ_i y| &\le |x| + |y| \qquad(\forall x \in \mca{O}_M(k), \, \forall y \in \mca{O}_M(l), \, 1 \le \forall  i \le k), \\ 
|\tau_k x| &= |x| \qquad\qquad (\forall x \in \mca{O}_M(k)), \\ 
|\mu|&= |\ep| = 0 
\end{align*} 
as stated in Proposition \ref{160420_2} (i). 
In particular, $(F^a\mca{O}_M(k))_{k \ge 0}$ is a cocyclic chain complex for every $a \in (0,\infty]$.
\end{rem} 

The rest of this section is devoted to proofs of the following lemmas. 

\begin{lem}\label{150624_2}
For any $k \ge 0$, the identity map $\bL^a_{k,\reg} \to \bL^a_k$ induces an isomorphism
$H^\dR_*(\bL^a_{k,\reg}) \cong H^\dR_*(\bL^a_k)$. 
\end{lem}

Let us recall the notation 
$\mca{L} = \mca{L}M:= C^\infty(S^1, M)$ and 
$\mca{L}^a:= \{ \gamma \in \mca{L} \mid \len(\gamma)<a\}$. 

\begin{lem}\label{150624_3} 
For any $k \ge 0$, let us define 
\[
\mca{L}^a_k:= \{(\gamma,t_1, \ldots, t_k) \in \L^a \times \Delta^k \mid \gamma^{(m)}(0)= \gamma^{(m)}(t_j)=0 \quad(1 \le \forall j \le k, \, \forall m \ge 1) \}, 
\]
and consider the differentiable structure on $\mca{L}^a_k$ as a subset of $\L^a \times \Delta^k$. 
Then, the inclusion map $\mca{L}^a_k \to \mca{L}^a \times \Delta^k$ induces an isomorphism 
$H^\dR_*(\mca{L}^a_k) \cong H^\dR_*(\mca{L}^a \times \Delta^k)$. 
\end{lem} 

\begin{lem}\label{150626_5} 
For any $k \ge 0$, let us define the map 
\[
\mca{L}^a_k \to \bL^a_k; \quad (\gamma, t_1, \ldots, t_k) \mapsto (\gamma_j, T_j)_{0 \le j \le k}
\]
by $T_j:= t_{j+1}-t_j$, $\gamma_j(t):= \gamma(t-t_j)$ (we set $t_0=0$, $t_{k+1}=1$). 
Then, the map $\mca{L}^a_k \to \bL^a_k$ is smooth, 
and induces an isomorphism 
$H^\dR_*(\mca{L}^a_k) \cong H^\dR_*(\bL^a_k)$. 
\end{lem} 

Summarizing these lemmas, 
we have the following zig-zag of quasi-isomorphisms: 
\begin{equation}\label{150628_2}  
\xymatrix{ 
C^\dR_*(\bL^a_{k,\reg}) \ar[r] & C^\dR_*(\bL^a_k) & C^\dR_*(\mca{L}^a_k) \ar[r]\ar[l] & C^\dR_*(\mca{L}^a \times \Delta^k). \\ 
}
\end{equation}
Each of these four sequences has the natural structure of a cocylic chain complex; 
for example, $\tau_k$ on $C^\dR_*(\mca{L}^a \times \Delta^k)$ is induced by the smooth map 
\[
\mca{L}^a \times \Delta^k \to \mca{L}^a \times \Delta^k; \quad 
(\gamma, t_1,\ldots, t_k) \mapsto (\gamma^{t_1}, t_2-t_1, \ldots, t_k-t_1, 1-t_1)
\] 
where $\gamma^{t_1}(\theta):= \gamma(\theta-t_1)$. 
The diagram (\ref{150628_2}) induces quasi-isomorphisms of these cocyclic chain complexes. 

\subsection{Proof of Lemma \ref{150624_2}}

Let us take a strictly increasing sequence of positive real numbers $(a_j)_{j \ge 1}$, such that $\lim_{j \to \infty} a_j=a$. 
Since the length functional is approximately smooth on $\bL_k$ and $\bL_{k, \reg}$,
Corollary \ref{150213_1} (ii) implies 
\[
\vlim_j  H^\dR_*(\bL^{a_j}_{k,\reg}) \cong H^\dR_*(\bL^a_{k,\reg}), \qquad
\vlim_j  H^\dR_*(\bL^{a_j}_k) \cong H^\dR_*(\bL^a_k).
\]
Now, the key technical step is the next lemma. 

\begin{lem}\label{150626_3}
For any integer $j \ge 1$, 
there exists a chain map 
$J: C^\dR_*(\bL^{a_j}_k) \to C^\dR_*(\bL^{a_{j+1}}_{k, \reg})$ such that 
the following diagram commutes up to chain homotopy: 
\[
\xymatrix{
C^\dR_*(\bL^{a_j}_{k,\reg}) \ar[r]^-{(\id_j)_*} \ar[d]_-{(I_\reg)_*}&C^\dR_*(\bL^{a_j}_k )\ar[d]^-{I_*}\ar[ld]^-{J} \\
C^\dR_*(\bL^{a_{j+1}}_{k,\reg}) \ar[r]_-{(\id_{j+1})_*}& C^\dR_*(\bL^{a_{j+1}}_k) .
}
\]
In the above diagram, $\id_j$, $\id_{j+1}$ are identity maps, and $I$, $I_\reg$ are inclusion maps. 
\end{lem}

Lemma \ref{150626_3} implies $\vlim_j H^\dR_*(\bL^{a_j}_{k,\reg}) \cong \vlim_j H^\dR_*(\bL^{a_j}_k)$, 
then we obtain $H^\dR_*(\bL^a_{k,\reg}) \cong H^\dR_*(\bL^a_k)$, that is, Lemma \ref{150624_2}. 

\begin{proof}[\textbf{Proof of Lemma \ref{150626_3}}]
The proof is quite similar to the proof of Lemma \ref{150210_5}, except that we have to pay attention to lengths of loops. 

Let $\delta:= a_{j+1}/a_j-1$. 
By Lemma \ref{150210_6} and Remark \ref{150626_4}, 
there exists an integer $D$ and a $C^\infty$-map
$F: M \times \R^D \to M$ such that 
\begin{itemize}
\item For any $z \in \R^D$, $F_z: M \to M; \, x \mapsto F(x,z)$ is a diffeomorphism. Moreover, there holds $|dF_z(v)| \le (1+\delta)|v|$ for any $v \in TM$.
\item $F_{(0,\ldots,0)} = \id_M$.
\item For any $x \in M$, $\R^D \to M; z \mapsto F(x,z)$ is a submersion.
\end{itemize} 

Let us define $\mca{F}: \bL^{a_j}_k \times \R^D \to \bL^{a_{j+1}}_k$ by 
\[
\mca{F}(\Gamma_0,  \ldots, \Gamma_k, z):= (F_z \circ \Gamma_0, \ldots, F_z \circ \Gamma_k).
\]
Then, 
$(U \times \R^D, \mca{F} \circ (\ph \times \id_{\R^D})) \in \mca{P}(\bL^{a_{j+1}}_{k,\reg})$ for any $(U,\ph) \in \mca{P}(\bL^{a_j}_k)$. 
Let us take $\nu \in \mca{A}^D_c(\R^D)$ such that $\int_{\R^D} \nu=1$. 
It is easy to see that 
\[
J: C^\dR_*(\bL^{a_j}_k) \to C^\dR_*(\bL^{a_{j+1}}_{k, \reg}); \quad 
[(U,\ph, \omega)] \mapsto
[(U \times \R^D, \mca{F} \circ(\ph \times \id_{\R^D}), \omega \times \nu)]
\]
is a well-defined chain map. 
We show that $J$ satisfies the requirement in Lemma \ref{150626_3}. 

To show that $J \circ (\id_j)_*$ and $(I_\reg)_*$ are homotopic, 
we take $a, b \in C^\infty(\R, [0,1])$ so that 
\begin{itemize}
\item $a(s) =0$ for any $s \le 0$, and $a(s) =1$ for any $s \ge 1$. 
\item $\supp \, b$ is compact, and there exists $\ep>0$ such that $b(s)=1$ for any $s \in [-\ep, 1+\ep]$. 
\end{itemize} 

For any $(U,\ph) \in \mca{P}(\bL^{a_j}_{k, \reg})$, 
we define $(U \times \R^D \times \R, \Phi) \in \mca{P}(\bL^{a_{j+1}}_k)$ by 
\[
\Phi(u, z, s):= \mca{F} (\ph(u), a(s) z). 
\] 
Let us show that the pair $(U \times \R^D \times \R, \Phi)$  is a plot of $\bL_{k, \reg}$, 
namely that $e_i \circ \Phi: U \times \R^D \times \R \to M$ is a submersion for every $0 \le i \le k$. 
It is easy to check that $e_i \circ \Phi(u,z,s)= F_{a(s)z} \circ e_i \circ \ph(u)$. 
Since $(U,\ph) \in \mca{P}(\bL_{k, \reg})$, $e_i \circ \ph: U \to M$ is a submersion. 
On the other hand, $F_{a(s)z}$ is a diffeomorphism on $M$. 
Thus, $e_i \circ \Phi$ is a submersion. 
Hence $(U \times \R^D \times \R, \Phi) \in \mca{P}(\bL^{a_{j+1}}_{k, \reg})$. 

A computation similar to the proof of Lemma \ref{150210_5} shows that the linear map 
\[
K: C^\dR_*(\bL^{a_j}_{k, \reg}) \to C^\dR_{*+1}(\bL^{a_{j+1}}_{k, \reg}); \quad
[(U,\ph,\omega)] \mapsto (-1)^{|\omega|+D} [(U \times \R^D  \times \R, \Phi, \omega \times \nu \times b(s))]
\]
is well-defined, and satisfies 
$\partial K + K \partial  = (I_\reg)_* - J \circ (\id_j)_*$. 

Similar arguments show that $(\id_{j+1})_* \circ J$ is homotopic to $I_*$. 
The homotopy operator is given by exactly the same formula as $K$. 
This case is easier, since we do not have to care about the submersion condition. 
\end{proof} 

\subsection{Proof of Lemma \ref{150624_3}}

Let us take strictly increasing sequences $(a_j)_{j \ge 1}$, $(E_j)_{j \ge 1}$ of positive real numbers, 
such that $\lim_{j \to \infty}  a_j = a$, $\lim_{j \to \infty} E_j=\infty$. 
We set 
\[
\mca{L}^{a_j, E_j}:= \{ \gamma \in \mca{L} \mid \len(\gamma)<a_j, \, \E(\gamma)<E_j\}, \qquad
\mca{L}^{a_j, E_j}_k:= \mca{L}^{a_j}_k \cap \mca{L}^{a_j, E_j} \times \Delta^k.
\]

\begin{lem}\label{150630_1} 
For every $j \ge 1$, 
there exists a smooth map $J: \L^{a_j,E_j}  \times \Delta^k \to \L^{a_{j+1},E_{j+1}}_k$  such that
the following diagram commutes up to smooth homotopy:
\[
\xymatrix{
\L^{a_j,E_j}_k \ar[r] \ar[d] & \L^{a_j,E_j} \times \Delta^k \ar[ld]^-J \ar[d] \\
\L^{a_{j+1}, E_{j+1}}_k \ar[r] & \L^{a_{j+1}, E_{j+1}} \times \Delta^k. 
}
\]
All maps other than $J$ are inclusion maps. 
\end{lem}

Assuming Lemma \ref{150630_1}, we can prove Lemma \ref{150624_3} by 
\[
H^\dR_*(\mca{L}^a_k) \cong \vlim_j H^\dR_*(\mca{L}^{a_j, E_j}_k) \cong \vlim_j H^\dR_*(\mca{L}^{a_j, E_j} \times \Delta^k) \cong H^\dR_*(\mca{L}^a \times \Delta^k).
\]
To prove Lemma \ref{150630_1}, we need the following sublemma. 

\begin{lem}\label{150630_2} 
For any $\delta>0$, 
there exists a $C^\infty$-map 
$\mu: \Delta^k \times [0,1] \to [0,1]$ such that the following properties hold for any $(t_1,\ldots, t_k)  \in \Delta^k$. 
\begin{itemize}
\item[(i): ] $\mu(t_1,\ldots, t_k,0)=0$, $\mu(t_1,\ldots,t_k, 1)=1$.
\item[(ii): ] $\partial_\theta \mu(t_1,\ldots,t_k, \theta) \in [0,1+\delta]$ for any $\theta \in [0,1]$. 
\item[(iii): ] $|\mu(t_1,\ldots, t_k, \theta) - \theta| \le \delta$ for any $\theta \in [0,1]$. 
\item[(iv): ] $\partial_\theta^m  \mu(t_1,\ldots, t_k, \theta)=0$ for any integer $m \ge 1$ and $\theta \in \{0,t_1, \ldots, t_k, 1\}$. 
\end{itemize} 
\end{lem}
\begin{proof}
We may assume $\delta<1$. 
Let us take $c \in (0, \delta/4(k+2))$. 
We also take 
$\kappa \in C^\infty(\R, [0,1])$ such that 
$\kappa^{(m)}(0)=0$ for any integer $m \ge 0$, 
and $\kappa(\theta)=1$ if $|\theta| \ge c$. 
For any $t \in [0,1]$, we set $\kappa_t(\theta):= \kappa(\theta-t)$, and 
we define $\nu, \tilde{\nu}, \mu \in C^\infty(\Delta^k \times [0,1])$ by 
\begin{align*}
\nu(t_1,\ldots, t_k, \theta)&: = \kappa_0(\theta) \cdot  \kappa_1(\theta) \cdot \prod_{1 \le j \le k} \kappa_{t_j}(\theta) ,  \\
\tilde{\nu}(t_1,\ldots,t_k, \theta)&:= \int_0^\theta \nu(t_1,\ldots, t_k, \theta') \, d \theta',  \\
\mu(t_1,\ldots, t_k, \theta)&:= \tilde{\nu}(t_1,\ldots,t_k, \theta)/ \tilde{\nu}(t_1,\ldots, t_k, 1). 
\end{align*}
It is clear that $\mu$ satisfies (i) and (iv). 
It is also easy to check 
$\max\{ 0, \theta -\delta/2\} \le \tilde{\nu}(t_1, \ldots, t_k, \theta) \le \theta$, 
then (ii) and (iii) are verified as 
\begin{align*}
\partial_\theta \mu(t_1, \ldots, t_k, \theta) &\le \nu(t_1, \ldots, t_k, \theta)/(1-\delta/2) \le 1+\delta, \\ 
\mu(t_1,\ldots,t_k,\theta)& \ge \tilde{\nu}(t_1,\ldots, t_k, \theta) \ge \theta -\delta/2,  \\
\mu(t_1,\ldots,t_k,\theta)& \le \theta/(1-\delta/2) \le \theta(1+\delta). 
\end{align*} 
\end{proof}

\begin{proof}[\textbf{Proof of Lemma \ref{150630_1}}]

Let us fix $E \in (E_j, E_{j+1})$,
and take $\delta>0$ so that 
\[
1+\delta < (E/E_j)^{1/2}, \qquad
(E_j \delta)^{1/2} < r \big( \min \{a_{j+1}/a_j, (E_{j+1}/E)^{1/2}\} -1 \big).
\]
For the definition of the function $r(\cdots )$ in the second inequality, 
see Lemma \ref{150625_2}. 

We take 
$\mu: \Delta^k \times [0,1] \to [0,1]$ as in Lemma \ref{150630_2}. 
For any $\gamma \in \L^{a_j,E_j}$ and 
$(t_1,\ldots, t_k) \in \Delta^k$, we define 
$\gamma_{t_1,\ldots,t_k} \in \L$ by 
$\gamma_{t_1,\ldots,t_k}(\theta):= \gamma (\mu(t_1,\ldots,t_k,\theta))$ (this is well-defined by Lemma \ref{150630_2} (i)). 
Then $\len(\gamma_{t_1,\ldots,t_k}) = \len(\gamma)$, and 
$\E(\gamma_{t_1,\ldots,t_k}) < E$ by Lemma \ref{150630_2} (ii). 
By Lemma \ref{150630_2} (iv), for any $m \ge 1$ and $\theta \in \{0,t_1, \ldots, t_k\}$, there holds 
$\gamma^{(m)}_{t_1,\ldots, t_k}(\theta)=0$. 
Therefore, 
$(\gamma_{t_1,\ldots,t_k}, t_1,\ldots,t_k)  \in \L^{a_j,E}_k$. 

Since $\L^{a_j,E}_k \subset \L^{a_{j+1}, E_{j+1}}_k$, 
one can define $J$ by 
\[
J: \L^{a_j, E_j} \times \Delta^k \to \L^{a_{j+1}, E_{j+1}}_k; \quad (\gamma, t_1,\ldots,t_k) \mapsto (\gamma_{t_1,\ldots,t_k}, t_1,\ldots, t_k).
\]
For any $\gamma \in \L^{a_j, E_j}$, we can prove 
\[
\max_{\theta \in S^1} d(\gamma(\theta), \gamma_{t_1,\ldots,t_k}(\theta)) \le (E_j \delta)^{1/2}. 
\]
To prove this inequality, for any $\theta \in [0,1]$ let $I_\theta$ denote an interval in $[0,1]$ 
such that $\partial I_\theta = \{\theta, \mu(t_1, \ldots, t_k, \theta)\}$. 
Then, using Cauchy-Schwarz inequality we obtain 
\begin{align*} 
d(\gamma(\theta), \gamma_{t_1,\ldots,t_k}(\theta)) 
&\le \int_{I_\theta} |\dot{\gamma}(t)| \, dt 
\le \biggl( \int_{I_\theta} 1 \, dt \biggr)^{1/2} \biggl( \int_{I_\theta} |\dot{\gamma}(t)|^2 \, dt \biggr)^{1/2} \\
&\le |I_\theta|^{1/2} \cdot |\mca{E}(\gamma)|^{1/2} \le (E_j \delta)^{1/2}.
\end{align*} 
Note that $|I_\theta| \le \delta$ follows from Lemma \ref{150630_2}(iii), 
and $\mca{E}(\gamma) < E_j$ since $\gamma \in \L^{a_j, E_j}$. 

For any $s \in [0,1]$, let 
$\gamma_{s,t_1,\ldots,t_k}(\theta):= F_s(\gamma(\theta), \gamma_{t_1,\ldots,t_k}(\theta))$. 
The map $F_s$ is defined right before Lemma \ref{150625_2}. 
Applying Lemma \ref{150625_2} for $\min\{ a_{j+1}/a_j, (E_{j+1}/E)^{1/2}\} -1$, 
we obtain  
$\gamma_{s,t_1,\ldots,t_k} \in \L^{a_{j+1}, E_{j+1}}$ for any $s \in [0,1]$. 
If $(\gamma, t_1,\ldots, t_k) \in \L^{a_j, E_j}_k$, there holds 
$\gamma_{s,t_1,\ldots,t_k}^{(m)}(\theta)=0$ for any 
$m \ge 1$, 
$0 \le s \le 1$, and 
$\theta \in \{0, t_1, \ldots, t_k\}$. 
Thus, 
$(\gamma_{s,t_1,\ldots,t_k},t_1,\ldots,t_k) \in \L^{a_{j+1}, E_{j+1}}_k$
for any $0 \le s \le 1$. 

Let us take $\alpha \in C^\infty(\R, [0,1])$ so that 
$\alpha(s)=0$ for $s \le 0$ and $\alpha(s)=1$ for $s \ge 1$. 
We define $H: \L^{a_j,E_j} \times \Delta^k \times \R  \to \L^{a_{j+1},E_{j+1}} \times \Delta^k$ by 
\[
H(\gamma, t_1,\ldots, t_k, s):= (\gamma_{\alpha(s),t_1,\ldots, t_k} , t_1,\ldots, t_k).
\]
Obviously, this is a smooth homotopy between $J$ and the 
inclusion map $\L^{a_j,E_j} \times \Delta^k \to \L^{a_{j+1}, E_{j+1}} \times \Delta^k$. 
Finally, the restriction of $H$ to $\L^{a_j,E_j}_k \times \R$ is a smooth homotopy between 
$J|_{\L^{a_j,E_j}_k}$ and the inclusion map $\L^{a_j,E_j}_k \to \L^{a_{j+1},E_{j+1}}_k$. 
\end{proof}

\subsection{Proof of Lemma \ref{150626_5}}
It is easy to see that the map $\mca{L}^a_k \to \bL^a_k$ is smooth, 
thus it is enough to show that the map induces an isomorphism on $H^\dR_*$. 
The proof consists of three steps. 

\textbf{Step 1.} 
We identify $\bL_k$ with the set consisting of tuples $(\gamma, t_1, \ldots, t_k, T)$ (see Remark \ref{150701_1}). 
Setting $p: [0,1] \to \R/\Z$ by $p(\theta):=[\theta]$, 
the  map $\mca{L}^a_k \to \bL^a_k$ is given by 
$(\gamma, t_1, \ldots, t_k) \mapsto (\gamma \circ p, t_1, \ldots, t_k, 1)$. 
The image of this map is contained in 
\[
\bL^a_{k, T>0}:=\{(\gamma, t_1, \ldots, t_k, T) \in \bL^a_k \mid T>0\}.
\]
Now, $\mca{L}^a_k \to \bL^a_{k, T>0}$ induces an isomorphism on $H^\dR_*$, since 
\[
\bL^a_{k, T>0} \to \mca{L}^a_k; \quad  (\gamma, t_1, \ldots, t_k, T) \mapsto (\gamma_T, t_1/T, \ldots, t_k/T)
\]
is its smooth homotopy inverse, where
$\gamma_T(\theta):= \gamma(T\theta)$. 
Therefore, it is enough to show that the inclusion map 
$\bL^a_{k, T>0} \to \bL^a_k$ gives an isomorphism on $H^\dR_*$. 

\textbf{Step 2.} 
Let $\E(\gamma,T):= \int_0^T |\dot{\gamma}|^2$. For any $E>0$, let us set
\[
\bL^{a,E}_k:= \{(\gamma,t_1,\ldots, t_k,T) \in \bL^a_k  \mid \E(\gamma,T) <E\}, \qquad
\bL^{a,E}_{k, T>0}:= \bL^a_{k, T>0} \cap \bL^{a,E}_k. 
\]
Since $\E$ is smooth as a function on $\bL^a_k$,
we obtain isomorphisms
$H^\dR_*(\bL^a_k) \cong \vlim_{E \to \infty} H^\dR_*(\bL^{a,E}_k)$ and 
$H^\dR_*(\bL^a_{k, T>0}) \cong \vlim_{E \to \infty} H^\dR_*(\bL^{a,E}_{k, T>0})$. 
Thus, 
it is enough to show that the inclusion map 
$\bL^{a,E}_{k, T>0} \to \bL^{a,E}_k$ induces an isomorphism on $H^\dR_*$ for every $E>0$. 

\textbf{Step 3.} 
For any $E, \delta>0$, let us set 
\begin{align*}
\bL^{a,E}_{k, \delta}&:= \{(\gamma,t_1,\ldots, t_k, T) \in \bL^{a,E}_k \mid \text{$\gamma$ is a constant loop if $T<\delta$} \},  \\
\bL^{a,E}_{k, T>0, \delta}&:= \bL^{a,E}_{k, T>0} \cap \bL^{a,E}_{k, \delta}. 
\end{align*}
Consider the following commutative diagram, where all maps are inclusion maps:
\[
\xymatrix{
\bL^{a,E}_{k, T>0,\delta} \ar[r]^{j_3} \ar[d]_{j_1}& \bL^{a,E}_{k, \delta} \ar[d]^{j_2} \\
\bL^{a,E}_{k, T>0} \ar[r]_{j_4} & \bL^{a,E}_k. 
}
\]
We want to show that $j_4$ induces an isomorphism on $H^\dR_*$. 
This follows from the following assertions: 
\begin{itemize}
\item 
There exists $\delta(E)>0$, depending only on $E$, such that 
if $\delta < \delta(E)$ then 
$j_1$ and $j_2$ induce isomorphisms on $H^\dR_*$. 
This is because one can define homotopy inverses of $j_1$ and $j_2$ 
in the following way. 

For any $(\gamma, T) \in \bar{\L}_0$ with $\len(\gamma)$ smaller than the injectivity radius of $M$, 
and $s \in [0,1]$, we define $(\gamma_s, T) \in \bar{\L}_0$ by 
$\gamma_s(\theta):= F_s(\gamma(0), \gamma(\theta))$, 
where the map $F_s$ was defined right before Lemma \ref{150625_2}. 
If $\len (\gamma)$ is smaller than a positive constant which is determined by the curvature of $M$, 
$\len(\gamma_s) \le \len(\gamma)$ and 
$\E(\gamma_s) \le \E(\gamma)$ for any $s \in [0,1]$. 
Then, for sufficiently small $T$, one obtains 
$(\gamma, T) \in \bar{\L}^{a,E}_0 \implies (\gamma_s, T) \in \bar{\L}^{a, E}_0$ for any $s \in [0,1]$, 
since $\len(\gamma) \le (ET)^{1/2}$.

Now let $\delta$ be a sufficiently small positive number, and take 
$\rho \in C^\infty(\R_{\ge 0}, [0,1])$ such that $\rho(T)=0$ if $T \in [0,\delta]$ and $\rho(T)=1$ if $T \ge 2\delta$. 
Let us define $j_2^\vee: \bar{\L}^{a,E}_k \to \bar{\L}^{a,E}_{k,\delta}$ by 
\[ 
j_2^\vee (\gamma, t_1, \ldots, t_k, T) = 
\begin{cases} 
(\gamma, t_1, \ldots, t_k, T) &(T \ge 2\delta) \\
(\gamma_{\rho(T)}, t_1, \ldots, t_k, T) &(T < 2\delta). 
\end{cases} 
\]
Then $j_2^\vee$ is a homotopy inverse of $j_2$. 
A homotopy inverse of $j_1$ is defined as the restriction of $j_2^\vee$ to 
$\bar{\L}^{a,E}_{k, T>0}$. 

\item For any $\delta>0$, $j_3$ induces 
an isomorphism on $H^\dR_*$, since its homotopy inverse is given by 
$(\gamma, t_1, \ldots, t_k, T) \mapsto (\gamma_T, t_1, \ldots, t_k, \nu(T))$, such that 
\begin{itemize}
\item $\nu(T) \ge T$ for any $T \ge 0$, $\nu(T)=T$ for any $T \ge \delta/2$, and $\nu(0)>0$. 
\item $\gamma_T$ is defined as $\gamma_T:= \begin{cases} \gamma &(T \ge \delta) \\ \text{constant loop at $\gamma(0)$} &(T < \delta). \end{cases}$
\end{itemize} 
\end{itemize} 
This completes the proof of Lemma \ref{150626_5}.  \qed

\section{Proof of Theorem \ref{160408_1}}

In this section we complete the proof of Theorem \ref{160408_1}. 
Along the way we also confirm Proposition \ref{160420_2}. 
Section 8.1 is devoted to some algebraic preliminaries. 
In Section 8.2, we define an isomorphism $\H_*(\L^a M) \cong H_*(F^a \widetilde{\mca{O}_M})$ for every $a \in (0, \infty]$
as we stated in Proposition \ref{160420_2}. 
When $a=\infty$, the isomorphism $\H_*(\L M) \cong \H_*(\widetilde{\mca{O}_M})$ is the isomorphism $\Phi$ in Theorem \ref{160408_1} (ii). 

In Section 8.3, we prove Theorem \ref{160408_1} (iii). 
Namely, we define a morphism $\mca{O}_M \to \endo(\mca{A}_M)$ of dg operads which preserves multiplications, 
and check that the induced map on homology $H_*(\widetilde{\mca{O}_M}) \to H^*(\mca{A}_M, \mca{A}_M)$ 
coincides with the map (\ref{141218_01}): $\H_*(\mca{L}M) \to H^*(\mca{A}_M, \mca{A}_M)$ 
which is defined by iterated integrals. 

In Section 8.4, we prove Theorem \ref{160408_1} (iv). 
Namely, we define a chain map $\iota_M: (\mca{A}_M)_* \to C^{\mca{L}M}_*$ 
and check the equation (\ref{150801_2}) and the commutativity of the diagram (\ref{150801_3}).

Sections 8.5--8.8 are devoted to the proof that 
the isomorphism $\Phi: \H_*(\L M) \cong H_*(\widetilde{\mca{O}_M})$ preserves the BV algebra structures. 
In Section 8.5 we prove that $\Phi$ preserves the rotation operator $\Delta$. 
In Sections 8.6--8.8, we prove that $\Phi$ preserves the product $\bullet$. 

\subsection{Algebraic preliminaries} 

Let $C=(C(k))_{k \ge 0}$ be a double complex with anti-chain maps $\delta_k: C(k-1)_* \to C(k)_* \,(\forall k \ge 1)$. 
In Section 2.5.1, we defined the total complex $(\tilde{C}, \tilde{\partial})$ by 
\[
\tilde{C}_*:= \prod_{k=0}^\infty C(k)_{*+k}, \qquad 
(\tilde{\partial} x)_k := \begin{cases} \partial x_0 &(k=0) \\ \partial x_k + \delta_k (x_{k-1}) &(k \ge 1). \end{cases} 
\]

\begin{lem}\label{150629_1} 
Let $C=(C(k))_{k \ge 0}$ be a double complex. If the sequence 
\[ 
\xymatrix{
0 \ar[r] & H_q(C(0)) \ar[r]_-{H_q(\delta_1)} & H_q(C(1)) \ar[r]_-{H_q(\delta_2)} & H_q(C(2)) \ar[r]_-{H_q(\delta_3)} & \cdots
}
\]
is exact for every $q \in \Z$, then the total complex $\tilde{C}$ is acyclic.  
\end{lem} 
\begin{proof} 
For any $l \ge 0$, let $F_l\tilde{C}_*:= \prod_{k \ge l} C(k)_{*+k}$. 
Then, $(F_l\tilde{C})_{l \ge 0}$ is a decreasing filtration on $\tilde{C}$ which is complete, i.e. 
$\tilde{C} \cong \plim_{l \to \infty} \tilde{C}/F_l\tilde{C}$. 
Let us consider the spectral sequence of this filtered complex. 
Then, the assumption implies that all $E^2$-terms vanish. 
Now, the convergence theorem 5.5.10 (2) in \cite{Weibel_94} pp.139 shows $H_*(\tilde{C})=0$. 
\end{proof} 

\begin{lem}\label{150627_5} 
Let $\ph: C \to D$ be a morphism of double complexes. 
Suppose that $\ph(k): C(k)_* \to D(k)_*$ is a quasi-isomorphism for every $k \ge 0$. 
Then, the chain map $\tilde{\ph}: \tilde{C} \to \tilde{D}$ is a quasi-isomorphism.  
\end{lem} 
\begin{proof}
Let us consider the filtrations $(F_l\tilde{C})_{l \ge 0}$ and $(F_l\tilde{D})_{l \ge 0}$ as in the proof of the previous lemma. 
Then, $\tilde{\ph}: \tilde{C} \to \tilde{D}$ induces a morphism of the spectral sequences, and 
the assumption implies that it induces isomorphisms on $E^1$-terms. 
Then, the comparison theorem 5.5.11 in \cite{Weibel_94} pp. 141 shows that $H_*(\tilde{\ph}): H_*(\tilde{C}) \to H_*(\tilde{D})$ is an isomorphism. 
\end{proof} 

As a consequence of Lemma \ref{150629_1}, we obtain the next lemma. 

\begin{lem}\label{150627_4} 
Let $C=(C(k)_*)_{k \ge 0}$ be a cosimplicial chain complex, and suppose that the chain map $C(k)_* \to C(0)_*$, 
which is induced by the cosimplicial structure, 
is a quasi-isomorphism for every $k \ge 0$. 
Then, the projection map $\pr_0: \tilde{C}_* \to C(0)_*; (x_k)_{k \ge 0} \mapsto x_0$ is a quasi-isomorphism. 
\end{lem} 
\begin{proof} 
Since $\pr_0$ is surjective, it is enough to show that 
$\ker(\pr_0) = \prod_{k=1}^\infty C(k)_{*+k}$ is acyclic. 
The assumption shows that, for any $q \in \Z$ the sequence 
\[
\xymatrix{
0 \ar[r] & H_q(C(1)) \ar[r]_-{H_q(\delta_2)} & H_q(C(2)) \ar[r]_-{H_q(\delta_3)} & H_q(C(3)) \ar[r]_-{H_q(\delta_4)} & \cdots
}
\]
is exact, since we can identify this sequence with 
\[
\xymatrix{
0 \ar[r] & H_q(C(0)) \ar[r]_-{\times 1} & H_q(C(0)) \ar[r]_-{0} & H_q(C(0)) \ar[r]_-{\times 1} & \cdots. 
}
\]
Thus Lemma \ref{150629_1} shows that $\ker(\pr_0)$ is acyclic. 
\end{proof} 

\subsection{The isomorphism $\Phi: H_*(\widetilde{\mca{O}_M}) \cong \H_*(\mca{L}M)$}

Let $M$ be a closed, oriented Riemannian manifold of dimension $d$. 
The aim of this subsection is to define an isomorphism 
$H_*(F^a \widetilde{\mca{O}_M}) \cong \H_*(\mca{L}^aM)$ for every $a \in (0, \infty]$, 
and check that these isomorphisms are compatible with the length filtration 
(i.e. the diagram in Proposition \ref{160420_2} (iii) commutes). 
When $a=\infty$ we obtain an isomorphism 
$H_*(\widetilde{\mca{O}_M}) \cong \H_*(\mca{L}M)$, 
and this is the isomorphism $\Phi$ in Theorem \ref{160408_1}. 
The definition of the isomorphism $H_*(F^a \widetilde{\mca{O}_M}) \cong \H_*(\L^a M)$ consists of three steps. 

\textbf{Step 1.} 
Let us abbreviate $\mca{L}^aM$ by $\mca{L}^a$, 
and recall the zig-zag (\ref{150628_2}) of quasi-isomorphisms from Section 7.2: 
\[
\xymatrix{
C^\dR_{*+d}(\bL^a_{k,\reg}) \ar[r] & C^\dR_{*+d}(\bL^a_k)  & \ar[l] C^\dR_{*+d}(\mca{L}^a_k)  \ar[r] & C^\dR_{*+d}(\mca{L}^a \times \Delta^k). 
}
\]
It induces a zig-zag of quasi-isomorphisms of cocylic chain complexes. 
Let $F^aC^{\mca{L}\Delta}_*$ denote the total complex of the cosimplicial chain complex $(C^\dR_{*+d}(\mca{L}^a \times \Delta^k))_{k \ge 0}$. 
Then, Lemma \ref{150627_5} implies the isomorphism 
$H_*(F^a \widetilde{\mca{O}_M}) \cong H_*(F^aC^{\mca{L}\Delta})$. 

\textbf{Step 2.} 
For every $k \ge 0$ the projection 
$\mca{L}^a \times \Delta^k \to \mca{L}^a$ induces a quasi-isomorphism on $C^\dR_*$. 
Then, Lemma \ref{150627_4} shows that $\pr_0: F^a C^{\mca{L}\Delta}_* \to C^\dR_{*+d}(\mca{L}^a)$ is a quasi-isomorphism. 
Hence we obtain an isomorphism 
$H_*(F^aC^{\mca{L}\Delta}) \cong H^\dR_{*+d}(\mca{L}^a)$. 
We need Lemma \ref{150629_3} below for later use. 

\begin{lem}\label{150629_3} 
Let $u=(u_k)_{k \ge 0}$ be as in Lemma \ref{150624_10} (i). Then, 
\[
E_u: C^\dR_{*+d}(\mca{L}^a) \to F^a C^{\mca{L}\Delta}_*; \quad x \mapsto ((-1)^{k(d+1)} x \times u_k)_{k \ge 0} 
\]
is a chain map such that $\pr_0 \circ E_u = \id_{C^\dR(\mca{L}^a)}$. 
\end{lem} 
\begin{proof}
The claim that $E_u$ is a chain map follows from direct computations, however one should notice the following signs, which might be confusing: 
\[
\partial (x \times u_k )= \partial x \times u_k + (-1)^{|x|+d} x \times \partial u_k, \qquad 
\delta_k(x \times u_{k-1}) = (-1)^{|x|} x \times \biggl( \sum_{i=0}^k (-1)^i (d_{k,i})_* (u_{k-1}) \biggr). 
\] 
The latter claim $\pr_0 \circ E_u = \id_{C^\dR(\mca{L}^a)}$ is clear from the definition of $E_u$. 
\end{proof} 

\textbf{Step 3.} 
Finally, we define an isomorphism $H_*(F^a \widetilde{\mca{O}_M}) \cong \H_*(\L^a)$ by 
\[ 
H_*(F^a \widetilde{\mca{O}_M}) \cong H_*(F^a C^{\L \Delta}) \cong H^\dR_{*+d}(\L^a) \cong \H_*(\L^a).
\]
The first isomorphism is defined in Step 1, 
the second isomorphism is defined in Step 2, 
and the last isomorphism is defined in Theorem \ref{150219_1}. 
These three isomorphisms are compatible with the length filtrations, 
hence the diagram in Proposition \ref{160420_2} (iii) commutes. 

\subsection{A morphism $\mca{O}_M \to \endo(\mca{A}_M)$}

We define a morphism $\mca{O}_M \to \endo(\mca{A}_M)$ of dg operads which preserves multiplications, 
and check that the induced map on homology $H_*(\widetilde{\mca{O}_M}) \to H^*(\mca{A}_M, \mca{A}_M)$ 
coincides with the map (\ref{141218_01}) $: \H_*(\L M) \to H^*(\mca{A}_M, \mca{A}_M)$ via the isomorphism $\Phi: H_*(\widetilde{\mca{O}_M}) \cong \H_*(\L M)$. 
Thus we confirm Theorem \ref{160408_1} (iii). 

For any $k \ge 0$, $(U,\ph) \in \mca{P}(\bL_{k,\reg})$ and $j=0, \ldots, k$, 
let $\ph_j:= e_j \circ \ph$. 
By definition, the maps $\ph_j$ are submersions for all $j$. 
We define a chain map $J_k: C^\dR_{*+d}(\bL_{k, \reg}) \to \Hom_*(\mca{A}_M^{\otimes k}, \mca{A}_M)$ by 
\[
J_k([(U,\ph,\omega)])(\eta_1 \otimes \cdots \otimes \eta_k):= (-1)^{(\dim U-d)(|\eta_1| + \cdots + |\eta_k|)} (\ph_0)_!(\omega \wedge \ph_1^*\eta_1 \wedge \cdots \wedge \ph_k^*\eta_k).
\]
Recall that $\mca{O}_M(k)_* = C^\dR_{*+d}(\bL_{k,\reg})$ and
$\endo(\mca{A}_M)(k)_* = \Hom_*(\mca{A}_M^{\otimes k}, \mca{A}_M)$ for every $k \ge 0$. 

\begin{lem}\label{160425_1} 
$(J_k)_{k \ge 0}: \mca{O}_M\to \endo(\mca{A}_M)$ is a morphism of
nonsymmetric dg operads preserving multiplications and units.
\end{lem}
\begin{proof} 
The only nontrivial part is to verify (with signs) that $(J_k)_{k \ge 0}$ preserves operad compositions.  
To verify this, let us recall that we defined (in Section 7.2) 
the map $i_0: M \to \bar{\mca{L}}_0; \, p \mapsto c_p$. 
Let us take $M' \in \mca{U}$ and 
an orientation-preserving diffeomorphism $\ph: M' \to M$, 
and define a chain map $\iota: \mca{A}_M \to \mca{O}_M(0)$ by 
$\iota(\omega):= [(M', i_0 \circ \ph, \ph^*\omega)]$. 
We also define a chain map 
$\ep: \mca{O}_M(0) \to \mca{A}_M$ by 
$\ep([(U, \ph, \omega)]):= (\ph_0)_! \omega$.
Then, it is easy to check that 
\[ 
J_k(x) (\eta_1 \otimes \cdots \otimes \eta_k) := \ep( (\cdots ((x \circ_1 \iota(\eta_1)) \circ_1 \iota(\eta_2)) \cdots ) \circ_1 \iota(\eta_k)). 
\] 
Now the fact that $(J_k)_{k \ge 0}$ preserves operad compositions is verified as follows: 
let $x \in \mca{O}_M(k)$, $y \in \mca{O}_M(l)$, $\eta_1, \ldots, \eta_{k+l-1} \in \mca{A}_M$, then 
(some parentheses are omitted): 
\begin{align*} 
&J_{k+l-1} (x \circ_i y) (\eta_1 \otimes \cdots \otimes \eta_{k+l-1}) \\
&\quad = \ep ( (x \circ_i y) \circ_1 \iota(\eta_1) \cdots \circ_1 \iota(\eta_{k+l-1})) \\
&\quad = (-1)^{|y|(|\eta_1| + \cdots + |\eta_{i-1}|)} \ep (x \circ_1 \iota(\eta_1) \cdots \circ_1 \iota(\eta_{i-1}) \circ_1 (y \circ_1 \iota (\eta_i) \cdots \circ_1 \iota(\eta_{i+l-1})) \circ_1 \iota(\eta_{i+l}) \cdots \circ_1 \iota(\eta_{k+l-1})) \\
&\quad = (-1)^{|y|(|\eta_1| + \cdots + |\eta_{i-1}|)} \ep (x \circ_1 \iota(\eta_1) \cdots \circ_1 \iota(\eta_{i-1}) \circ_1 \iota \circ \ep (y \circ_1 \iota (\eta_i) \cdots \circ_1 \iota(\eta_{i+l-1})) \circ_1 \iota(\eta_{i+l}) \cdots \circ_1 \iota(\eta_{k+l-1})) \\
&\quad = (J_k(x) \circ_i J_l(y))(\eta_1 \otimes \cdots \otimes \eta_{k+l-1}). 
\end{align*} 
The third equality follows from the following observation: 
for any $a \in \mca{O}_M(1)$ and $b \in \mca{O}_M(0)$, 
there holds $\ep(a \circ_1 (\iota \circ \ep(b))) = \ep(a \circ_1 b)$ 
although $a \circ_1 (\iota \circ \ep(b)) \ne a \circ_1 b$ in general. 
This is because, 
although $a \circ_1 (\iota \circ \ep(b))$ and $a \circ_1 b$ are (in general) represented by different (sets of) loops, their origins coincide. 
Now we obtain the third equality by applying this observation for 
\[ 
a:= (x \circ_1 \iota(\eta_1)  \cdots \circ_1 \iota(\eta_{i-1})) \circ_2 \iota(\eta_{i+l}) \cdots \circ_2 \iota(\eta_{k+l-1}), \qquad
b:= y \circ_1 \iota(\eta_i)  \cdots \circ_1 \iota(\eta_{i+l-1}). 
\] 
\end{proof} 

By Lemma \ref{160425_1}, 
the sequence of chain maps $(J_k)_{k \ge 0}$ induces a chain map 
$J: \widetilde{\mca{O}_M}_* \to C^*(\mca{A}_M, \mca{A}_M)$. 
Now let us prove Lemma \ref{150725_1} below. 

\begin{lem}\label{150725_1} 
The map 
$H_*(J): H_*(\widetilde{\mca{O}_M}) \to H^*(\mca{A}_M, \mca{A}_M)$ 
corresponds to the map (\ref{141218_01}) $: \H_*(\L M) \to H^*(\mca{A}_M, \mca{A}_M)$ via the isomorphism $\Phi: H_*(\widetilde{\mca{O}_M}) \cong \H_*(\L M)$. 
\end{lem} 
\begin{proof} 
For any $j=0, \ldots, k$, we define $e_j: \mca{L} \times \Delta^k \to M$ by 
\[
e_j(\gamma, t_1, \ldots, t_k):= \begin{cases} \gamma(0) &(j=0) \\ \gamma(t_j) &(1 \le j \le k). \end{cases}
\]
For any $(U, \ph) \in \mca{P}(\mca{L}\times \Delta^k)$, we set $\ph_j:= e_j \circ \ph$. 
Let us define a chain map 
$J'_k: C^\dR_*(\mca{L} \times \Delta^k) \to \Hom(\mca{A}_M^{\otimes k}, \mca{A}_M^\vee[d])$ by 
\[
J'_k [(U,\ph,\omega)](\eta_1 \otimes \cdots \otimes \eta_k)(\eta_0) := 
(-1)^{(\dim U-d)(|\eta_0|+|\eta_1| + \cdots + |\eta_k|)} \int_U \omega \wedge \ph_1^*\eta_1 \wedge \cdots \wedge \ph_k^* \eta_k \wedge \ph_0^*\eta_0.
\]
Then, one can define a chain map $J'=(J'_k)_{k \ge 0}: C^{\mca{L}\Delta}_* \to C^*(\mca{A}_M, \mca{A}_M^\vee[d])$. 
As is obvious from the construction, the following diagram commutes
(the left vertical map is defined in Section 8.2 Step 1): 
\[ 
\xymatrix{ 
H_*(\widetilde{\mca{O}_M}) \ar[r]^-{H_*(J)} \ar[d]_-{\cong} & H^*(\mca{A}_M, \mca{A}_M) \ar[d]^-{\cong} \\
H_*(C^{\mca{L}\Delta}) \ar[r]_-{H_*(J')} & H^*(\mca{A}_M, \mca{A}^\vee_M[d]). 
}
\]
Let us recall the chain map 
$I: C^\sm_{*+d}(\L) \to C^*(\mca{A}_M, \mca{A}_M^\vee[d])$ in (\ref{141222_4}). 
We need to show that 
$H_*(J')$ corresponds to $H_*(I)$ via the isomorphism $H^\sm_{*+d}(\mca{L}) \cong H_*(C^{\mca{L}\Delta})$ 
which is a composition of the isomorphisms $H^\sm_{*+d}(\L) \cong H^\dR_{*+d}(\L)$ (Theorem \ref{150219_1})
and $H^\dR_{*+d}(\L) \cong H_*(C^{\L \Delta})$ (Section 8.2 Step 2). 

Let us take $u=(u_k)_{k \ge 0}$ as in Lemma \ref{150624_10} (i), 
and consider chain maps
\[
\iota^u(\mca{L})_*: C^\sm_*(\mca{L}) \to C^\dR_*(\mca{L}), \qquad 
E_u: C^\dR_{*+d}(\mca{L}) \to C^{\mca{L}\Delta}_*.
\]
The first map is defined right after Lemma \ref{150624_10}, 
and 
the second map is defined in Lemma \ref{150629_3}. 
Then, it is enough to show that the following diagram commutes: 
\[
\xymatrix{
H^\sm_{*+d}(\mca{L}) \ar[r]^-{H_*(\iota^u)}_-{\cong} \ar[rd]_-{H_*(I)} & H^\dR_{*+d}(\mca{L}) \ar[r]^-{H_*(E_u)}_-{\cong} & H_*(C^{\mca{L}\Delta}) \ar[ld]^-{H_*(J')} \\
   & H^*(\mca{A}_M, \mca{A}_M^\vee[d]). &
}
\] 
Let us check the commutativity of this diagram at the chain level. 
Let $\sigma: \Delta^l \to \mca{L}$ be a smooth map, 
and define $\sigma_{k, i}: \Delta^l \times \Delta^k \to M$ as in Section 2.4. 
Then 
$(E_u \circ \iota^u)(\sigma) = (-1)^{k(d+1)} \sigma_*(u_l) \times u_k$. 
On the other hand, 
\[ 
J'_k(\sigma_*(u_l) \times u_k) (\eta_1 \otimes \cdots \otimes \eta_k)(\eta_0) = 
(-1)^{(k+l)(k+l-d)} 
\langle \sigma_{k,1}^* \eta_1 \wedge \cdots \wedge \sigma_{k,k}^* \eta_k \wedge \sigma_{k,0}^* \eta_0, \, u_l \times u_k \rangle
\]
where $\langle \, , \, \rangle$ in the right hand side is defined in Section 4.8. 
Hence we obtain 
\begin{align*} 
&(J'_k \circ E_u \circ \iota^u)(\sigma) (\eta_1 \otimes \cdots \otimes \eta_k) (\eta_0) \\ 
&\quad = (-1)^{k(d+1) + (k+l)(k+l-d) + (k+l)(k+l-1)/2} \int_{\Delta^l \times \Delta^k} \sigma_{k,1}^*\eta_1 \wedge \cdots \wedge \sigma_{k,k}^* \eta_k \wedge \sigma_{k,0}^* \eta_0 \\ 
&\quad = I_k(\sigma) (\eta_1 \otimes \cdots \otimes \eta_k)(\eta_0), 
\end{align*} 
where the first equality follows from Lemma \ref{150629_2}, 
and the second equality follows since the exponent on the second line is equal to 
$l(d+1) + (k+l)(k+l-1)/2$ modulo $2$. 
\end{proof} 

\subsection{Chain map $\iota_M$}

We define a chain map $\iota_M: (\mca{A}_M)_* \to C^{\mca{L}M}_*$ (see Theorem \ref{160408_1} (iv)) 
and check the equation (\ref{150801_2}) and the commutativity of the diagram (\ref{150801_3}). 

We take $M' \in \mca{U}$ and an orientation-preserving diffeomorphism $\ph: M' \to M$. 
Let us recall that we defined the map $i_0: M \to \bL_0; \, p \mapsto c_p$, where $c_p$ is the constant loop at $p$. 
Then let us define $\iota_M$ by 
\[
(\iota_M(\omega))_k := \begin{cases} [(M', i_0 \circ \ph, \ph^*\omega)] &(k=0), \\  0 &(k \ge 1). \end{cases}
\]
It is easy to check that this is well-defined, 
i.e. it does not depend on the choices of $M'$ and $\ph$. 

The equation (\ref{150801_2}) follows from the explicit formulas for $\bullet$ and $\{ \, , \,\}$ 
in Theorem \ref{160420_1}. 
Commutativity of the diagram (\ref{150801_3}) 
follows from the commutativity of the diagram 
\[ 
\xymatrix{ 
H^{-*}_\dR(M)  \ar[r]^-{\cong}  \ar[d]^-{H_*(\iota_M)} & H^\dR_{*+d}(M_\reg)  \ar[r]^-{\cong}\ar[d]^{H_{*+d}(i_0)} & H^\dR_{*+d}(M)  \ar[d]^{H_{*+d}(i_0)} \\
H_*(\widetilde{\mca{O}_M}) \ar[r]^-{\cong}_-{H_*(\pr_0)} & H^\dR_{*+d}(\bar{\L}_{0, \reg})  \ar[r]^-{\cong} & H^\dR_{*+d}(\bar{\L}_0)
}
\]
and the fact that 
isomorphisms $H^\dR_{*+d}(M) \cong \H_*(M)$ and $H^\dR_{*+d}(\bar{\L}_0) \cong \H_*(\mca{L}M)$ intertwines 
$H_{*+d}(i_0)$ (the right vertical map in the above diagram) 
with $\H_*(i_M): \H_*(M) \to \H_*(\mca{L}_M)$. 

\subsection{The rotation operator $\Delta$}

The aim of this subsection is to show that the isomorphism 
$\Phi: H_*(\widetilde{\mca{O}_M}) \cong \H_*(\L M)$ preserves the rotation operator $\Delta$. 

For any cocyclic chain complex 
$C=(C(k)_*)_{k \ge 0}$ and its total complex $\tilde{C}$, 
let us define an  anti-chain map $\Delta: \tilde{C}_* \to \tilde{C}_{*+1}$ by
\[
(\Delta x)_k: =\sum_{i=1}^{k+1} (-1)^{|x| + k(i-1)+1}  \sigma_{k+1-i} \tau^i_{k+1} x_{k+1}.
\]
This is a generalization of the definition of $\Delta$ in Theorem \ref{160420_3} (i). 
In particular, we can define $\Delta$ on $C^{\mca{L}\Delta}_* = \prod_{k \ge 0} C^\dR_{*+d+k}(\mca{L} \times \Delta^k)$
so that the isomorphism $H_*(\widetilde{\mca{O}_M}) \cong H_*(C^{\mca{L}\Delta})$ preserves $\Delta$. 
Since we already proved that the isomorphism  $H^\dR_*(\mca{L}) \cong H_*(\mca{L})$ preserves $\Delta$ (Corollary \ref{150627_2}), 
it is enough to show that the isomorphism $H^\dR_{*+d}(\mca{L}) \cong H_*(C^{\mca{L}\Delta})$ preserves $\Delta$. 

The isomorphism $H^\dR_{*+d}(\mca{L}) \cong H_*(C^{\mca{L}\Delta})$ is induced by 
$\pr_0: C^{\mca{L}\Delta}_* \to C^\dR_{*+d}(\mca{L})$, 
and its inverse is induced by $E_u$ (Lemma \ref{150629_3}), where $u=(u_k)_{k \ge 0}$ is as in Lemma \ref{150624_10} (i). 
Therefore, it is enough to prove
$H_*(\pr_0) \circ \Delta \circ H_*(E_u)=\Delta$ on $H^\dR_{*+d}(\mca{L})$. 
It is easy to check that 
$\pr_0 \circ \Delta \circ E_u(x)= (-1)^{|x|+d} \sigma_0 \tau_1(x \times u_1)$ for any $x \in C^\dR_{*+d}(\mca{L})$. 
The chain maps $\tau_1$ and $\sigma_0$ are induced by 
$\L \times \Delta^1 \to \L \times \Delta^1: (\gamma, t) \mapsto (\gamma^t, 1-t)$ 
and $\L \times \Delta^1 \to \L: (\gamma, t) \mapsto \gamma$, 
respectively
($\gamma^t \in \mca{L}$ is defined by $\gamma^t(\theta):= \gamma(\theta-t)$). 

Let us consider a sequence of maps 
\[
\xymatrix{
\mca{L} \times \Delta^1 \ar[r]^-{\id_{\mca{L}} \times p} & \mca{L} \times S^1 \ar[r] & S^1 \times \mca{L} \ar[r]^-{r}& \mca{L}, 
}
\]
where $p: \Delta^1 \to S^1$ is defined by $p(\theta):=[\theta]$, 
the second map is inversion, and the last map $r$ is rotation; $r(t,\gamma)(\theta):= \gamma(\theta-t)$. 
Therefore $\sigma_0 \tau_1 (x \times u_1) = (-1)^{|x|+d}  r_*(p_*(u_1) \times x)$ for any $x \in C^\dR_{*+d}(\mca{L})$. 
Since $p_*(u_1) \in C^\dR_1(S^1)$ is a cycle which represents $[S^1]$, 
for any cycle $x$ we obtain 
$H_*(\pr_0) \circ \Delta \circ H_*(E_u)([x]) = H_*(r)([S^1] \times [x])$. 
This completes the proof. 

\subsection{The loop product $\bullet$}

We are going to show that the isomorphism $\Phi$ 
preserves the operator $\bullet$. 
In this subsection, we reduce the proof to Lemma \ref{150629_4}, 
which is proved in the next two subsections. 

Let us consider the concatenation map 
\[
c: \bL_0 \fbp{e_0}{e_0} \bL_0 \to \bL_0; \qquad  ((\gamma_0, T_0), (\gamma_1, T_1)) \mapsto (\gamma_0 * \gamma_1, T_0+T_1), 
\]
and define a chain map  
$\bullet_0: C^\dR_{*+d}(\bL_{0,\reg})^{\otimes 2} \to C^\dR_{*+d}(\bL_{0,\reg})$ by 
$a \bullet_0 b:= c_*(a \fbp{e_0}{e_0} b)$. 
Then, $\pr_0: \widetilde{\mca{O}_M}_* \to C^\dR_{*+d}(\bL_{0,\reg}); (x_k)_{k \ge 0} \mapsto x_0$ 
intertwines the operators $\bullet$ and $\bullet_0$: 
\[
(x \bullet y)_0 = (\mu \circ_1 x_0) \circ_1 y_0 = c_*( x_0 \fbp{e_0}{e_0} y_0) = x_0 \bullet_0 y_0.
\]
Therefore, it is enough to prove the next proposition. 

\begin{prop}\label{150627_3}
Let us consider the isomorphism $H^\dR_{*+d}(\bL_{0,\reg}) \cong \H_*(\L)$, 
which is the composition of $H^\dR_{*+d}(\bL_{0,\reg}) \cong H^\dR_{*+d}(\L)$ (this follows from the zig-zag (\ref{150628_2}))
and $H^\dR_{*+d}(\L) \cong \H_*(\L)$ (Theorem \ref{150219_1}). 
This isomorphism intertwines the operator $\bullet_0$ on $H^\dR_{*+d}(\bL_{0,\reg})$ and the loop product $\bullet$ on $\H_*(\mca{L})$. 
\end{prop} 

Let us define $e:\mca{L} \to M$ by $e(\gamma):=\gamma(0)$. 
A key step in the definition of the loop product on $\H_*(\mca{L})$ is to define the fiber product 
\[
H_*(\mca{L})^{\otimes 2} \to H_{*-d}(\mca{L} \fbp{e}{e} \mca{L})
\]
via the Thom isomorphism (see Section 2.3). 
On the other hand, 
let us consider the following differentiable structure on $\mca{L}$, and denote the resulting differentiable space by $\mca{L}_\reg$: 
\[
\mca{P}(\mca{L}_\reg):= \{(U,\ph) \in  \mca{P}(\mca{L}) \mid \text{$e \circ \ph: U \to M$ is a submersion}\}. 
\]
Then, as in Section 4.3, one can define the fiber product 
\[
H^\dR_*(\mca{L}_\reg)^{\otimes 2} \to H^\dR_{*-d}(\mca{L}_\reg \fbp{e}{e} \mca{L}_\reg).
\]

We have isomorphisms $H_*(\mca{L}) \cong H^\dR_*(\mca{L}) \cong H^\dR_*(\mca{L}_\reg)$; 
the first isomorphism is by Theorem \ref{150219_1}, 
and the second isomorphism is obtained by similar arguments as the proof of Lemma \ref{150624_2}. 
We also have isomorphisms 
\[
H_*(\mca{L} \fbp{e}{e} \mca{L}) \cong 
H^\dR_*(\mca{L} \fbp{e}{e} \mca{L})  \cong 
H^\dR_*(\mca{L}_\reg \fbp{e}{e} \mca{L}_\reg)
\]
by similar arguments. Let us state a key technical result:  

\begin{lem}\label{150629_4} 
The following diagram commutes: 
\[
\xymatrix{ 
H^\dR_*(\L_\reg)^{\otimes 2} \ar[r] \ar[d]_{\cong} & H^\dR_{*-d}(\L_\reg \fbp{e}{e} \L_\reg) \ar[d]^{\cong} \\
H_*(\L)^{\otimes 2} \ar[r] & H_{*-d}(\L \fbp{e}{e} \L), 
}
\]
where horizontal maps are fiber products. 
\end{lem} 

It is easy to deduce Proposition \ref{150627_3} from Lemma \ref{150629_4}, 
and details are left to the reader. 
The rest of this section is devoted to the proof of Lemma \ref{150629_4}. 
In the next subsection, 
we prove Lemma \ref{150726_1}, 
which is a preliminary result in the finite-dimensional setting. 

\subsection{A preliminary result in the finite-dimensional setting}

Let $X$ be an oriented $C^\infty$-manifold, and $e: X \to M$ be a submersion. 
Then, one can define the fiber product 
$H_*(X)^{\otimes 2} \to H_{*-d}(X \fbp{e}{e} X)$ via the Thom isomorphism 
for the tubular neighborhood of 
$X \fbp{e}{e} X \subset X \times X$. 

On the other hand, let us consider the following differentiable structure on $X$, 
and denote the resulting differentiable space by $X_{\reg/M}$: 
\[
\mca{P}(X_{\reg/M}) := \{(U,\ph) \mid \ph \in C^\infty(U, X) \, \text{and $e \circ \ph: U \to M$ is a submersion} \}. 
\]
Then, one can define the fiber product 
$H^\dR_*(X_{\reg/M})^{\otimes 2} \to H^\dR_{*-d}(X_{\reg/M} \fbp{e}{e} X_{\reg/M})$. 

It is obvious that the identity maps on $X$ and $X \fbp{e}{e} X$ induce smooth maps 
\[
X_{\reg/M} \to X, \qquad 
X_{\reg/M} \fbp{e}{e} X_{\reg/M} \to X \fbp{e}{e} X.
\] 
These maps induce isomorphisms 
\[
H^\dR_*(X_{\reg/M}) \cong H^\dR_*(X), \qquad 
H^\dR_*(X_{\reg/M} \fbp{e}{e} X_{\reg/M}) \cong H^\dR_*(X \fbp{e}{e} X).
\]
This fact is proved by similar arguments as 
Proposition \ref{150210_4}, and details are omitted. 
Then, we obtain isomorphisms 
\begin{align*} 
&H^\dR_*(X_{\reg/M}) \cong H^\dR_*(X) \cong H_*(X),  \\ 
&H^\dR_*(X_{\reg/M} \fbp{e}{e} X_{\reg/M}) \cong 
H^\dR_*(X \fbp{e}{e} X) \cong 
H_*(X \fbp{e}{e} X).
\end{align*} 

\begin{lem}\label{150726_1}
The following diagram commutes: 
\begin{equation}\label{eq:fbp_finite}
\xymatrix{
H^\dR_*(X_{\reg/M})^{\otimes 2} \ar[r] \ar[d]^{\cong} & H^\dR_{*-d}(X_{\reg/M} \fbp{e}{e} X_{\reg/M}) \ar[d]^{\cong} \\
H_*(X)^{\otimes 2} \ar[r] & H_{*-d}(X \fbp{e}{e} X). 
}
\end{equation}
\end{lem} 
\begin{proof}
Via the isomorphisms $H^\dR_*(X_{\reg/M}) \cong H^{\dim X-*}_{c,\dR}(X) \cong H_*(X)$ and 
\[
H^\dR_{*-d}(X_{\reg/M} \fbp{e}{e} X_{\reg/M}) \cong H^{2\dim X-*}_{c,\dR}(X \fbp{e}{e} X) \cong H_{*-d}(X \fbp{e}{e} X),
\]
both horizontal maps in (\ref{eq:fbp_finite}) are identified with the map 
\[
H^*_{c,\dR}(X)^{\otimes 2} \to H^*_{c,\dR}(X \fbp{e}{e} X) ; \quad [\omega] \otimes [\eta] \mapsto [\omega \times \eta|_{X \,_e \times_e X}], 
\]
hence (\ref{eq:fbp_finite}) is commutative. 
\end{proof} 

\subsection{Proof of Lemma \ref{150629_4}}
We use the notation of Section 6.1.
We abbreviate 
$\mca{L}^{\infty, E}$ by $\mca{L}^E$, and 
$\mca{F}^{\infty, E}_N$ by $\mca{F}^E_N$. 
Let $(E_j)_{j \ge 1}$ be a strictly increasing sequence of positive real numbers, such that $\lim_{j \to \infty} E_j = \infty$. 
Let us take a sequence $(N_j)_{j \ge 1}$ of positive integers, so that 
$N_j | N_{j+1}$ and $N_j \ge N(E_j, E_{j+1})$ for every $j \ge 1$
(see Remark \ref{rem:a=infty}).
By Lemma \ref{150625_1}, 
there exists a continuous map $g_j: \mca{F}^{E_j}_{N_j} \to \L^{E_{j+1}}$ such that 
\begin{equation}\label{eq:ejgj}
\xymatrix{
\L^{E_j} \ar[r]\ar[d]_-{f_{N_j}}&\L^{E_{j+1}} \ar[d]^-{f_{N_{j+1}}}\\
\mca{F}^{E_j}_{N_j} \ar[r] \ar[ru]_-{g_j}&\mca{F}^{E_{j+1}}_{N_{j+1}}
}
\end{equation}
commutes up to homotopy. 
Then, $H_*(\mca{L}) \cong \vlim_j H_*(\mca{L}^{E_j}) \cong \vlim_j  H_*(\mca{F}^{E_j}_{N_j})$. 
In the following arguments, we abbreviate 
$\mca{F}^{E_j}_{N_j}$ by $\mca{F}^j$, 
and $f_{N_j}$ by $f_j$. 

Let $e_j: \mca{F}^j \to M; \, (x_k)_{0 \le k \le  N_j} \mapsto x_0$. 
As is clear from the proof of Lemma \ref{150625_1}, 
one may take $g_j$ so that 
$e \circ g_j=e_j$
(thus $g_j \times g_j: \mca{F}^j \fbp{e_j}{e_j} \mca{F}^j \to \mca{L}^{E_j} \fbp{e}{e} \mca{L}^{E_j}$ is well-defined), 
and the following diagram commutes up to homotopy: 
\begin{equation}\label{eq:ejgj_2}
\xymatrix{
\L^{E_j} \fbp{e}{e} \L^{E_j} \ar[r]\ar[d]_-{f_j \times f_j}&\L^{E_{j+1}} \fbp{e}{e}  \L^{E_{j+1}} \ar[d]^-{f_{j+1} \times f_{j+1} }\\
\mca{F}^j \fbp{e_j}{e_j}  \mca{F}^j  \ar[r] \ar[ru]_-{g_j \times g_j }&  \mca{F}^{j+1} \fbp{e_{j+1}}{e_{j+1}}  \mca{F}^{j+1}. 
}
\end{equation}
Then, 
$H_*(\mca{L} \fbp{e}{e} \mca{L}) \cong \vlim_j H_*(\mca{L}^{E_j} \fbp{e}{e} \mca{L}^{E_j}) \cong  \vlim_j H_*(\mca{F}^j \fbp{e_j}{e_j} \mca{F}^j)$. 

Since $e_j: \mca{F}^j \to M$ is a submersion, one can define the differentiable space $\mca{F}^j_{\reg/M}$ as in the previous subsection. 
Then, $f_j$ maps plots of $\mca{L}^{E_j}_\reg$ to plots of $\mca{F}^j_{\reg/M}$. 
Also, one may take $g_j$ so that it maps plots of 
$\mca{F}^j_{\reg/M}$ to plots of $\mca{L}^{E_{j+1}}_\reg$, and 
the diagrams  (\ref{eq:ejgj}) and (\ref{eq:ejgj_2}) commute up to smooth homotopy with these differentiable structures. 
Hence we obtain 
\begin{align*}
H^\dR_*(\L_\reg)&\cong \vlim_j H^\dR_*(\L^{E_j}_\reg) \cong  \vlim_j H^\dR_*(\mca{F}^j_{\reg/M}),   \\
H^\dR_*(\L_\reg \fbp{e}{e} \L_\reg)&\cong  \vlim_j H^\dR_*(\L^{E_j}_\reg \fbp{e}{e}  \L^{E_j}_\reg) \cong
 \vlim_j H^\dR_*(\mca{F}^j_{\reg/M} \fbp{e_j}{e_j} \mca{F}^j_{\reg/M}). 
\end{align*}

For every $j \ge 1$, let us consider the following diagram: 
\begin{equation}\label{eq:bigsquare}
\xymatrix{
H^\dR_*(\L^{E_j}_\reg)^{\otimes 2} \ar[rrr]\ar[ddd]\ar[rd]^{\star}&&& H^\dR_{*-d} (\L^{E_j}_\reg \fbp{e}{e} \L^{E_j}_\reg)\ar[ld]_{\star}\ar[ddd] \\
&H^\dR_*\big(\mca{F}^j_{\reg/M}  \big)^{\otimes 2}\ar[r]\ar[d]&H^\dR_{*-d} \big( \mca{F}^j_{\reg/M}  \fbp{e_j}{e_j}  \mca{F}^j_{\reg/M} \big)\ar[d]& \\
&H_*\big(\mca{F}^j  \big)^{\otimes 2}\ar[r]&H_{*-d} \big(\mca{F}^j \fbp{e_j}{e_j} \mca{F}^j \big)& \\
H_*(\L^{E_j})^{\otimes 2}\ar[ru]^{\star}\ar[rrr]&&&H_{*-d} (\L^{E_j} \fbp{e}{e}  \L^{E_j}).\ar[lu]_{\star}
}
\end{equation}
The center square in (\ref{eq:bigsquare})
 is commutative by Lemma \ref{150726_1}. 
The commutativity of the other four squares in  (\ref{eq:bigsquare}) is easy to check from the definitions. 

Taking direct limits as $j \to \infty$, all maps in (\ref{eq:bigsquare}) pass to the limit. 
Moreover, the limits of the maps with $\star$ are isomorphisms. 
Therefore, 
the limit of the big square in (\ref{eq:bigsquare}) is 
commutative, and this completes the proof of Lemma \ref{150629_4}.

\textbf{Funding.} 
This work is supported by the Japan Society for the Promotion of Science (JSPS) 
Grant-in-Aid for Young Scientists (B)
[grant number 25800041]
and JSPS Postdoctoral Fellowship for Research Abroad. 

\textbf{Acknowledgements.} 
The author appreciates 
Kenji Fukaya and Kaoru Ono for their encouragement and comments on this project, 
Benjamin Ward for helpful communication,
and the referees for carefully reading the manuscript and giving various critical and helpful comments. 
The author also acknowledges 
the Simons Center for Geometry and Physics at Stony Brook University, where the revision of this paper took place, 
for the great working environment.

\end{document}